\documentclass{article}

\usepackage{graphicx} 
\usepackage[top = 0.9in, bottom = 0.9in, left =1in, right = 1in]{geometry}
\usepackage{geometry}
\usepackage[utf8]{inputenc}
\usepackage{hyperref}
\usepackage{listings}
\usepackage{multimedia} 
\usepackage{comment}
\usepackage[english]{babel}
\usepackage{amsfonts,amscd,amssymb,amsmath,mathrsfs}
\usepackage{wrapfig}
\usepackage{multirow}
\usepackage{verbatim,cite}
\usepackage{float}
\usepackage{cancel}
\usepackage{caption}
\usepackage{subcaption}
\usepackage{mathdots}
\usepackage{bbm}
\usepackage{bm}
\usepackage{tikz}
\usepackage{setspace}
\usepackage{xcolor}
\usepackage{amsthm}
\usepackage{cases}
\usepackage{tikz}
\usepackage{bbm}
\usepackage{authblk}
\usepackage[acronym]{glossaries}
\usepackage{datetime2}
\usepackage{datetime2-calc}

\let\vec = \mathbf

\DTMnewdatestyle{monthandyear}{%
}
\DTMsetdatestyle{monthandyear}

\hypersetup{
  colorlinks   = true, 
  urlcolor     = blue, 
  linkcolor    = blue, 
  citecolor   = blue 
}

\usetikzlibrary{shapes.geometric, arrows.meta, positioning, fit, shadows}

\definecolor{derivationblue}{RGB}{52, 152, 219}
\definecolor{computationred}{RGB}{231, 76, 60}
\definecolor{forwardgreen}{RGB}{39, 174, 96}
\definecolor{inverseorange}{RGB}{243, 156, 18}
\definecolor{lightgray}{RGB}{248, 249, 250}
\definecolor{darkgray}{RGB}{44, 62, 80}

\tikzset{
    derivation box/.style={
        rectangle, rounded corners=8pt, 
        fill=lightgray, draw=derivationblue, line width=2pt,
        text width=5.5cm, align=left, 
        minimum height=1.2cm, inner sep=12pt,
        drop shadow
    },
    computation box/.style={
        rectangle, rounded corners=8pt, 
        fill=lightgray, draw=computationred, line width=2pt,
        text width=5.5cm, align=left, 
        minimum height=1.2cm, inner sep=12pt,
        drop shadow
    },
    derivation header/.style={
        rectangle, rounded corners=8pt, 
        fill=derivationblue, text=white,
        text width=5.5cm, align=center, 
        minimum height=0.8cm, inner sep=8pt,
        font=\Large\bfseries
    },
    computation header/.style={
        rectangle, rounded corners=8pt, 
        fill=computationred, text=white,
        text width=5.5cm, align=center, 
        minimum height=0.8cm, inner sep=8pt,
        font=\Large\bfseries
    },
    step number/.style={
        circle, fill=#1, text=white, 
        minimum size=0.6cm, font=\bfseries\small
    },
    forward label/.style={
        rectangle, rounded corners=12pt, 
        fill=forwardgreen!20, draw=forwardgreen, line width=1.5pt,
        text=forwardgreen, font=\tiny\bfseries,
        inner sep=3pt, minimum height=0.4cm
    },
    inverse label/.style={
        rectangle, rounded corners=12pt, 
        fill=inverseorange!20, draw=inverseorange, line width=1.5pt,
        text=inverseorange, font=\tiny\bfseries,
        inner sep=3pt, minimum height=0.4cm
    },
    connection arrow/.style={
        ->, >=Stealth, line width=3pt, 
        color=darkgray!60
    }
}

\def\Xint#1{\mathchoice
   {\XXint\displaystyle\textstyle{#1}}%
   {\XXint\textstyle\scriptstyle{#1}}%
   {\XXint\scriptstyle\scriptscriptstyle{#1}}%
   {\XXint\scriptscriptstyle\scriptscriptstyle{#1}}%
   \!\int}
\def\XXint#1#2#3{{\setbox0=\hbox{$#1{#2#3}{\int}$}
     \vcenter{\hbox{$#2#3$}}\kern-.5\wd0}}

\def\dashint{\Xint-}

\newtheorem{theorem}{Theorem}
\newtheorem{corollary}{Corollary}
\newtheorem*{remark}{Remark}
\newtheorem{lemma}{Lemma}
\newtheorem{proposition}{Proposition}
\newtheorem{definition}{Definition}
\newtheorem{rhp}{RHP}

\title{A Riemann--Hilbert approach to the computation of transform pairs}
\author[1]{Kaitlynn Lilly}
\author[2]{Thomas Trogdon}
\affil[1,2]{Department of Applied Mathematics, University of Washington}
\affil[1]{\texttt{klilly@uw.edu}}
\affil[2]{\texttt{trogdon@uw.edu}}
\date{\DTMtoday}

\newacronym{rhp}{RHP}{Riemann--Hilbert problem}
\newacronym{pde}{PDE}{partial differential equation}
\newacronym{ode}{ODE}{ordinary differential equation}
\newacronym{fft}{FFT}{fast Fourier transform}
\newacronym{gmres}{GMRES}{generalized minimal residual}
\newacronym{uclm}{UCLM}{ultraspherically-collocated Levin's method}
\newacronym{akns}{AKNS}{Ablowitz-Kaup-Newell-Segur}
\begin{document}

\maketitle

\begin{abstract}
    We develop a unified methodology that integrates spectral theory, Riemann--Hilbert problems, and inverse scattering theory for the construction and numerical evaluation of transform pairs associated with linear variable-coefficient partial differential equations. The approach combines analytical formulae with numerical methods for ordinary differential equations and Riemann--Hilbert problems, yielding a hybrid analytical-numerical strategy for working with these transforms. Results are presented for transforms arising in the Dirac equation, demonstrating accurate computations, even in the presence of discontinuous coefficients.
\end{abstract}

\section{Introduction}

The study of \glspl{pde} with variable coefficients presents significant challenges in mathematical analysis and computation. While constant-coefficient \glspl{pde} can often be solved explicitly using Fourier or Laplace transforms, variable-coefficient equations arise more naturally in real-world models and are considerably more difficult to analyze. Foundational approaches such as Green’s function methods and eigenfunction expansions~\cite{stakgold_boundary_1967}, as well as the inverse scattering method\footnote{The inverse scattering method was originally developed for integrable nonlinear \glspl{pde} rather than for linear equations with variable coefficients. However, it has provided tools for a range of subsequent approaches to linear problems.}~\cite{faddeev_hamiltonian_1987}, have provided important theoretical tools. More recently, the unified transform method~\cite{fokas_unified_2008} has been extended to linear \glspl{pde} with variable coefficients \cite{deconinck_variable-coefficient_2025,fokas_boundary-value_2004}. Because these formulations are naturally expressed in terms of \glspl{rhp}, recent advances in numerical Riemann--Hilbert techniques~\cite{trogdon_scattering_2021} have substantially expanded their computational potential. Despite these developments in the unified transform method and numerical Riemann--Hilbert techniques, explicitly representable transform pairs for variable-coefficient problems remain scarce, and effective numerical implementations are largely undeveloped. Notably, the transform pairs introduced by Fokas~\cite{fokas_boundary-value_2004} are defined through \glspl{rhp}, but are formulated from an analytical, rather than computational, perspective.

In this paper, we combine a Riemann--Hilbert–based transform construction following the framework of Fokas with advancements in computational methods. Our approach is demonstrated on the Dirac operator, a first-order system that plays a central role in both mathematical physics and integrable systems. 
The Dirac equation is given by
\begin{equation}\label{eq:Dirac}
   i \frac{\mathrm{d}}{\mathrm{d}t}\mathbf{n}(x,t)=i\sigma_3\frac{\mathrm{d}}{\mathrm{d}x}\mathbf{n}(x,t)-i\sigma_3\mathbf{Q}(x)\mathbf{n}(x,t), \quad (x,t)\in\mathbb{R}\times (0,\infty),\quad \mathbf{n}(x,t)\in\mathbb{C}^{2\times 1},
\end{equation}
where
\begin{equation}\label{eq:Potential}
    \quad \mathbf{Q}(x)=\begin{bmatrix}
        0&q(x)\\ \tau\bar{q}(x)&0
    \end{bmatrix} \quad \text{and}\quad\sigma_3=\begin{bmatrix}
        1&0\\
        0&-1
    \end{bmatrix},
\end{equation}
where $\tau=\pm 1$, $\sigma_3$ is the third Pauli matrix, and $\bar{q}$ denotes the complex conjugate of $q$. Initially, for simplicity, we assume that $q(x)$ is a Schwartz-class function, \emph{i.e.}, that $q\in\mathcal{S}(\mathbb{R},\mathbb{C})$, where
\begin{align*}
    \mathcal{S}(\mathbb{R},\mathbb{C}):=\left\{q\in C^\infty(\mathbb{R},\mathbb{C})~\Big|~\forall\alpha,\beta\in\mathbb{N},\sup_{x\in\mathbb{R}}\left |x^\alpha q^{(\beta)}(x)\right|<\infty\right\},
\end{align*}
with 
\begin{equation*}
    C^\infty(\mathbb{R},\mathbb{C})=\{q:\mathbb{R}\rightarrow\mathbb{C}~|~q^{(n)}\text{ exists and is continuous for all }n\in\mathbb{N}\}.
\end{equation*}
This ensures sufficient smoothness and rapid decay at infinity for the potential. Individual results below can be established for less regular potentials, and we highlight these on occasion. Applying separation of variables allows us to identify the spatial differential operator associated with~\eqref{eq:Dirac}. We use the ansatz
\[
\mathbf{n}(x,t) = e^{-i\lambda \mathbb{I} t} \mathbf{v}(x),\quad \mathbf{v}(x) = \mathbf{v}(x;\lambda)\in\mathbb{C}^{2\times 1},
\]
where $\mathbb{I}$ is the $2\times 2$ identity matrix and $\lambda\in\mathbb{C}$ is the spectral parameter. Substituting this into~\eqref{eq:Dirac} leads to the spatial \gls{ode}
\begin{equation}\label{eq:DiracIntroSpace}
    \frac{\mathrm{d}}{\mathrm{d}x}\mathbf{v}(x;\lambda)=\left(-i\lambda\sigma_3+\mathbf{Q}(x)\right)\mathbf{v}(x;\lambda).
\end{equation}
We pause to note that, after slight rearrangement, this represents the spectral problem for the Dirac operator which, in turn, admits a rich scattering theory as it is a special case of the \gls{akns} scattering problem~\cite{ablowitz_inverse_1974}. See~\cite{deift_inverse_1979,ablowitz_solitons_1991,ablowitz_complex_2003,beals_direct_1988,drazin_solitons_1989,zakharov_exact_1970} for theoretical developments related to the scattering transforms for this system. The associated scattering transform maps $\mathbf{Q}$ to scattering coefficients. Scattering coefficients are defined on the continuous spectrum of the associated operator, while discrete eigenvalues and associated norming constants together encode contributions related to the discrete spectrum. The inverse to this scattering transform, in particular, expresses generalized eigenfunctions as functions of the spectral parameter through a nonlinear map.

Although the connection between generalized Fourier transform formulations arising from the spectral theorem underlies works such as~\cite{wilkening_spectral_2015}, we establish here, using explicitly solvable \glspl{rhp}, what is to the best of our knowledge, the most explicit relationship between spectral theory and generalized Fourier transforms. The rich scattering theory of the Dirac operator then allows, via the numerical solution of \glspl{rhp}, the effective implementation of the associated transform. In short, and as noted above, the scattering theory gives a convenient representation of the generalized eigenfunctions as a function of the spectral parameter. When computed, these generalized eigenfunctions allow the inversion integrals to be computed effectively.

In recent years, numerical techniques for computing with \glspl{rhp} have improved, as demonstrated in~\cite{trogdon_scattering_2021} on the \gls{akns} scattering problem. The \glspl{rhp} we consider are equivalent to singular integral equations and the (infinite-dimensional) \gls{gmres} algorithm~\cite{GMRES} can be applied as in~\cite{trogdon_application_2015} using the so-called oscillatory Cauchy operator applied to a rational basis, without any need for contour deformations. When implemented appropriately, the method allows the computation of these generalized eigenfunctions, as functions of the spectral parameter, with increasing efficiency as oscillations increase. 

The present work contributes on two fronts. Theoretically, it provides a direct and concrete bridge between scattering data and spectral projections by establishing that the recovery formulas for the inverse transform are those implied by Stone’s formula in spectral theory. Computationally, the combination of oscillatory rational basis functions, Riemann–Hilbert formulations, and \gls{ode} solvers yields a framework for computing with a class of generalized Fourier transforms that is robust, accurate, and, we believe, widely applicable. In particular, the method remains accurate and effective even in the presence of discontinuous potential functions, discontinuous forcing terms, and high oscillations, suggesting that it captures the right operator-theoretic structure for handling irregular data. 
By making the operator-theoretic framework explicit and implementable, we hope to lay the groundwork for efficient numerical solvers for classes of variable-coefficient \glspl{pde}.

The paper is organized as follows. In Section~\ref{sec:Spectral}, we introduce key ideas from spectral theory that will be instrumental in demonstrating the equivalence of our recovery formulas and Stone's formula. In Section~\ref{sec:Fourier}, we illustrate the derivation and computation of the classical Fourier transform on $\mathbb{R}$, a familiar example, in order to demonstrate the general methodology. In Section~\ref{sec:GeneralizedSteps} 
we then detail the construction of the generalized forward and inverse transforms associated to the Dirac equation, including the treatment of poles corresponding to the discrete spectrum. Appendix~\ref{sec:RationalBasis} provides further details on the oscillatory rational basis and the evaluation of Cauchy integral operators, while Appendix~\ref{sec:Proof} contains proofs of some results stated in the main text.

\begin{remark}[Notational convention]
    We denote by $\bar{f}(x)$ the complex conjugate of a complex-valued function $f$. With the exception of the first and third Pauli matrices defined by
    \begin{equation*}
        \sigma_1 := \begin{bmatrix}
            0 & 1\\
            1 & 0
        \end{bmatrix}\quad\text{and}\quad\sigma_3 := \begin{bmatrix}
            1 & 0\\
            0&-1
        \end{bmatrix},
    \end{equation*}
    and the identity matrix $\mathbb{I}$, we use boldface capital letters to denote matrices and boldface lowercase letters to denote vectors. For two $n\times 1$ (column) vectors $\mathbf{u}$ and $\mathbf{v}$, we denote the $n\times 2$ matrix whose columns are $\mathbf{u}$ and $\mathbf{v}$ by $\begin{bmatrix}
        \mathbf{u} & \mathbf{v}
    \end{bmatrix}$. We use capital calligraphic letters to denote operators. Finally, we use $\mathbb{C}^\pm$ to denote the open upper and lower half-planes, respectively.
\end{remark}

\section{Key Aspects from Spectral Theory}\label{sec:Spectral}

    Because we are interested in transform pairs associated with differential equations, which are tightly connected to the spectral properties of the associated differential operator, we first introduce key concepts from the spectral theory of unbounded operators. A comprehensive reference is~\cite{reed_functional_1980}. Let $D(\mathcal{T})$ denote the domain of an operator $\mathcal{T}$.
    \begin{definition}[Adjoint]
        Let $\mathcal{H}$ be a Hilbert space with inner product $(\cdot,\cdot)$, and let $\mathcal{T}:D(\mathcal{T})\rightarrow\mathcal{H}$, $D(\mathcal{T})\subset\mathcal{H}$, be a densely defined linear operator.
        Let $D(\mathcal{T}^*)$ be the set of $\varphi\in\mathcal{H}$ for which there is an $\eta\in\mathcal{H}$ with 
        \begin{equation*}
            (\mathcal{T}\psi,\varphi)=(\psi,\eta)\quad\text{for all }\psi\in D(\mathcal{T}).
        \end{equation*}
        The adjoint of $\mathcal{T}$ relative to its domain $D(\mathcal{T})$ is the operator $\mathcal{T}^*:D(\mathcal{T}^*)\rightarrow\mathcal{H}$, $D(\mathcal{T}^*)\subset \mathcal{H}$, such that
        \begin{equation*}
            (\mathcal{T}\psi,\varphi)=(\psi,\mathcal{T}^*\varphi)\quad\text{for all }\psi\in D(\mathcal{T}),~\varphi\in D(\mathcal{T}^*).
        \end{equation*}
        Because $\mathcal{T}$ is densely defined, the adjoint $\mathcal{T}^*$ exists and is unique for the given domain $D(\mathcal{T})$.
    \end{definition}
    
    \begin{definition}[Resolvent]
        Let $\mathcal{T}$ be a closed operator on a Hilbert space $\mathcal{H}$. A complex number $\lambda$ is in the resolvent set, $\rho(\mathcal{T})$, if and only if $\mathcal{T}-\lambda $ is a bijection of $D(\mathcal{T})$ onto $\mathcal{H}$ with a bounded inverse. If $\lambda\in\rho(\mathcal{T})$, $(\mathcal{T}-\lambda)^{-1}$ is called the resolvent of $\mathcal{T}$ at $\lambda$. If $\lambda\notin\rho(\mathcal{T})$, then $\lambda$ is said to be in the spectrum $\sigma(\mathcal{T}):=\rho(\mathcal{T})^c$ of $\mathcal{T}$.
    \end{definition}

If $\mathcal{T}$ is self-adjoint (see definition~\ref{def:self-adjoint}), then its spectrum is contained in $\mathbb{R}$.

    \begin{definition}[Symmetric]
        A densely defined operator $\mathcal{T}$ on a Hilbert space is called symmetric (or Hermitian) if it is contained in its adjoint, $\mathcal{T}^*$, relative to $D(\mathcal{T})$, that is, if $D(\mathcal{T})\subset D(\mathcal{T}^*)$ and $\mathcal{T}\varphi=\mathcal{T}^*\varphi$ for all $\varphi\in D(\mathcal{T})$. Equivalently, $\mathcal{T}$ is symmetric if and only if 
        \begin{equation*}
            (\mathcal{T}\varphi,\psi)=(\varphi,\mathcal{T}\psi)\quad\forall~\varphi,\psi\in D(\mathcal{T}).
        \end{equation*}
    \end{definition}
    \begin{definition}[Self-adjoint]\label{def:self-adjoint}
        $\mathcal{T}$ is called self-adjoint if $\mathcal{T}=\mathcal{T}^*$, that is, if and only if $\mathcal{T}$ is symmetric and $D(\mathcal{T})=D(\mathcal{T}^*)$.
    \end{definition}
    The spectral theorem admits several formulations; here we use the projection-valued measure framework. Let $\mathcal{L}(\mathcal{H})$ denote the space of all bounded linear operators on a Hilbert space $\mathcal{H}$.
    \begin{definition}[Projection]
        If $P\in\mathcal{L}(\mathcal{H})$ and $P^2=P$, then $P$ is called a projection. If in addition $P=P^*$, then $P$ is called an orthogonal projection.
    \end{definition}
    \begin{definition}[Strong operator limit]
        Let $\{\mathcal{T}_n\}$ be a sequence of bounded operators on a Hilbert space $\mathcal{H}$. We say that $\mathcal{T}_n$ converges strongly to an operator $\mathcal{T}$, and write
        \begin{equation*}
            \mathcal{T}=s\text{-}\lim_{n\rightarrow\infty}\mathcal{T}_n,
        \end{equation*}
        if 
        \begin{equation*}
            \|\mathcal{T}_nx-\mathcal{T}x\|\rightarrow0\quad\text{for every }x\in\mathcal{H}.
        \end{equation*}
    \end{definition}
    For the next definition, let $\mathcal{B}(\mathbb{R})$ denote the Borel $\sigma$-algebra on $\mathbb{R}$ and let $\mathcal{I}:\mathcal{H}\rightarrow\mathcal{H}$ denote the identity operator on $\mathcal{H}$.
    \begin{definition}[Projection-valued measure]
        A mapping $P:\mathcal{B}(\mathbb{R})\rightarrow\mathcal{L}(\mathcal{H})$ that assigns to each Borel set $\Omega\subset\mathbb{R}$ an operator $P_\Omega$ is called a projection-valued measure if:
    \begin{itemize}
        \item[(a)] Each $P_\Omega$ is an orthogonal projection.
        \item[(b)] $P_\emptyset=0$, $P_{\mathbb{R}}=\mathcal{I}$.
        \item[(c)] If $\Omega=\bigcup_{n=1}^\infty\Omega_n$ with $\Omega_n\cap\Omega_m=\emptyset$ if $n\neq m$, then $P_\Omega=s$-$\lim_{N\rightarrow\infty}\sum_{n=1}^NP_{\Omega_n}$.
        \item[(d)] $P_{\Omega_1}P_{\Omega_2}=P_{\Omega_1\cap\Omega_2}$. 
    \end{itemize}
    \end{definition}
For $\varphi \in \mathcal{H}$, the map $\Omega \mapsto (\varphi, P_\Omega \varphi)$ defines a finite Borel measure on $\mathbb{R}$, which we denote formally by $\mathrm{d}(\varphi, P_\lambda \varphi)$, so that for any Borel set $\Omega \subset \mathbb{R}$,
\[
(\varphi, P_\Omega \varphi) = \int_\Omega \mathrm{d}(\varphi, P_\lambda \varphi).
\]
The, in general, complex measure $\mathrm{d}(\varphi, P_\lambda \psi)$ for $\varphi, \psi \in \mathcal{H}$ is defined by polarization from $\mathrm{d}(\varphi, P_\lambda \varphi)$.
Let $g:\mathbb{R} \to \mathbb{C}$ be a bounded measurable function. Then the operator $g(\mathcal{A})$ is bounded on all of $\mathcal{H}$ and defined via its sesquilinear form by
\[
(\varphi, g(\mathcal{A}) \psi) = \int_{-\infty}^\infty g(\lambda) \, \mathrm{d}(\varphi, P_\lambda \psi), \quad \varphi, \psi \in \mathcal{H}.
\]
If $g$ is unbounded, the operator $g(\mathcal{A})$ may only be defined on a dense subset of $\mathcal{H}$. Its domain is\footnote{The square in $|g(\lambda)|^2$ appears because
$\| g(\mathcal{A}) \varphi \|^2 = (g(\mathcal{A}) \varphi, g(\mathcal{A}) \varphi) = \int_{-\infty}^\infty |g(\lambda)|^2 \, \mathrm{d}(\varphi, P_\lambda \varphi)$,
which follows from the spectral theorem.}
\[
D(g(\mathcal{A})) := \left\{ \varphi \in \mathcal{H} \ \middle|\ \int_{-\infty}^\infty |g(\lambda)|^2 \, \mathrm{d}(\varphi, P_\lambda \varphi) < \infty \right\}.
\]
Requiring it to be finite ensures that $g(\mathcal{A})\varphi\in\mathcal{H}$. The set $D(g(\mathcal{A}))$ is dense in $\mathcal{H}$: for any $\varphi \in \mathcal{H}$, the truncated vectors $\varphi_n = P_{\{|g| \le n\}} \varphi$ belong to $D(g(\mathcal{A}))$ and satisfy $\varphi_n \to \varphi$ in norm as $n\rightarrow\infty$. On $D(g(\mathcal{A}))$, $g(\mathcal{A})$ is defined via 
\[
(\varphi, g(\mathcal{A}) \psi) = \int_{-\infty}^\infty g(\lambda) \, \mathrm{d}(\varphi, P_\lambda \psi), \quad \varphi,\psi \in D(g(\mathcal{A})).
\]
Symbolically, we write
\[
g(\mathcal{A}) = \int g(\lambda) \, \mathrm{d}P_\lambda.
\]
In particular, for $\varphi, \psi \in D(\mathcal{A})$,
\[
(\varphi, \mathcal{A} \psi) = \int_{-\infty}^\infty \lambda \, \mathrm{d}(\varphi, P_\lambda \psi).
\]
If $g$ is real-valued, $g(\mathcal{A})$ is self-adjoint on $D(g(\mathcal{A}))$. 
We have the following.
    \begin{theorem}[Spectral Theorem]\label{thm:SpectralThm}
        There is a one-to-one correspondence between self-adjoint operators $\mathcal{A}$ and projection-valued measures $\{P_\Omega\}$ on $\mathcal{H}$. The correspondence is given by 
        \begin{equation*}
            \mathcal{A}=\int_{-\infty}^\infty\lambda\mathrm{d}P_\lambda.
        \end{equation*}
    \end{theorem}
    The spectral theorem gives us a way to define generalized Fourier transform pairs associated with a given self-adjoint differential operator via its spectral decomposition. A useful consequence of Theorem~\ref{thm:SpectralThm} is a formula relating the resolvent and projection-valued measure of any self-adjoint operator.

   \begin{theorem}[Stone's Formula]\label{thm:Stone}
       Let $\mathcal{A}$ be an unbounded\footnote{As noted in \cite{reed_functional_1980}, although Stone’s formula is stated for bounded operators (Theorem VII.13), the same proof applies to unbounded operators. We therefore state it in the unbounded case.} self-adjoint operator. Then for any $f\in\mathcal{H}$,
       \begin{equation}\label{eq:Stone}
           \lim_{\epsilon\rightarrow0^+}\left(\frac{1}{2\pi i}\int_{-\infty}^{\infty}\left[(\mathcal{A}-\lambda-i\epsilon)^{-1}-(\mathcal{A}-\lambda+i\epsilon)^{-1}\right]f\mathrm{d}\lambda\right) = f.
       \end{equation}
   \end{theorem}
This relation reconstructs $f$ via the spectral decomposition encoded in the jump of the resolvent across the real axis and serves as the a priori theoretical justification for our derivation of transform pairs.

\section{Fourier Transform as an Example}\label{sec:Fourier}

We illustrate the derivation and computation of forward and inverse transforms using the Fourier transform as a model case. Following~\cite[Section 7.4.2]{ablowitz_complex_2003}, consider the following ODE
\begin{equation}\label{eq:Fourier}
    \frac{\mathrm{d}}{\mathrm{d}x}\mu(x;\lambda)-i\lambda\mu(x;\lambda)=f(x),
\end{equation}
where $\lambda\in\mathbb{C}$ is the spectral variable. We assume that, for convenience, $f\in\mathcal{S}(\mathbb{R},\mathbb{C})$. All solutions to~\eqref{eq:Fourier} are of the form
\begin{equation}\label{eq:FourierSolution}
    \mu(x;\lambda)= \mu(a;\lambda) e^{i\lambda (x-a)}+\int_{a}^x e^{i\lambda(x-s)}f(s)\mathrm{d}s,
\end{equation}
where $a \in \mathbb R$. Take $\mu(a;\lambda)=0$ and $a=\mp\infty$ and define the following solutions 
\begin{equation}\label{eq:mu}
    \mu_\pm(x;\lambda)=\int_{\mp\infty}^xe^{i\lambda(x-s)}f(s)\mathrm{d}s.
\end{equation}
Observe that $\mu_+$ is analytic in the upper half-plane and $\mu_-$ is analytic in the lower half-plane.
Integrating $\mu_\pm(x;\lambda)$ by parts results in
\begin{equation*}
    \mu_\pm(x;\lambda)=-\frac{1}{i\lambda}f(x)+\frac{1}{i\lambda}\int_{\mp\infty}^xe^{i\lambda(x-s)}f'(s)\mathrm{d}s = \mathcal{O}\left(\frac{1}{\lambda}\right),\quad |\lambda|\rightarrow\infty,\quad\pm\text{Im}(\lambda)\geq0.
\end{equation*}
Using~\eqref{eq:Fourier}, 
\begin{equation*}
    \frac{\mathrm{d}}{\mathrm{d}x}\mu_\pm(x;\lambda)=\int_{\mp\infty}^xe^{i\lambda(x-s)}f'(s)\mathrm{d}s=\mathcal{O}\left(\frac{1}{\lambda}\right),\quad |\lambda|\rightarrow\infty,\quad\pm\text{Im}(\lambda)\geq0.
\end{equation*}
Define the sectionally analytic function $m:\mathbb{C}\setminus\mathbb{R}\rightarrow\mathbb{C}$
\begin{equation*}
    m(x;\lambda)=\begin{cases}
        \mu_+(x;\lambda),\quad\text{Im}(\lambda)>0,\\
        \mu_-(x;\lambda),\quad\text{Im}(\lambda)<0.
    \end{cases}
\end{equation*}
This leads to the following \gls{rhp}.
\begin{rhp}\label{rhp:1}
    Find $m(x;\diamond):\mathbb{C}\setminus\mathbb{R}\rightarrow\mathbb{C}$ such that\footnote{In this work, we look for solutions of \glspl{rhp} to be bounded, with continuous boundary values.}
    \begin{align*}
        m^+(x;\lambda)-m^-(x;\lambda)&=\left(\int_{-\infty}^\infty e^{-i\lambda s}f(s)\mathrm{d}s\right)e^{i\lambda x}=:\hat{f}(\lambda)e^{i\lambda x},\quad \lambda\in\mathbb{R},\\
        m(x;\lambda)&=\mathcal{O}\left(\frac{1}{\lambda}\right) \text{ as }|\lambda|\rightarrow\infty,\quad \lambda\in\mathbb{C}\setminus\mathbb{R},
    \end{align*}
    where
    \begin{equation}\label{eq:sup_pm}
        m^\pm(x;\lambda)=\lim_{\epsilon\rightarrow 0^+}m(x;\lambda\pm i\epsilon).
    \end{equation}
\end{rhp}
Throughout this work, a $\pm$ superscript denotes the boundary value defined in~\eqref{eq:sup_pm}. Observe that $\hat{f}(\lambda)$ in the jump condition of RHP~\ref{rhp:1} is exactly the Fourier transform. 
We now solve RHP~\ref{rhp:1} and derive a recovery formula for $f(x)$ to obtain the inverse Fourier transform. The solutions to this type of \gls{rhp} are given by the Cauchy integral using the Plemelj Lemma~\cite[Chapter 2, Section 17]{muskhelishvili_singular_2013},
\begin{equation}\label{eq:m}
    m(x;\lambda)=\frac{1}{2\pi i}\int_{-\infty}^\infty\frac{e^{i\lambda'x}\hat{f}(\lambda')}{\lambda'-\lambda}\mathrm{d}\lambda'.
\end{equation}
Taking a limit of~\eqref{eq:Fourier} results in
\begin{equation*}
    \lim_{|\lambda|\rightarrow\infty}\left(\frac{\mathrm{d}}{\mathrm{d}x}m(x;\lambda)-i\lambda m(x;\lambda)=f(x)\right)\Rightarrow f(x)=\lim_{|\lambda|\rightarrow\infty}-i\lambda m(x;\lambda).
\end{equation*}
The recovery formula is expressed in terms of the solution to RHP~\ref{rhp:1}, an inhomogeneous \gls{rhp}. Using our expression for $m(x;\lambda)$ results in
\begin{align*}
    f(x)=\lim_{|\lambda|\rightarrow\infty}\frac{-1}{2\pi}\int_{-\infty}^\infty\frac{\lambda}{\lambda'-\lambda}e^{i\lambda'x}\hat{f}(\lambda')\mathrm{d}\lambda'.
\end{align*}
Suppose we take the limit along the imaginary axis. This allows us to straightforwardly apply the dominated convergence theorem, resulting in
\begin{equation}\label{eq:invFourier}
    f(x)=\frac{1}{2\pi}\int_{-\infty}^\infty e^{i\lambda x}\hat{f}(\lambda)\mathrm{d}\lambda.
\end{equation}

We now show that this derivation is exactly equivalent to the recovery of $f$ via Theorem~\ref{thm:Stone}. Rewriting~\eqref{eq:Fourier} gives
\begin{align*}
    \left(-i\frac{\mathrm{d}}{\mathrm{d}x}-\lambda\right)\mu(x;\lambda)=-if(x),\quad \lambda\in\mathbb{C}\setminus\mathbb{R}.
\end{align*}
Let $\mathcal{A} = -i\frac{\mathrm{d}}{\mathrm{d}x}$ with domain $H^1(\mathbb{R},\mathbb{C})\subset L^2(\mathbb{R},\mathbb{C})$. Then $\mathcal{A}$ is self-adjoint, so its resolvent $(\mathcal{A}-\lambda)^{-1}$ exists for $\lambda\in\mathbb{C}\setminus\mathbb{R}$. For $\text{Im}(\lambda)>0$, the function $\mu_+(x;\lambda)$ (and hence $m(x;\lambda)$) is given by~\eqref{eq:mu}. Writing $\lambda=\xi+i\eta$ with $\eta>0$ gives $|e^{i\lambda(x-s)}|=e^{-\eta(x-s)}$, and hence
\begin{equation*}
    |m(x;\lambda)|\leq(K_\eta *|f|)(x),\quad K_\eta(t)=e^{-\eta t}\mathbbm{1}_{t\geq0},
\end{equation*}
where $\mathbbm{1}$ is the indicator function. Since $K_\eta\in L^1(\mathbb{R},\mathbb{C})$ with $\|K_\eta\|_{L^1}=1/\eta$, and $f\in\mathcal{S}(\mathbb{R},\mathbb{C})\subset L^2(\mathbb{R},\mathbb{C})$, Young's inequality implies that
\begin{equation*}
    \|m(\diamond;\lambda)\|_{L^2(\mathbb{R},\mathbb{C})}\leq\|K_\eta\|_{L^1}\|f\|_{L^2(\mathbb{R},\mathbb{C})}=\frac{1}{\text{Im}(\lambda)}\|f\|_{L^2(\mathbb{R},\mathbb{C})}<\infty.
\end{equation*}
Thus, $m(\diamond;\lambda)\in L^2(\mathbb{R},\mathbb{C})$ for $\text{Im}(\lambda)>0$. The same argument applies for Im($\lambda)<0$ using $\mu_-(x;\lambda)$. Since $m$ satisfies $(\mathcal{A}-\lambda)m=-if$, and $(\mathcal{A}-\lambda)^{-1}$ exists as a bounded operator on $L^2(\mathbb{R},\mathbb{C})$ for $\lambda\in\mathbb{C}\setminus\mathbb{R}$, it follows by the uniqueness of $L^2$ solutions that
\begin{equation}\label{eq:OpSol}
    m(x;\lambda)=\left(\mathcal{A}-\lambda\right)^{-1}(-if(x)).
\end{equation}
Observe, using RHP~\ref{rhp:1}, that
\begin{align*}
    e^{i\lambda x}\hat{f}(\lambda)&=m^+(x;\lambda)-m^-(x;\lambda)=\lim_{\epsilon\rightarrow 0^+}\left[m(x;\lambda+i\epsilon)-m(x;\lambda-i\epsilon)\right]\\
    &=\lim_{\epsilon\rightarrow 0^+}\left[(\mathcal{A}-\lambda-i\epsilon)^{-1}-(\mathcal{A}-\lambda+i\epsilon)^{-1}\right](-if)(x),
\end{align*}
where the last equality follows from replacing $m$ with~\eqref{eq:OpSol}. Applying this to~\eqref{eq:invFourier} results in
\begin{align}\label{eq:preDCT}
    f(x)&=\frac{1}{2\pi}\int_{-\infty}^\infty \lim_{\epsilon\rightarrow 0^+}\left[(\mathcal{A}-\lambda'-i\epsilon)^{-1}-(\mathcal{A}-\lambda'+i\epsilon)^{-1}\right](-if)(x)\mathrm{d}\lambda'.
\end{align}
With some integration by parts, the dominated convergence theorem justifies interchanging the limit and the integral, giving
\begin{align*}
        f(x)&=\lim_{\epsilon\rightarrow 0^+}\left(\frac{1}{2\pi i}\int_{-\infty}^\infty\left[(\mathcal{A}-\lambda'-i\epsilon)^{-1}-(\mathcal{A}-\lambda'+i\epsilon)^{-1}\right](f)(x)\mathrm{d}\lambda'\right).
\end{align*}
This coincides with the reconstruction formula of Theorem~\ref{thm:Stone}, establishing the equivalence with Stone’s formula. 

\subsection{Computation}

We briefly review ideas from Levin’s method~\cite{levin_analysis_1997} as they will be fundamental to our computations. Consider an oscillatory integral
\begin{equation*}
    I(\lambda)=\int_a^bf(x)e^{i\lambda g(x)}\mathrm{d}x,\quad g:[a,b]\rightarrow\mathbb{R}.
\end{equation*}
Define $u(x)=\int^xf(t)e^{i\lambda g(t)}\mathrm{d}t$ and write $u(x)=v(x)e^{i\lambda g(x)}$. Substituting into $u'(x)=f(x)e^{i\lambda g(x)}$ yields the first-order ODE
\begin{equation}\label{eq:Levin}
    \frac{\mathrm{d}v(x)}{\mathrm{d}x}+i\lambda g'(x)v(x)=f(x).
\end{equation}
Factoring out $e^{i\lambda g(x)}$ isolates the oscillatory component, ideally leaving $v(x)$ as a slowly varying function. Levin's collocation strategy selects a slowly varying, non-oscillatory solution $v(x)$ by representing it in a slowly varying basis, without imposing boundary conditions. Once $v(x)$ is approximated, the integral is recovered via $I(\lambda)=v(b)e^{i\lambda g(b)}-v(a)e^{i\lambda g(a)}=u(b)-u(a)$.

Motivated by this mechanism, we express the Fourier transform in terms of solutions to~\eqref{eq:Fourier}. As in the previous section, where the Fourier transform arose from the jump $m^+(x;\lambda)-m^-(x;\lambda)$, the functions $\mu_+(x;\lambda)$ and $\mu_-(x;\lambda)$ furnish an analogous pair of solutions for the ODE~\eqref{eq:Fourier}. Their difference recovers the Fourier transform,
\begin{align*}
    \mu_+(x;\lambda)-\mu_-(x;\lambda)&=\int_{-\infty}^xe^{i\lambda(x-s)}f(s)\mathrm{d}s-\int_{\infty}^xe^{i\lambda(x-s)}f(s)\mathrm{d}s=\int_{-\infty}^\infty e^{i\lambda(x-s)}f(s)\mathrm{d}s=e^{i\lambda x}\hat{f}(\lambda).
\end{align*}
Thus,
\begin{equation}\label{eq:CompFourierTrans}
    \hat{f}(\lambda)=e^{-i\lambda x}\left(\mu_+(x;\lambda)-\mu_-(x;\lambda)\right).
\end{equation}
All \glspl{ode} here are solved using the ultraspherical rectangular collocation method of~\cite{trogdon_ultraspherical_2024}. In contrast to Levin's original approach, we explicitly impose decay of $\mu_\pm(x;\lambda)$ as $x\rightarrow\mp\infty$. In practice, the problem is truncated to a finite interval outside of which the solution is smaller than numerical precision, and the decay condition is imposed as a boundary condition in the ultraspherical collocation scheme. Note that the general solution of~\eqref{eq:Fourier} is~\eqref{eq:FourierSolution}. The decay condition at $-\infty$ (for $\mu_+$) and $+\infty$ (for $\mu_-$) eliminates the oscillatory homogeneous solution. Consequently, the numerical solution essentially coincides with the non-oscillatory particular solution selected in Levin's method. While Levin's method is only guaranteed to be accurate for $\lambda$ sufficiently large, our approach will be successful\footnote{See \cite{trogdon_ultraspherical_2024} for convergence results that imply uniform convergence with respect to $\lambda$.} for all valid values of $\lambda$. We refer to this as the \gls{uclm}. 
Figure~\ref{fig:FourierTransform} shows a plot of the Fourier transform of
\begin{equation}\label{eq:Fouriertest}
    f(x)=\begin{cases}0,\quad &x<0,\\\frac{1}{2},\quad &x=0,\\e^{-x^2},&x>0,\end{cases}
\end{equation}
computed using the \gls{uclm}. 

\begin{figure}[htbp]
    \centering
    \begin{subfigure}[b]{0.49\textwidth}
        \centering
        \includegraphics[width=\textwidth]{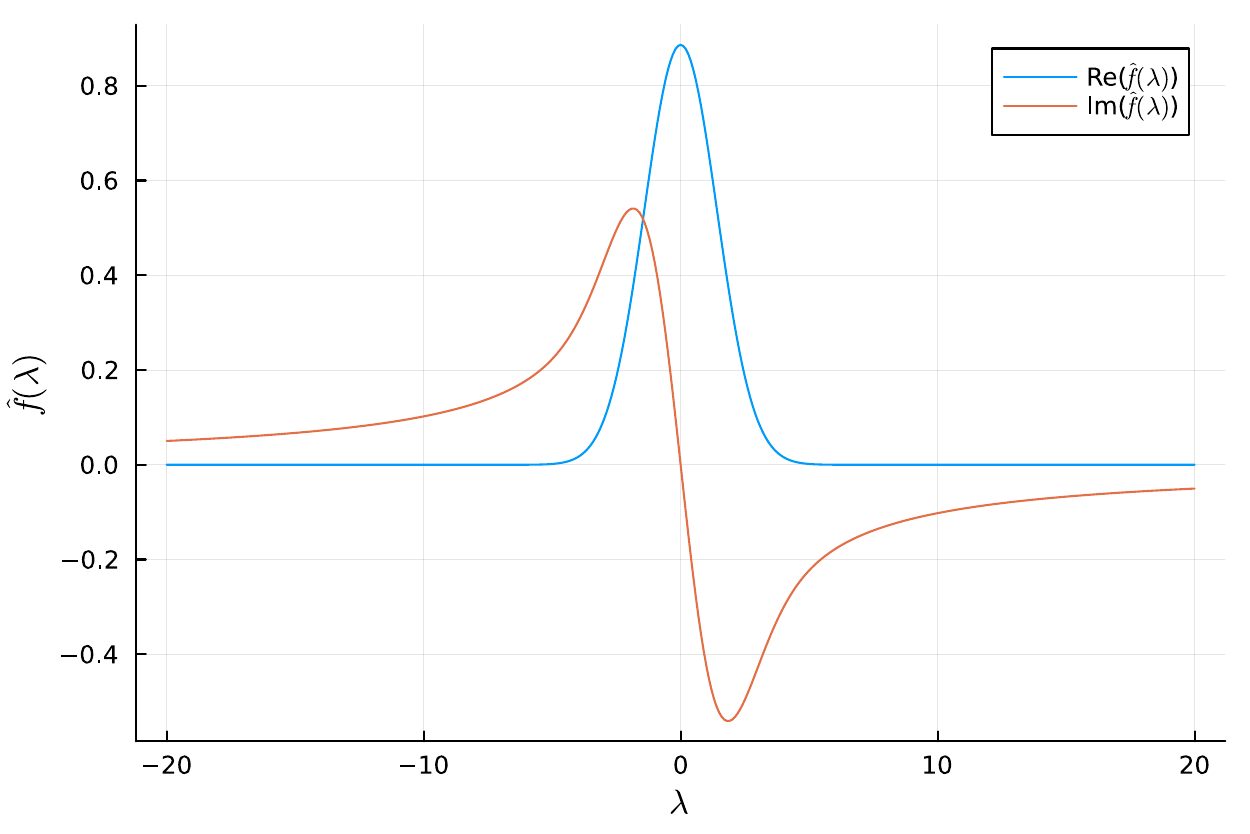}
        \caption{Fourier transform}
        \label{fig:FourierTransformPlot}
    \end{subfigure}
    \hfill
    \begin{subfigure}[b]{0.49\textwidth}
        \centering
        \includegraphics[width=\textwidth]{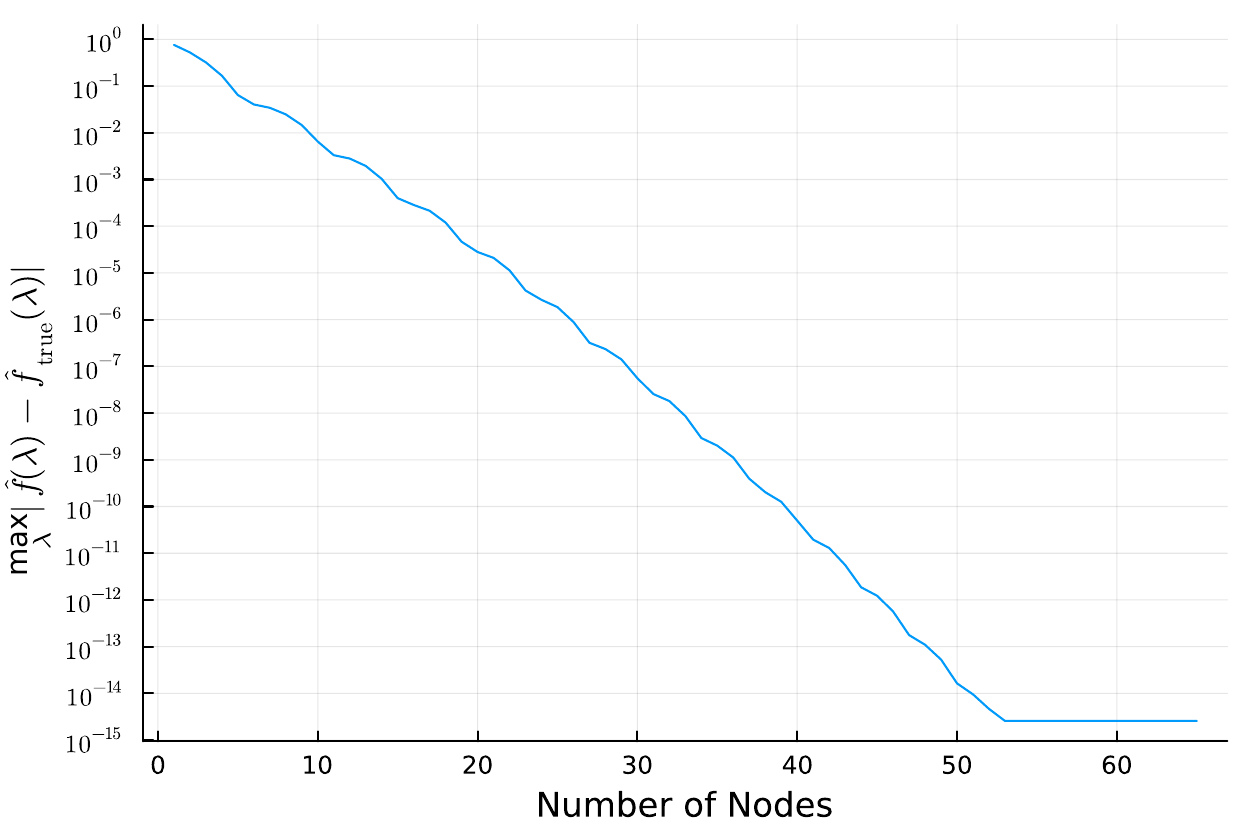}
        \caption{Error}
        \label{fig:FourierTransformError}
    \end{subfigure}
    \caption{(a) Real and imaginary parts of the computed Fourier transform of~\eqref{eq:Fouriertest} via \gls{uclm} on $[-20,20]$ using 55 nodes. (b) Max absolute error between the computed and exact Fourier transforms evaluated on a uniform grid in $\lambda \in [-20,20]$ as a function of the number of nodes used.}
    \label{fig:FourierTransform}
\end{figure}

We next compute the inverse transform. We represent the inverse transform using a rational basis expansion in $\lambda$, following ideas presented in~\cite{trogdon_scattering_2021}. While the \gls{uclm} could also be applied to the inverse transform, the rational basis approach allows explicit evaluation of the integrals term by term, which is convenient and efficient, and will extend more directly to other transforms. Define the following set of oscillatory rational basis functions\footnote{This choice of basis will be motivated in later sections on generalized transform pairs. An in-depth description of this basis and its properties is given in Appendix~\ref{sec:RationalBasis}.}
\begin{equation*}
    R_{j,\alpha}(\lambda)=e^{i\lambda\alpha}\left[\left(\frac{\lambda-i}{\lambda+i}\right)^j-1\right],\quad j\in\mathbb{Z},~\alpha\in\mathbb{R}.
\end{equation*}
Suppose
\begin{equation*}
    \hat{f}(\lambda)=\sum_jc_jR_{j,0}(\lambda).
\end{equation*}
Then the inverse Fourier transform becomes, 
\begin{equation}\label{eq:InvFourier}
    f(x)=\frac{1}{2\pi}\sum_jc_j\int_{-\infty}^\infty R_{j,x}(\lambda)\mathrm{d}\lambda=\begin{cases}
        -\sum\limits_j c_j|j|\quad &x=0,\\
        -2\sum\limits_{j ~:~ \mathrm{sign}(x)j>0}c_jL_{|j|-1}^{(1)}(2|x|)e^{-|x|}\quad&\mathrm{otherwise},
    \end{cases}
\end{equation}
where $L_j^{(1)}$ are the generalized Laguerre polynomials of degree $j$~\cite{olver_nist_2010} and an empty sum is taken to be zero. The second equality follows from the formula for integrals of the oscillatory rational basis functions derived in~\cite{trogdon_application_2015} and described in Appendix~\ref{subsec:Integration}. In practice, the series in~\eqref{eq:InvFourier} is truncated to a finite set of indices $|j|<J$, where $J$ is determined by the desired accuracy and decay of the coefficients $c_j$. We evaluate the Laguerre polynomials using their three-term recurrence. This gives a method to compute the inverse Fourier transform~\cite{weber_numerical_1980,trogdon_rational_2014}. A plot of the recovery of~\eqref{eq:Fouriertest} using~\eqref{eq:InvFourier} and its error is shown in Figure~\ref{fig:InvFourierTransform}. Figure~\ref{fig:InvFourierTransform} shows accurate recovery, with improved absolute error as $|x|$ increases. This method is accurate for all values of $x$, despite the discontinuity of $f(x)$ at $x=0$.

\begin{figure}[htbp]
    \centering
    \begin{subfigure}[b]{0.49\textwidth}
        \centering
        \includegraphics[width=\textwidth]{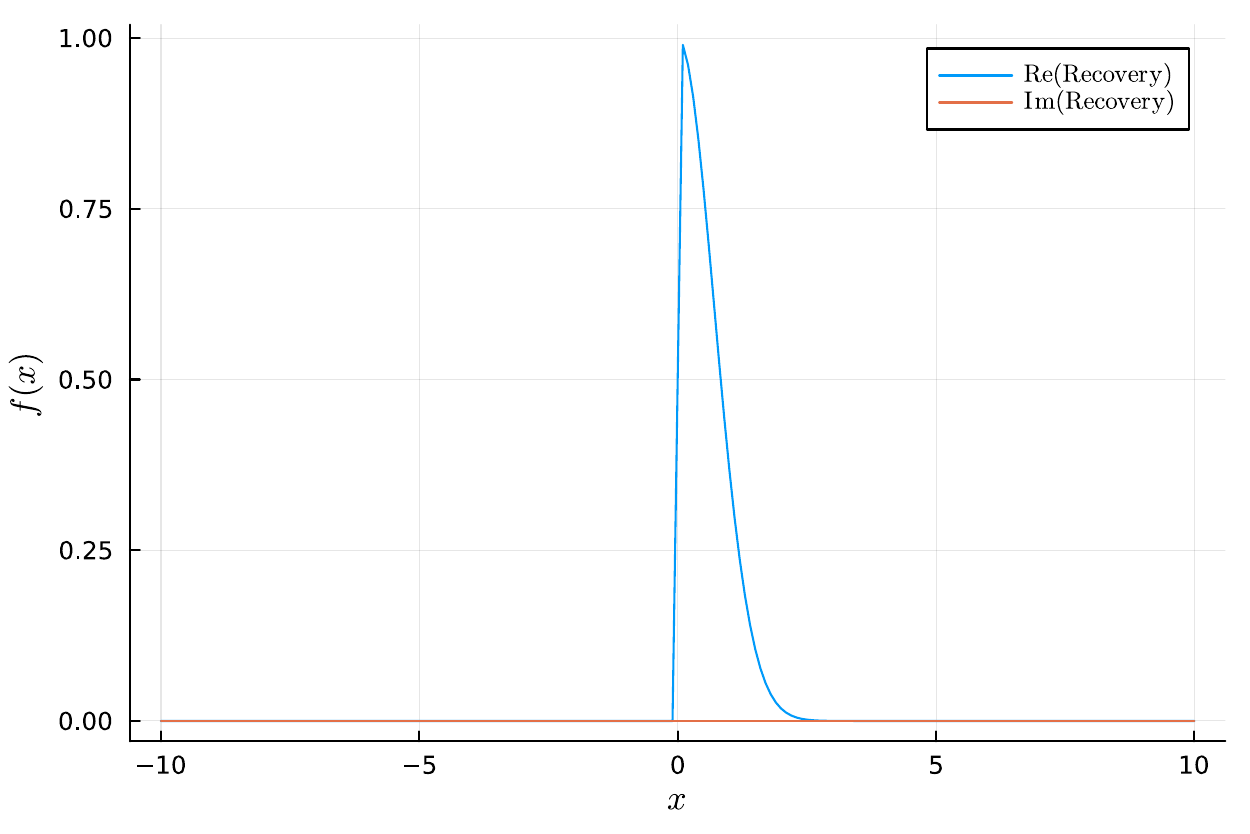}
        \caption{Inverse Fourier transform}
        \label{fig:InvFourierTransformPlot}
    \end{subfigure}
    \hfill
    \begin{subfigure}[b]{0.49\textwidth}
        \centering
        \includegraphics[width=\textwidth]{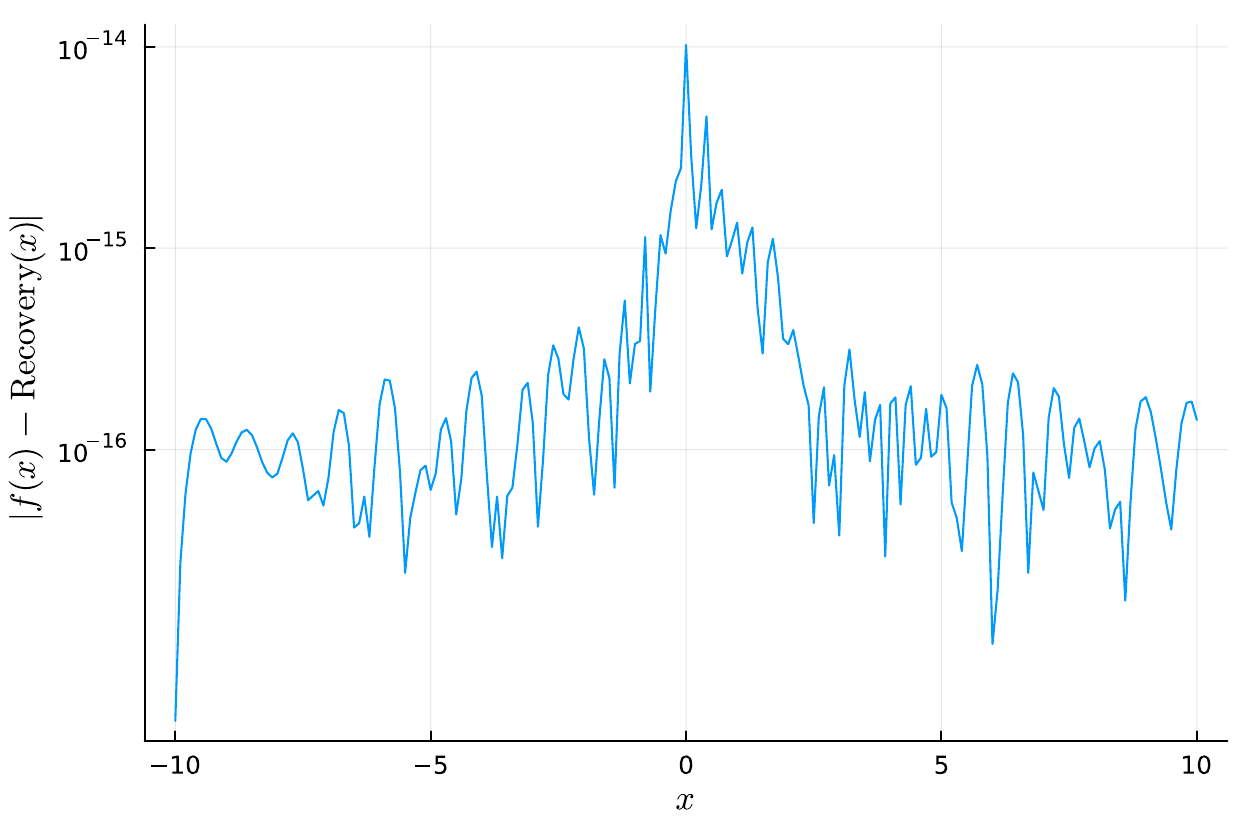}
        \caption{Error}
        \label{fig:InvFourierTransformError}
    \end{subfigure}
    \caption{(a) Real and imaginary parts of the computed inverse Fourier transform of~\eqref{eq:Fouriertest} via~\eqref{eq:InvFourier}. (b) Absolute error between the computed and true result.}
    \label{fig:InvFourierTransform}
\end{figure}

\section{Generalized Transform Pairs}\label{sec:GeneralizedSteps}

In Section~\ref{sec:Fourier}, the Fourier transform was realized through the resolvent of the underlying spatial operator. This formulation suggests a general framework: transform pairs can be constructed from the resolvent. We now apply this approach to the Dirac system by constructing the resolvent of~\eqref{eq:DiracIntroSpace}.
We consider
\begin{equation}\label{eq:DiracSpatial}
    \frac{\mathrm{d}}{\mathrm{d}x}\mathbf{v}(x;\lambda)+\left(i\lambda\sigma_3-\mathbf{Q}(x)\right)\mathbf{v}(x;\lambda)=\mathbf{f}(x),\quad \mathbf{f}(x)\in\mathbb{C}^{2\times 1},
\end{equation}
or, equivalently,
\begin{equation*}
    \left(i\sigma_3\frac{\mathrm{d}}{\mathrm{d}x}-i\sigma_3\mathbf{Q}(x)-\lambda\right)\mathbf{v}(x;\lambda)=i\sigma_3\mathbf{f}(x),
\end{equation*}
where we used the fact that $\sigma_3\sigma_3=\mathbb{I}$. We consider the operator $\mathcal{T}=i\sigma_3\frac{\mathrm{d}}{\mathrm{d}x}-i\sigma_3\mathbf{Q}(x)$. Then we can represent a solution to~\eqref{eq:DiracSpatial} in terms of the resolvent of $\mathcal{T}$ as
\begin{equation}\label{eq:DiracResolvent}
    \mathbf{v}(x;\lambda)=(\mathcal{T}-\lambda)^{-1}(i\sigma_3\mathbf{f})(x),\quad \lambda\in\rho(\mathcal{T}).
\end{equation}
Let $\mathcal{A}=i\sigma_3\frac{\mathrm{d}}{\mathrm{d}x}$ with domain $D(\mathcal{A})=H^1(\mathbb{R},\mathbb{C}^2)\subset L^2(\mathbb{R},\mathbb{C}^{2})$. We equip $L^2(\mathbb{R},\mathbb{C}^{2})$ with the standard inner product
\begin{equation*}
    \left( \mathbf{f},\mathbf{g}\right)=\int_{-\infty}^\infty\mathbf{f}(x)\cdot\bar{\mathbf{g}}(x)\,\mathrm{d}x,
\end{equation*}
where $\bar{\mathbf{g}}(x)$ denotes the component-wise complex conjugate and $\cdot$ denotes the standard dot product.
The operator $\mathcal{A}$ is self-adjoint because $\sigma_3$ is Hermitian and commutes with $\frac{\mathrm{d}}{\mathrm{d}x}$. When $\tau=1$ in~\eqref{eq:Potential}, let $\mathcal{B}=-i\sigma_3\mathbf{Q}(\diamond)$ with domain $D(\mathcal{B})=L^2(\mathbb{R},\mathbb{C}^{2})$. Since $q(x)$ is Schwartz-class, the matrix-valued function $\mathbf{Q}(x)$ is bounded, and therefore $\mathcal{B}$ defines a bounded operator on $L^2(\mathbb{R},\mathbb{C}^{2})$. Then, by the Kato-Rellich theorem~\cite{hislop_introduction_1996},
we conclude that $\mathcal{T}$ is self-adjoint.
We apply Theorems~\ref{thm:SpectralThm} and~\ref{thm:Stone} which imply that our \gls{rhp} approach for deriving transform pairs via the jump in the resolvent of $\mathcal{T}=i\sigma_3\frac{\mathrm{d}}{\mathrm{d}x}-i\sigma_3\mathbf{Q}(x)$ will succeed. 
\begin{remark}
    We do not have a self-adjoint operator in the case where $\tau=-1$ in~\eqref{eq:Potential}, 
    and we do not know a priori that our approach extends by appealing to the spectral theorem alone. This extension will be dealt with in Section~\ref{sec:DiracPoles}.
\end{remark}

To construct the resolvent, we begin by solving the free homoegenous system corresponding to $\mathbf{Q}(x)=\mathbf{0}$, $\mathbf{f}(x)=\mathbf{0}$,
\begin{equation*}
    \mathbf{H}'(x;\lambda)+i\lambda\sigma_3\mathbf{H}(x;\lambda)=0.
\end{equation*}
This problem has solution $\mathbf{H}(x;\lambda)=e^{-i\lambda\sigma_3x}\mathbf{C}(\lambda)$, where $\mathbf{C}(\lambda)\in\mathbb{C}^{2\times2}$ is a constant matrix.
We now solve the full homogeneous system corresponding to $\mathbf{f}(x)=\mathbf{0}$,
\begin{equation}\label{eq:f=0}
    \mathbf{M}'(x;\lambda)+i\lambda\sigma_3\mathbf{M}(x;\lambda)=\mathbf{Q}(x)\mathbf{M}(x;\lambda).
\end{equation}
To solve this, we will use variation of parameters. Consider the following proposition:
\begin{proposition}[Variation of parameters for systems]\label{prop:VarofParam}
    Consider
    \begin{equation}\label{eq:VarofParam1}
        \begin{cases}
            \mathbf{x}'(t)=\mathbf{A}(t)\mathbf{x}(t)+\mathbf{f}(t),\\
            \mathbf{x}(0)=\mathbf{x}_0,
        \end{cases}
    \end{equation}
    where $\mathbf{x}(t)\in\mathbb{C}^{n}$, $\mathbf{A}$ is an $n\times n$ matrix-valued function and $\mathbf{f}$ is an appropriately-sized vector-valued function. Now, let $\mathbf{x}_1(t),...,\mathbf{x}_n(t)$ be $n$ linearly independent solutions of the homogeneous problem ($\mathbf{f}=\mathbf{0}$), and form the matrix
    \begin{equation*}
        \mathbf{X}(t)=\begin{bmatrix}
            \mathbf{x}_1(t)&\hdots&\mathbf{x}_n(t)
        \end{bmatrix}.
    \end{equation*}
    For coefficient functions $\mathbf{A}$, $\mathbf{f}$, the solution to~\eqref{eq:VarofParam1} is given by
    \begin{equation*}
        \mathbf{x}(t)=\mathbf{X}(t)\mathbf{X}(0)^{-1}\mathbf{x}_0+\mathbf{X}(t)\int_0^t\mathbf{X}(s)^{-1}\mathbf{f}(s)\mathrm{d}s.
    \end{equation*}
\end{proposition}

In Section~\ref{sec:Fourier}, we were able to define solutions to the original \gls{ode} of interest that were analytic in the upper and lower half-planes, respectively. It was these analyticity properties that inspired us to formulate an \gls{rhp}. We would like to replicate that process for~\eqref{eq:DiracSpatial}. To do this, we must investigate the analyticity of solutions. In order to construct solutions to~\eqref{eq:DiracSpatial} that are analytic in the upper and lower half-planes, we first need to make an appropriate choice of $\mathbf{M}(x;\lambda)$. Inspired by the associated inverse scattering theory~\cite{ablowitz_solitons_1991}, we specifically choose to use solutions normalized as $\exp(-i\lambda\sigma_3x)$ as $x\rightarrow\pm\infty$. 
Applying Proposition~\ref{prop:VarofParam} to~\eqref{eq:f=0}, with the fundamental matrix $\mathbf{X}(x;\lambda) = \exp(-i\lambda\sigma_3 x)$, treating $\mathbf{Q}(x)\mathbf{M}(x;\lambda)$ as the inhomogeneity, results in
\begin{equation*}
    \mathbf{M}(x;\lambda)=e^{-i\lambda\sigma_3x}\mathbf{D}(x_0;\lambda)+\int_{x_0}^xe^{-i\lambda\sigma_3(x-s)}\mathbf{Q}(s)\mathbf{M}(s;\lambda)\mathrm{d}s.
\end{equation*}
We enforce the normalization $\mathbf{M}\sim e^{-i\lambda\sigma_3 x}$ as $x\rightarrow\pm\infty$, corresponding to taking $x_0\rightarrow\pm\infty$ and $\mathbf{D}(x_0;\lambda)\rightarrow\mathbb{I}$. This leads to Volterra integral equations on $(-\infty,x)$ and $(\infty,x)$, namely 
\begin{equation*}
    \begin{cases}
        \mathbf{M}_1(x;\lambda)=e^{-i\lambda\sigma_3 x}+\int_{-\infty}^xe^{-i\lambda\sigma_3(x-s)}\mathbf{Q}(s)\mathbf{M}_1(s;\lambda)\mathrm{d}s,\\
        \mathbf{M}_2(x;\lambda)=e^{-i\lambda\sigma_3x}+\int_\infty^xe^{-i\lambda\sigma_3(x-s)}\mathbf{Q}(s)\mathbf{M}_2(s;\lambda)\mathrm{d}s.
    \end{cases}
\end{equation*}
Standard results \cite[Chapter II, §9, Theorem II]{lovitt_linear_1924} guarantee the existence and uniqueness of continuous solutions for $\mathbf{Q}\in L^1(\mathbb{R},\mathbb{C}^{2\times 2})$, see also Lemma~\ref{lemma:Jost}. We now split $\mathbf{M}_1$ and $\mathbf{M}_2$ up into columns,
\begin{align*}
    \mathbf{m}_{1,-}(x;\lambda)&=\begin{bmatrix}
        e^{-i\lambda x}\\0
    \end{bmatrix}+\int_{-\infty}^xe^{-i\lambda\sigma_3(x-s)}\mathbf{Q}(s)\mathbf{m}_{1,-}(s;\lambda)\mathrm{d}s,\\
    \mathbf{m}_{1,+}(x;\lambda)&=\begin{bmatrix}
        0\\e^{i\lambda x}
    \end{bmatrix}+\int_{-\infty}^xe^{-i\lambda\sigma_3(x-s)}\mathbf{Q}(s)\mathbf{m}_{1,+}(s;\lambda)\mathrm{d}s,\\
    \mathbf{m}_{2,-}(x;\lambda)&=\begin{bmatrix}
     e^{-i\lambda x}\\0
    \end{bmatrix}+\int_{\infty}^xe^{-i\lambda\sigma_3(x-s)}\mathbf{Q}(s)\mathbf{m}_{2,-}(s;\lambda)\mathrm{d}s,\\
    \mathbf{m}_{2,+}(x;\lambda)&=\begin{bmatrix}
     0\\e^{i\lambda x}
    \end{bmatrix}+\int_{\infty}^xe^{-i\lambda\sigma_3(x-s)}\mathbf{Q}(s)\mathbf{m}_{2,+}(s;\lambda)\mathrm{d}s,
\end{align*}
and analyze the analyticity of each of our defined solutions with respect to the spectral parameter $\lambda$. Consider the following Jost-like functions, $\mathbf{r}_{\ell,\pm}(x;\lambda)=\mathbf{m}_{\ell,\pm}(x;\lambda)e^{\mp i\lambda x}$, $\ell\in\{1,2\}$. The proof of the following is included in Appendix~\ref{sec:Proof}, see \cite[Lemma 3.4]{trogdon_riemannhilbert_2016} for similar arguments.

\begin{lemma}\label{lemma:Jost}
    Suppose $\vec P \in L^1(\mathbb R, \mathbb C^{2 \times 2})$ and $\vec K(x,s;\lambda) \in \mathbb C^{2\times 2}$, $s \leq x$, is analytic for $\lambda \in \mathbb C^+$.  Suppose further that $\vec K(x,s;\lambda)$ extends to be continuous  for $ \lambda \in \overline{\mathbb C^+}$ satisfying
    \begin{align*}
        \|\vec K(x,s;\lambda) \| \leq C_0,\quad s \leq x,~ \lambda \in \overline{\mathbb C^+},
    \end{align*}
    and
    \begin{align*}
        \|\partial_\lambda \mathbf{K}(x,s;\lambda)\| \le C_\lambda, \quad s \leq x,~\lambda \in \mathbb C^+.
    \end{align*}
    
    \begin{enumerate}
    \item For $\mathbf{f} \in L^\infty(
    \mathbb{R},
    \mathbb C^2)$ there is a unique solution $\mathbf{r}(\diamond;\lambda) \in L^\infty(
    \mathbb{R},\mathbb C^2)$ of the Volterra integral equation
    \begin{equation*}
        \mathbf{r}(x;\lambda)=\mathbf{f}(x)+\int_{-\infty}^x\mathbf{K}(x,s;\lambda)\mathbf{P}(s)\mathbf{r}(s;\lambda)\mathrm{d}s.
    \end{equation*}
    \item With $M = \|\vec P\|_{L^1(
    \mathbb{R},\mathbb C^{2\times 2})}$, we have
    \begin{equation*}
    \|\mathbf{r}(\diamond;\lambda)\|_{L^\infty(
    \mathbb{R},\mathbb C^2)} \le e^{C_0 M}, \quad \lambda \in \mathbb C^+.
    \end{equation*}
    \item For each fixed $x \in \mathbb R$, $\mathbf{r}(\diamond;\lambda)$ is analytic as a function of $\lambda \in \mathbb C^+$ and is continuous on $\overline{\mathbb C^+}$.
    \end{enumerate}
\end{lemma}

We observe that Lemma~\ref{lemma:Jost} immediately applies to $\vec r_{1,-}$ and $\vec r_{2,+}$. It then applies to $\mathbf{r}_{2,-}$ and $\mathbf{r}_{1,+}$ after using $x \to -x$, $s \to -s$.  We arrive at the following.
\begin{corollary}\label{cor:Jost}
    The matrix functions
    \begin{align*}
    \mathbf{R}_+(x;\lambda)&=\begin{bmatrix}
    \mathbf{r}_{1,-}(x;\lambda) & \mathbf{r}_{2,+}(x;\lambda)
\end{bmatrix}, \quad (x,\lambda) \in  \mathbb R \times \mathbb C^+,\\
    \mathbf{R}_-(x;\lambda)&=\begin{bmatrix}
    \mathbf{r}_{2,-}(x;\lambda) & \mathbf{r}_{1,+}(x;\lambda)
\end{bmatrix}, \quad (x,\lambda) \in  \mathbb R \times\mathbb C^-,
\end{align*}
are uniformly bounded and analytic on their domains of definition and extend continuously up to the real axis.
\end{corollary}

We now state the following lemma.  The proof is also found in Appendix~\ref{sec:Proof}.
\begin{lemma}\label{lemma:JostDeriv}
    Suppose $\mathbf{Q}\in L^1(\mathbb{R},\mathbb{C}^{2\times 2})\cap L^\infty(\mathbb{R},\mathbb{C}^{2\times 2})$ and $\partial_x\mathbf{Q}\in L^1(\mathbb{R},\mathbb{C}^{2\times 2})$. Then there exist constants $C_j$, $j = 1,2,3,4$, such that
    \begin{equation*}
        \|\partial_x\mathbf{r}_{1,-}(x;\lambda)\|\leq C_1,\quad\|\partial_x\mathbf{r}_{2,+}(x;\lambda)\|\leq C_2,\quad (x,\lambda)\in \mathbb R \times \mathbb{C}^+,
    \end{equation*}
    and
    \begin{equation*}
        \|\partial_x\mathbf{r}_{2,-}(x;\lambda)\|\leq C_3,\quad \|\partial_x\mathbf{r}_{1,+}(x;\lambda)\| \leq C_4,\quad(x,\lambda)\in \mathbb R \times \mathbb{C}^-.
    \end{equation*}
\end{lemma}

\begin{remark}
Lemma~\ref{lemma:JostDeriv} could be strengthened to show that $\partial_x\mathbf{r}_{\ell,\pm}=\mathcal{O}(1/|\lambda|)$, $\ell\in{1,2}$, $|\lambda|\rightarrow\infty$, for $\lambda$ in the appropriate half-plane. One way to obtain this estimate is through an additional integration by parts, which requires the stronger assumption that $\partial_x^2\mathbf{Q}\in L^1(\mathbb{R},\mathbb{C}^{2\times2})$. 
\end{remark}
Note that in the proof of Lemma~\ref{lemma:JostDeriv}, we showed that
\begin{equation*}
    \mathbf{r}_{\ell,+}(x;\lambda)\rightarrow\begin{bmatrix}
        0\\1
    \end{bmatrix}+\mathcal{O}\left(\frac{1}{|\lambda|}\right)\quad\text{and}\quad \mathbf{r}_{\ell,-}(x;\lambda)\rightarrow\begin{bmatrix}
        1\\
        0
    \end{bmatrix}+\mathcal{O}\left(\frac{1}{|\lambda|}\right),\quad\ell\in\{1,2\},
\end{equation*}
as $|\lambda|\rightarrow\infty$ in the appropriate half-plane.

We have defined two solutions, $\mathbf{M}_1(x;\lambda)$ and $\mathbf{M}_2(x;\lambda)$. Liouville's formula~\cite{chicone_ordinary_2006} implies that the determinants of these solutions are constant in $x$. Since the determinants are nonzero at $\pm\infty$, the columns of these matrix solutions define a linearly independent set which spans the solution space. Thus, the members of one set can be written as a linear combination of the members of the other set. This fact is captured by the scattering relation
\begin{equation*}
    \mathbf{M}_1(x;\lambda)=\mathbf{M}_2(x;\lambda)\mathbf{S}(\lambda), \quad \mathbf{S}(\lambda)=\begin{bmatrix}
        a(\lambda)&B(\lambda)\\
        b(\lambda)&A(\lambda)
    \end{bmatrix},
\end{equation*}
where the elements of $\mathbf{S}(\lambda)$ are known as the scattering coefficients. Let $\mathbf{M}_1(x;\lambda)=\begin{bmatrix}
    \mathbf{m}_{1,-}(x;\lambda) & \mathbf{m}_{1,+}(x;\lambda)
\end{bmatrix}$ and $\mathbf{M}_2(x;\lambda)=\begin{bmatrix}
    \mathbf{m}_{2,-}(x;\lambda) & \mathbf{m}_{2,+}(x;\lambda)
\end{bmatrix}$. Then
\begin{equation*}
    \begin{cases}
        \mathbf{m}_{1,-}(x;\lambda)=a(\lambda)\mathbf{m}_{2,-}(x;\lambda)+b(\lambda)\mathbf{m}_{2,+}(x;\lambda),\\
        \mathbf{m}_{1,+}(x;\lambda)=B(\lambda)\mathbf{m}_{2,-}(x;\lambda)+A(\lambda)\mathbf{m}_{2,+}(x;\lambda).
    \end{cases}
\end{equation*}
The scattering data can be found by taking Wronskians of the scattering relations with respect to $\mathbf{m}_{\ell,\pm}(x;\lambda)$, $\ell\in\{1,2\}$. We find that
\begin{align*}
        a(\lambda)&=\frac{\text{W}(\mathbf{m}_{1,-}(x;\lambda),\mathbf{m}_{2,+}(x;\lambda))}{\text{W}(\mathbf{m}_{2,-}(x;\lambda),\mathbf{m}_{2,+}(x;\lambda))},\quad B(\lambda)=\frac{\text{W}(\mathbf{m}_{1,+}(x;\lambda),\mathbf{m}_{2,+}(x;\lambda))}{\text{W}(\mathbf{m}_{2,-}(x;\lambda),\mathbf{m}_{2,+}(x;\lambda))}, \\
        b(\lambda)&=\frac{\text{W}(\mathbf{m}_{1,-}(x;\lambda),\mathbf{m}_{2,-}(x;\lambda))}{\text{W}(\mathbf{m}_{2,+}(x;\lambda),\mathbf{m}_{2,-}(x;\lambda))},\quad A(\lambda)=\frac{\text{W}(\mathbf{m}_{1,+}(x;\lambda),\mathbf{m}_{2,-}(x;\lambda))}{\text{W}(\mathbf{m}_{2,+}(x;\lambda),\mathbf{m}_{2,-}(x;\lambda))}.
\end{align*}
By Liouville’s formula, $\text{W}(\mathbf{m}_{2,+}(x;\lambda),\mathbf{m}_{2,-}(x;\lambda))=\text{det}\left(\begin{bmatrix}
    \mathbf{m}_{2,+}(x;\lambda) & \mathbf{m}_{2,-}(x;\lambda)
\end{bmatrix}\right)=-1$, and \newline $\text{W}(\mathbf{m}_{2,-}(x;\lambda),\mathbf{m}_{2,+}(x;\lambda))=1$. The scattering data can then be simplified to
\begin{align*}
    a(\lambda)&=\text{W}(\mathbf{m}_{1,-}(x;\lambda),\mathbf{m}_{2,+}(x;\lambda)),\quad B(\lambda) = \text{W}(\mathbf{m}_{1,+}(x;\lambda),\mathbf{m}_{2,+}(x;\lambda)),\\
    b(\lambda)&=-\text{W}(\mathbf{m}_{1,-}(x;\lambda),\mathbf{m}_{2,-}(x;\lambda)), \quad A(\lambda)=-\text{W}(\mathbf{m}_{1,+}(x;\lambda),\mathbf{m}_{2,-}(x;\lambda)).
\end{align*}
By the analyticity of $\mathbf{m}_{\ell,\pm}(x;\lambda)$, $\ell\in\{1,2\}$, $a(\lambda)$ extends analytically to $\mathbb{C}^+$ and $A(\lambda)$ to $\mathbb{C}^-$. Using the relation $\mathbf{m}_{\ell,\pm}(x;\lambda)=\mathbf{r}_{\ell,\pm}(x;\lambda)e^{\mp i\lambda x}$ together with the large-$|\lambda|$ limits of the Jost solutions in the appropriate half-planes, we obtain
\begin{align*}
    \mathbf{m}_{1,-}(x;\lambda) = e^{-i\lambda x}\left(\begin{bmatrix}1\\0\end{bmatrix}+\mathcal{O}\left(\frac{1}{|\lambda|}\right)\right),\quad
    \mathbf{m}_{2,+}(x;\lambda) = e^{i\lambda x}\left(\begin{bmatrix}0\\1\end{bmatrix}+\mathcal{O}\left(\frac{1}{|\lambda|}\right)\right),
\end{align*}
as $|\lambda|\to\infty$. Substituting these expressions into the Wronskian representation of the scattering coefficients and using the bilinearity of the Wronskian together with cancellation of the exponential factors, we find
\begin{equation*}
    a(\lambda)
    =
    1 + \mathcal{O}\left(\frac{1}{|\lambda|}\right),
\end{equation*}
as $|\lambda|\to\infty$ with $\lambda\in\overline{\mathbb{C}^+}$. It can be shown similarly that $A(\lambda)\to 1$ as $|\lambda|\to\infty$ with $\lambda\in\overline{\mathbb{C}^-}$.

Since the spectrum of $\mathcal{T}$ lies on the real axis, we seek solutions to~\eqref{eq:DiracSpatial} that are analytic in the upper and lower half-planes, respectively. To do this, we use $\mathbf{m}_{\ell,\pm}(x;\lambda)$, $\ell\in\{1,2\}$, to build solutions that are analytic in the appropriate half-plane. We define
\begin{align*}
    \mathbf{M}_+(x;\lambda) &:= \begin{bmatrix}
        \mathbf{m}_{1,-}(x;\lambda) & \mathbf{m}_{2,+}(x;\lambda)
    \end{bmatrix}=\begin{bmatrix}
        m_{1,-}^{(1)}(x;\lambda) & m_{2,+}^{(1)}(x;\lambda)\\
        m_{1,-}^{(2)}(x;\lambda) & m_{2,+}^{(2)}(x;\lambda)
    \end{bmatrix},\\
    \mathbf{M}_-(x;\lambda) &:= \begin{bmatrix}
        \mathbf{m}_{2,-}(x;\lambda) & \mathbf{m}_{1,+}(x;\lambda)
    \end{bmatrix}=\begin{bmatrix}
        m_{2,-}^{(1)}(x;\lambda) & m_{1,+}^{(1)}(x;\lambda)\\
        m_{2,-}^{(2)}(x;\lambda) & m_{1,+}^{(2)}(x;\lambda)
    \end{bmatrix},
\end{align*}
where $m_{j,\pm}^{(n)}$, $n = 1,2$, denotes the two elements of $\mathbf{m}_{j,\pm}$. To solve~\eqref{eq:DiracSpatial}, we apply Proposition~\ref{prop:VarofParam} again, using $\mathbf{M}_\pm(x;\lambda)$ as our homogeneous solution. This results in
\begin{equation}\label{eq:VarofParam}
    \mathbf{v}_\pm(x;\lambda) = \mathbf{M}_\pm(x;\lambda) \mathbf{a}_\pm(\lambda) + 
    \mathbf{M}_\pm(x;\lambda) \int_{x_1}^x \mathbf{M}_\pm(s;\lambda)^{-1} \mathbf{f}(s) \, \mathrm{d}s, 
    \quad \mathbf{a}_\pm(\lambda) \in \mathbb{C}^{2\times 1}.
\end{equation}
It is often convenient to write $\mathbf{v}_\pm(x;\lambda)$ in terms of the Jost functions since they are bounded uniformly in $x$ and $\lambda$. Recall that $\mathbf{R}_+(x;\lambda)=\begin{bmatrix}
    \mathbf{r}_{1,-}(x;\lambda) & \mathbf{r}_{2,+}(x;\lambda)
\end{bmatrix}=\mathbf{M}_+(x;\lambda)e^{i\lambda x\sigma_3}$. A uniform bound is established in Corollary~\ref{cor:Jost}.

Rewriting~\eqref{eq:VarofParam} in terms of the Jost functions
and choosing $x_1=-\infty$ results in 
\begin{align*}
    \mathbf{v}_+(x;\lambda)&=\mathbf{R}_+(x;\lambda)e^{-i\lambda x\sigma_3}\mathbf{a}_+(\lambda)+\mathbf{R}_+(x;\lambda)\int_{-\infty}^x e^{-i\lambda(x-s)\sigma_3}\mathbf{R}_+(x;\lambda)^{-1}\mathbf{f}(s)\mathrm{d}s\\
    &=\mathbf{R}_+(x;\lambda)\left(\begin{bmatrix}
        e^{-i\lambda x}a^{(1)}(\lambda)\\
        e^{i\lambda x}a^{(2)}(\lambda)
    \end{bmatrix}+\frac{1}{a(\lambda)}\begin{bmatrix}
        \int_{-\infty}^x e^{-i\lambda(x-s)}\begin{bmatrix}
            r_{2,+}^{(2)}(s;\lambda) & -r_{2,+}^{(1)}(s;\lambda)
        \end{bmatrix}\mathbf{f}(s)\mathrm{d}s\\
        \int_{-\infty}^x e^{i\lambda(x-s)}\begin{bmatrix}
            -r_{1,-}^{(2)}(s;\lambda) & r_{1,-}^{(1)}(s;\lambda)
        \end{bmatrix}\mathbf{f}(s)\mathrm{d}s
    \end{bmatrix}\right).
\end{align*}
We choose $\mathbf{a}_+(\lambda)$ to ensure analyticity for $\lambda\in\mathbb{C}^+$. This construction is summarized in Lemma~\ref{lemma:analyticity}. 

\begin{lemma}\label{lemma:analyticity}
Let $\mathbf{f}\in\mathbb{C}^{2\times 1}$ be a vector-valued Schwartz-class function and $\mathbf{Q}\in L^\infty(\mathbb{R},\mathbb{C}^{2\times 2})$. Define
\begin{align*}
    \mathbf{a}_+(\lambda)&=\begin{bmatrix}
        -\frac{1}{a(\lambda)}\int_{-\infty}^\infty \begin{bmatrix}
            m_{2,+}^{(2)}(s;\lambda) & -m_{2,+}^{(1)}(s;\lambda)
        \end{bmatrix}\mathbf{f}(s)\mathrm{d}s\\
        0
    \end{bmatrix},\\
    \mathbf{a}_-(\lambda)&=\begin{bmatrix}
        0\\
        -\frac{1}{A(\lambda)}\int_{-\infty}^\infty \begin{bmatrix}
            -m_{2,-}^{(2)}(s;\lambda) & m_{2,-}^{(1)}(s;\lambda)
        \end{bmatrix}\mathbf{f}(s)\mathrm{d}s
    \end{bmatrix}.
\end{align*}
Then~\eqref{eq:VarofParam} is analtyic for $\lambda\in\mathbb{C}^\pm \setminus \sigma(\mathcal T)$ and satisfies~\eqref{eq:DiracSpatial}, and
\[
\mathbf{v}_\pm(x;\lambda),\partial_x \mathbf{v}_\pm(x;\lambda) \in L^2(\mathbb{R},\mathbb{C}^2), \quad \lambda\in\mathbb{C}^\pm \setminus \sigma(\mathcal T).
\]
Therefore,
\[
\mathbf{v}_\pm(x;\lambda) = (\mathcal{T}-\lambda)^{-1} (i \sigma_3 \mathbf{f})(x), \quad \lambda\in\mathbb{C}^\pm \setminus \sigma(\mathcal T).
\]
\end{lemma}

\begin{proof}

 In this proof, we will concern ourselves with $\mathbf{v}_+(x;\lambda)$.  Analogous arguments can be made for $\mathbf{v}_-(x;\lambda)$.  We note that when $\tau = 1$, $\mathcal T$ is self-adjoint and $\sigma(\mathcal T) \cap \mathbb C^\pm  = \varnothing$.
With the choice of $\mathbf{a}_+(\lambda)$ we find that 
\begin{align}\label{eq:Jost}
    \mathbf{v}_+(x;\lambda)&=\frac{1}{a(\lambda)}\Bigg(\mathbf{r}_{1,-}(x;\lambda)\int_\infty^xe^{-i\lambda(x-s)}\begin{bmatrix}
            r_{2,+}^{(2)}(s;\lambda)&-r_{2,+}^{(1)}(s;\lambda)\end{bmatrix}\mathbf{f}(s)\mathrm{d}s \\\nonumber
            &\quad\quad\quad\quad\quad+ \mathbf{r}_{2,+}(x;\lambda)\int_{-\infty}^xe^{i\lambda(x-s)}\begin{bmatrix}
                -r_{1,-}^{(2)}(s;\lambda) & r_{1,-}^{(1)}(s;\lambda)
            \end{bmatrix}\mathbf{f}(s)\mathrm{d}s\Bigg).
\end{align}
Provided that $a(\lambda) \neq 1$, $\mathbf{v}_+(x;\lambda)$ is a solution of \eqref{eq:DiracSpatial}.

Now, suppose $\lambda \in \mathbb C^+$ is such that $a(\lambda) = 0$.  Then the solutions $\mathbf{m}_{1,-}(x;\lambda)$ and $\mathbf{m}_{2,+}(x;\lambda)$ are linearly dependent. From the boundedness of the Jost solutions $\mathbf{r}_{1,-}(x;\lambda)$ and $\mathbf{r}_{2,+}(x;\lambda)$ we conclude that $\mathbf{m}_{1,-}(x;\lambda)$ is an $L^2$ eigenfunction of $\mathcal T$, giving $\lambda \in \sigma(\mathcal T)$ and 
\begin{align*}
    \mathcal A_0 := \{ \lambda \in \mathbb C^+ \!:\! a(\lambda) = 0 \} \subset \mathbb C^+ \cap \sigma(\mathcal T).
\end{align*}

Write $\lambda=\xi+i\eta$ with $\eta>0$.
Each term in the integrands of~\eqref{eq:Jost} is analytic in the upper half-plane. 
Each exponential factor satisfies $|e^{- i\lambda(x-s)}|=e^{-\eta(\lambda)|s-x|}\leq 1$, $x\leq s$ and $|e^{i\lambda(x-s)}|=e^{-\eta|x-s|}\leq 1$, $x\geq s$.
Since $\mathbf{r}_{1,-}(x;\lambda)$ and $\mathbf{r}_{2,+}(x;\lambda)$ are uniformly bounded in $x$ and $\lambda$ (see Corollary~\ref{cor:Jost}) and $\mathbf{f}$ is Schwartz-class, the integrands and their derivatives are dominated by an integrable function independent of $\lambda$. Complex differentiability can be established with standard arguments using the dominated convergence theorem.

We now show that 
\begin{equation*}
    \mathbf{v}_+(x;\lambda)=(\mathcal{T}-\lambda)^{-1}(i\sigma_3\mathbf{f})(x), \quad \lambda \in \mathbb C^+ \setminus \mathcal A_0.
\end{equation*}
Each integral term in~\eqref{eq:Jost} is bounded by convolution with the kernel $K_\eta(t)=e^{-\eta t}\mathbbm{1}_{t\geq0}\in L^1(\mathbb{R},\mathbb{C})$, with $\|K_\eta\|_{L^1(\mathbb{R},\mathbb{C})}=1/\eta$. Since $\mathbf{f}\in L^2(\mathbb{R},\mathbb{C}^2)$ and the Jost solutions are uniformly bounded, Young’s inequality implies that $\mathbf{v}_+(x;\lambda)\in L^2(\mathbb{R},\mathbb{C}^2)$.  This also establishes that $\mathcal A_0 = \mathbb C^+ \cap \sigma(\mathcal T)$.

To show that $\partial_x\mathbf{v}_\pm(\diamond;\lambda)\in L^2(\mathbb{R},\mathbb{C}^2)$, we use the differential equation~\eqref{eq:DiracSpatial}, which may be written as
\begin{equation*}
    \partial_x\mathbf{v}(x;\lambda)=-i\lambda\mathbf{v}(x;\lambda)+\mathbf{Q}(x)\mathbf{v}(x;\lambda)+\mathbf{f}(x).
\end{equation*}
Because $\mathbf{v}_+(x;\lambda)\in L^2(\mathbb{R},\mathbb{C}^2)$, $\mathbf{f}(x)\in L^2(\mathbb{R},\mathbb{C}^2)$, and $\mathbf{Q}(x)\in L^\infty(\mathbb{R},\mathbb{C}^{2\times 2})$, each term on the right-hand side belongs to $L^2(\mathbb{R},\mathbb{C}^2)$.
The same arguments apply for $\text{Im}(\lambda)<0$ using $\mathbf{v}_-(x;\lambda)$.

\end{proof}

The recovery formula for $\mathbf{f}(x)$ is established in Lemma~\ref{lemma:Recovery}.

\begin{lemma}[Recovery formula]\label{lemma:Recovery}
    Let $q\in\mathcal{S}(\mathbb{R},\mathbb{C})$ and $\mathbf{f}\in\mathcal{S}(\mathbb{R},\mathbb{C}^2)$. Then, along the imaginary axis 
    \begin{equation*}
        \lim_{\substack{\lambda=i\eta\\ \eta\rightarrow\pm\infty}}i\lambda\sigma_3\mathbf{v}(x;\lambda)=\mathbf{f}(x).
    \end{equation*}
\end{lemma}
\begin{proof}
    We prove the result for $\mathbf{v}_+(x;\lambda)$. The $\mathbf{v}_-(x;\lambda)$ case is analogous. Using the representation of $\mathbf{v}_+(x;\lambda)$ in terms of the Jost solutions,
    \begin{equation*}
    \mathbf{v}_+(x;\lambda)
    = \frac{1}{a(\lambda)}\Big(
    \mathbf{r}_{1,-}(x;\lambda)\,v_1(x;\lambda)
    + \mathbf{r}_{2,+}(x;\lambda)\,v_2(x;\lambda)
    \Big),
    \end{equation*}
    where
    \begin{align*}
    v_1(x;\lambda)
    &= \int_\infty^x e^{-i\lambda(x-s)}
    \begin{bmatrix}
    r_{2,+}^{(2)}(s;\lambda) & -r_{2,+}^{(1)}(s;\lambda)
    \end{bmatrix}
    \mathbf{f}(s)\,\mathrm{d}s\\
    v_2(x;\lambda)&=\int_{-\infty}^xe^{i\lambda(x-s)}\begin{bmatrix}
                -r_{1,-}^{(2)}(s;\lambda) & r_{1,-}^{(1)}(s;\lambda)
            \end{bmatrix}\mathbf{f}(s)\mathrm{d}s.
    \end{align*}
    We multiply through by $i\lambda\sigma_3$ to obtain
    \begin{equation*}
        i\lambda\sigma_3\mathbf{v}_+(x;\lambda)=\frac{1}{a(\lambda)}\left(\sigma_3\mathbf{r}_{1,-}(x;\lambda)[i\lambda v_1(x;\lambda)] + \sigma_3\mathbf{r}_{2,+}(x;\lambda)[i\lambda v_2(x;\lambda)]\right).
    \end{equation*}
    By Lemma~\ref{lemma:JostDeriv}, along the imaginary axis $\lambda=i\eta$ with $\eta\rightarrow +\infty$,
    \begin{equation*}
        \mathbf{r}_{1,-}(x;i\eta)=\begin{bmatrix}
            1\\0
        \end{bmatrix}+\mathcal{O}\left(\frac{1}{\eta}\right),\quad \mathbf{r}_{2,+}(x;i\eta)=\begin{bmatrix}
            0\\1
        \end{bmatrix} + \mathcal{O}\left(\frac{1}{\eta}\right),\quad a(i\eta)=1+\mathcal{O}\left(\frac{1}{\eta}\right),
    \end{equation*}
    uniformly in $x$. It therefore suffices to show that 
    \begin{equation*}
        -\eta v_1(x;i\eta)\rightarrow \begin{bmatrix}
            1 & 0
        \end{bmatrix}\mathbf{f}(x),\quad -\eta v_2(x;i\eta)\rightarrow -\begin{bmatrix}
            0 & 1
        \end{bmatrix}\mathbf{f}(x),\quad \eta\rightarrow\infty.
    \end{equation*}

    We have that
    \begin{align*}
        -\eta v_1(x;i\eta)&=\eta\int_x^\infty e^{-\eta(s-x)} \begin{bmatrix}
    r_{2,+}^{(2)}(s;i\eta) & -r_{2,+}^{(1)}(s;i\eta)
    \end{bmatrix}
    \mathbf{f}(s)\,\mathrm{d}s\\
    &=\eta\int_0^\infty e^{-\eta t}\begin{bmatrix}
        1 & 0
    \end{bmatrix}\mathbf{f}(x+t)\mathrm{d}t + \mathcal{O}\left(\frac{1}{\eta}\right).
    \end{align*}
    Since $\vec f(x)$ is assumed continuous, the result for $v_1(x;i \eta)$ follows.    An analogous computation for $v_2(x;i \eta)$ gives the result.
\end{proof}

Our ultimate goal is to form an \gls{rhp}. Solving this \gls{rhp} will give an alternate representation of $\vec v(x;\lambda)$ making the recovery formula in Lemma~\ref{lemma:Recovery} useful. To do this, we must establish a jump condition and a condition as $|\lambda|\rightarrow\infty$, $\lambda\in\mathbb{C}\setminus\mathbb{R}$. 

\begin{lemma}[Jump condition]\label{lemma:jump}
    Let $\mathbf{v}(x;\lambda)$ be defined by
    \begin{equation*}
        \mathbf{v}(x;\lambda)=\begin{cases}
            \mathbf{v}_+(x;\lambda),\quad&\lambda\in\mathbb{C}^+,\\
            \mathbf{v}_-(x;\lambda),\quad&\lambda\in\mathbb{C}^-,
        \end{cases}
    \end{equation*}
    where $\mathbf{v}_\pm(x;\lambda)$ are given in~\eqref{eq:VarofParam} with $\mathbf{a}_\pm(\lambda)$ defined in Lemma~\ref{lemma:analyticity}. 
    Then for each fixed $x\in\mathbb{R}$, the boundary values
    \begin{equation*}
        \mathbf{v}^\pm(x;\lambda)=\lim_{\epsilon\rightarrow0^+}\mathbf{v}(x;\lambda\pm i\epsilon)
    \end{equation*}
    exist for $\lambda\in\mathbb{R}$ and satisfy
    \begin{equation}\label{eq:jump}
        \mathbf{v}^+(x;\lambda)-\mathbf{v}^-(x;\lambda)=c_+(\lambda)\mathbf{m}_{2,+}(x;\lambda)+c_-(\lambda)\mathbf{m}_{2,-}(x;\lambda),
    \end{equation}
    where
    \begin{align*}
        c_+(\lambda)&=\frac{1}{a(\lambda)}\int_{-\infty}^\infty\begin{bmatrix}
            -m_{1,-}^{(2)}(s;\lambda) & m_{1,-}^{(1)}(s;\lambda)
        \end{bmatrix}\mathbf{f}(s)\mathrm{d}s,\\
        c_-(\lambda)&=-\frac{1}{A(\lambda)}\int_{-\infty}^\infty\begin{bmatrix}
            m_{1,+}^{(2)}(s;\lambda) & -m_{1,+}^{(1)}(s;\lambda)
        \end{bmatrix}\mathbf{f}(s)\mathrm{d}s.
    \end{align*}
\end{lemma}
\begin{proof}
    By the dominated convergence theorem, using the continuity and uniform bounds that follow from applying Corollary~\ref{cor:Jost}, the boundary values $\mathbf{v}^\pm(x;\lambda)$ exist pointwise for each fixed $x\in\mathbb{R}$. To derive the jump condition, we first determine the leading-order asymptotic behavior of $\mathbf{v}^\pm(x;\lambda)$ as $x\rightarrow+\infty$. The absolute convergence of the integrals in~\eqref{eq:VarofParam} imply the leading order behavior as $x\rightarrow+\infty$ is given by
    \[
    \mathbf v^+(x;\lambda) = c_+(\lambda)\mathbf m_{2,+}(x;\lambda) + o(1),\qquad
    \mathbf v^-(x;\lambda) = -c_-(\lambda)\mathbf m_{2,-}(x;\lambda) + o(1).
    \]

    For $\lambda\in\mathbb{R}$, the boundary values $\mathbf{v}^\pm(x;\lambda)$ satisfy $(\mathcal{T}-\lambda)\mathbf{v}^\pm(x;\lambda)=i\sigma_3\mathbf{f}(x)$. Hence, their difference
    \begin{equation*}
        \mathbf{w}(x;\lambda):=\mathbf{v}^+(x;\lambda)-\mathbf{v}^-(x;\lambda)
    \end{equation*}
    satisfies the homogeneous equation $(\mathcal{T}-\lambda)\mathbf{w}(x;\lambda)=\mathbf{0}$. For fixed $\lambda\in\mathbb{R}$, the solution space of the homogeneous system is spanned by $\mathbf{m}_{2,+}(x;\lambda)$ and $\mathbf{m}_{2,-}(x;\lambda)$. Thus, there exist coefficients $\alpha(\lambda)$, $\beta(\lambda)$ such that
    \begin{equation}\label{eq:w}
        \mathbf{w}(x;\lambda)=\alpha(\lambda)\mathbf{m}_{2,+}(x;\lambda)+\beta(\lambda)\mathbf{m}_{2,-}(x;\lambda).
    \end{equation}
  We obtain 
    \begin{equation*}
        \mathbf{w}(x;\lambda)=c_+(\lambda)\mathbf{m}_{2,+}(x;\lambda)+c_-(\lambda)\mathbf{m}_{2,-}(x;\lambda)+o(1)\quad\text{as}\quad x\rightarrow\infty.
    \end{equation*}
    Comparing this with the representation~\eqref{eq:w}, and using the linear independence of $\mathbf{m}_{2,\pm}$, we conclude that $\alpha(\lambda)=c_+(\lambda)$ and $\beta(\lambda)=c_-(\lambda)$. This yields~\eqref{eq:jump}.
\end{proof}

\begin{lemma}[Large-$|\lambda|$ behavior of $\mathbf{v}_\pm$]\label{lemma:LargeLambda}
Let $q \in \mathcal{S}(\mathbb{R},\mathbb{C})$ and $\mathbf{f} \in \mathcal{S}(\mathbb{R},\mathbb{C}^2)$. Then for $\lambda \in \mathbb{C}^\pm$,
\begin{equation*}
\mathbf{v}_\pm(x;\lambda)= \mathcal{O}\!\left(\frac{1}{|\lambda|}\right),
\quad \text{as } |\lambda|\to\infty,
\end{equation*}
uniformly in $x \in \mathbb{R}$.
\end{lemma}

\begin{proof}
We prove the result for $\lambda \in \mathbb{C}^+$. The $\mathbb{C}^-$ case is analogous.
We use the representation of $\mathbf{v}_+$ in terms of the Jost solutions,
\begin{equation*}
\mathbf{v}_+(x;\lambda)
= \frac{1}{a(\lambda)}\Big(
\mathbf{r}_{1,-}(x;\lambda)\,v_1(x;\lambda)
+ \mathbf{r}_{2,+}(x;\lambda)\,v_2(x;\lambda)
\Big),
\end{equation*}
where
\begin{align*}
v_1(x;\lambda)
&= \int_\infty^x e^{-i\lambda(x-s)}
\begin{bmatrix}
r_{2,+}^{(2)}(s;\lambda) & -r_{2,+}^{(1)}(s;\lambda)
\end{bmatrix}
\mathbf{f}(s)\,\mathrm{d}s = \int_\infty^x e^{-i\lambda(x-s)}g_1(s;\lambda)\mathrm{d}s,\\
v_2(x;\lambda)&=\int_{-\infty}^xe^{i\lambda(x-s)}\begin{bmatrix}
            -r_{1,-}^{(2)}(s;\lambda) & r_{1,-}^{(1)}(s;\lambda)
        \end{bmatrix}\mathbf{f}(s)\mathrm{d}s = \int_{-\infty}^x e^{i\lambda(x-s)}g_2(s;\lambda)\mathrm{d}s.
\end{align*}
The Jost solutions and their $x$-derivatives are uniformly bounded in $x$ and $\lambda$ (see Corollary~\ref{cor:Jost} and Lemma~\ref{lemma:JostDeriv}).
Under the standing assumption $\tau=1$, $a(\lambda)\neq 0$ for $\lambda \in \mathbb{C}^+$, and $1/a(\lambda)=1+\mathcal{O}(1/|\lambda|)$ as $|\lambda|\rightarrow\infty$ with $\lambda\in\overline{\mathbb{C}^+}$. Thus, it suffices to show that $v_1$ and $v_2$ are $\mathcal{O}(1/|\lambda|)$.
Integrating by parts yields
\begin{equation*}
v_1(x;\lambda)
= \frac{1}{i\lambda}
\left(
g_1(x;\lambda)
- \int_\infty^x e^{-i\lambda(x-s)} g_1'(s;\lambda)\,\mathrm{d}s
\right),
\end{equation*}
where
\begin{equation*}
g_1'(s;\lambda)
=
\begin{bmatrix}
\partial_s r_{2,+}^{(2)}(s;\lambda) & -\partial_s r_{2,+}^{(1)}(s;\lambda)
\end{bmatrix}
\mathbf{f}(s)
+
\begin{bmatrix}
r_{2,+}^{(2)}(s;\lambda) & -r_{2,+}^{(1)}(s;\lambda)
\end{bmatrix}
\mathbf{f}'(s).
\end{equation*}
Using the uniform bounds on the Jost solutions,
\begin{equation*}
|g_1(s;\lambda)| \le C |\mathbf{f}(s)|, 
\quad 
|g_1'(s;\lambda)| \le C'\big(|\mathbf{f}(s)| + |\mathbf{f}'(s)|\big),
\end{equation*}
so that
\begin{equation*}
|v_1(x;\lambda)|
\le \frac{\hat{C}}{|\lambda|}
\left(
\|\mathbf{f}\|_{L^\infty(\mathbb{R},\mathbb{C}^2)}
+ \|\mathbf{f}\|_{L^1(\mathbb{R},\mathbb{C}^2)}
+ \|\mathbf{f}'\|_{L^1(\mathbb{R},\mathbb{C}^2)}
\right).
\end{equation*}
Since $\mathbf{f} \in \mathcal{S}(\mathbb{R},\mathbb{C}^2)$, this implies
\begin{equation*}
v_1(x;\lambda) = \mathcal{O}(|\lambda|^{-1}),\quad |\lambda|\rightarrow\infty.
\end{equation*} 
An analogous estimate holds for $v_2(x;\lambda)$.
\end{proof}

Utilizing the jump from Lemma~\ref{lemma:jump} 
and the condition at infinity from Lemma~\ref{lemma:LargeLambda}, 
we arrive at an \gls{rhp} that applies when $\sigma(\mathcal T) \subset \mathbb R$ (recall $\tau =1$ guarantees this).
\begin{rhp}\label{rhp:2}
    Find $\mathbf{v}(x;\diamond):\mathbb{C}\setminus\mathbb{R}\rightarrow\mathbb{C}^{2\times1}$ analytic such that for $\lambda\in\mathbb{R}$,
    \begin{align*}
        \mathbf{v}^+(x;\lambda)-\mathbf{v}^-(x;\lambda)
        &=c_+(\lambda)\mathbf{m}_{2,+}(x;\lambda)+c_-(\lambda)\mathbf{m}_{2,-}(x;\lambda),\\
        \mathbf{v}(x;\lambda)&=\mathcal O\left(\frac{1}{|\lambda|} \right),
\quad \text{as } |\lambda|\to\infty,\ \lambda\in\mathbb{C}\setminus\mathbb{R}.
    \end{align*}
\end{rhp}
The solution to RHP~\ref{rhp:2} is given by the Cauchy integral,
\begin{equation*}
    \mathbf{v}(x;\lambda)=\frac{1}{2\pi i}\int_{-\infty}^\infty\frac{c_+(\lambda')\mathbf{m}_{2,+}(x;\lambda')+c_-(\lambda')\mathbf{m}_{2,-}(x;\lambda')}{\lambda'-\lambda}\mathrm{d}\lambda'.
\end{equation*}
Substituting the Cauchy integral solution into the recovery formula established in Lemma~\ref{lemma:Recovery} gives
\begin{equation*}
    \mathbf{f}(x)=i\sigma_3\lim_{\substack{\lambda=i\eta\\ \eta\rightarrow+\infty}}\left(\frac{1}{2\pi i}\int_{-\infty}^\infty\frac{\lambda\left(c_+(\lambda')\mathbf{m}_{2,+}(x;\lambda')+c_-(\lambda')\mathbf{m}_{2,-}(x;\lambda')\right)}{\lambda'-\lambda}\mathrm{d}\lambda'\right).
\end{equation*}
Since $\mathbf{m}_{\ell,\pm}(x;\lambda')=e^{\pm i\lambda' x}\mathbf{r}_{\ell,\pm}(x;\lambda'),~\ell\in\{1,2\}$, the coefficients $c_\pm(\lambda')$ are oscillatory integrals in $s$ with amplitude functions involving Jost functions and $\mathbf{f}$. Assume that $q\in L^1(\mathbb{R},\mathbb{C})\cap L^\infty(\mathbb{R},\mathbb{C})$ and that $\mathbf{f}$ has two weak derivatives in $L^1(\mathbb{R},\mathbb{C}^2)$. Lemma~\ref{lemma:JostDeriv} guarantees boundedness of the Jost functions and their first $x$-derivatives. Under these assumptions, the amplitudes arising in $c_\pm(\lambda')$ are twice differentiable in $s$ with all derivatives in $L^1(\mathbb{R})$, so two integrations by parts in $s$ are justified. This yields
$c_\pm(\lambda') = 
\mathcal{O}(|\lambda'|^{-2})$ as $|\lambda'|\to\infty$. Consequently,
\begin{equation*}
    \int_{-\infty}^\infty\left|c_+(\lambda')\mathbf{m}_{2,+}(x;\lambda')
    +c_-(\lambda')\mathbf{m}_{2,-}(x;\lambda')\right|\mathrm{d}\lambda'<\infty.
\end{equation*}
The dominated convergence theorem allows the limit to pass inside the integral, yielding
\begin{equation*}
    \mathbf{f}(x)=-\frac{\sigma_3}{2\pi}\int_{-\infty}^\infty \left(c_+(\lambda)\mathbf{m}_{2,+}(x;\lambda)+c_-(\lambda)\mathbf{m}_{2,-}(x;\lambda)\right)\mathrm{d}\lambda.
\end{equation*}
This results in the following transform pair for the variable-coefficient Dirac equation, which we call the Dirac transform, when $\sigma(\mathcal T) \subset \mathbb R$,
\begin{align}\label{eq:InvTrans}
    \mathbf{f}(x)&=-\frac{\sigma_3}{2\pi}\int_{-\infty}^\infty\begin{bmatrix}
        \mathbf{m}_{2,+}(x;\lambda) & \mathbf{m}_{2,-}(x;\lambda)
    \end{bmatrix}\hat{\mathbf{f}}(\lambda)\mathrm{d}\lambda,\\
    \hat{\mathbf{f}}(\lambda)&=\begin{bmatrix}
        c_+(\lambda)\\
        c_-(\lambda)
    \end{bmatrix}=\begin{bmatrix}
        \frac{1}{a(\lambda)}\int_{-\infty}^\infty\begin{bmatrix}
            -m_{1,-}^{(2)}(s;\lambda) & m_{1,-}^{(1)}(s;\lambda)
        \end{bmatrix}\mathbf{f}(s)\mathrm{d}s\\
        -\frac{1}{A(\lambda)}\int_{-\infty}^\infty\begin{bmatrix}
            m_{1,+}^{(2)}(s;\lambda) & -m_{1,+}^{(1)}(s;\lambda)
        \end{bmatrix}\mathbf{f}(s)\mathrm{d}s
    \end{bmatrix}.
\end{align}
We refer to $1/a(\lambda)$ and $1/A(\lambda)$ as the transmission coefficients.

This derivation is equivalent to the recovery of $\mathbf{f}$ given by Theorem~\ref{thm:Stone} when $\tau = 1$. Using RHP~\ref{rhp:2} and Lemma~\ref{lemma:analyticity}, we find that 
\begin{align*}
    i\sigma_3\mathbf{f}(x) &= -(i\sigma_3)\frac{\sigma_3}{2\pi}\int_{-\infty}^\infty \left(c_+(\lambda)\mathbf{m}_{2,+}(\lambda)+c_-(\lambda)\mathbf{m}_{2,-}(x;\lambda)\right)\mathrm{d}\lambda = \frac{1}{2\pi i}\int_{-\infty}^\infty \left( \mathbf{v}^+(x;\lambda)-\mathbf{v}^-(x;\lambda)\right)\mathrm{d}\lambda\\
    &= \frac{1}{2\pi i}\int_{-\infty}^\infty \left(\lim_{\epsilon\rightarrow0^+}\left[\mathbf{v}(x;\lambda+i\epsilon)-\mathbf{v}(x;\lambda-i\epsilon)\right]\right)\mathrm{d}\lambda\\
    &=\frac{1}{2\pi i}\int_{-\infty}^\infty \left(\lim_{\epsilon\rightarrow0^+}\left[(\mathcal{T}-\lambda-i\epsilon)^{-1}-(\mathcal{T}-\lambda+i\epsilon)^{-1}\right](i\sigma_3\mathbf{f})(x)\right)\mathrm{d}\lambda\\
    &=\lim_{\epsilon\rightarrow 0^+}\left(\frac{1}{2\pi i}\int_{-\infty}^\infty\left[(\mathcal{T}-\lambda-i\epsilon)^{-1}-(\mathcal{T}-\lambda+i\epsilon)^{-1}\right](i\sigma_3\mathbf{f})(x)\mathrm{d}\lambda\right).
\end{align*}
This matches the recovery formula in Theorem~\ref{thm:Stone} for~\eqref{eq:DiracResolvent}.
Note that we were able to pass the limit outside of the integral using an argument similar to that in Section~\ref{sec:Fourier}.

\subsection{Alternate Representation of the Inverse Transform}\label{sec:DiracAlternate}

In Section~\ref{sec:Fourier}, the inverse transform recovery formula was expressed via a solution to an inhomogeneous \gls{rhp}. We now look for a similar formulation for the Dirac transform. We assume $\tau = 1$ in this subsection. We begin by constructing a jump matrix $\mathbf{G}(x;\lambda)$ from the scattering data. Define the matrix
\begin{equation*}
    \mathbf{K}(x;\lambda)=\begin{cases}
        \begin{bmatrix}
            \mathbf{m}_{1,-}(x;\lambda) & \mathbf{m}_{2,+}(x;\lambda)
        \end{bmatrix},\quad\lambda\in\mathbb{C}^+,\\[2ex]
        \begin{bmatrix}
            \mathbf{m}_{2,-}(x;\lambda) & \mathbf{m}_{1,+}(x;\lambda)
        \end{bmatrix},\quad\lambda\in\mathbb{C}^-.
    \end{cases}
\end{equation*}
We showed above that,
\begin{equation*}
\mathbf m_{\ell,-}(x;\lambda)=e^{-i\lambda x}(\mathbf e_1+o(1)),\quad
\mathbf m_{\ell,+}(x;\lambda)=e^{i\lambda x}(\mathbf e_2+o(1)),\quad\ell\in\{1,2\},\quad |\lambda|\rightarrow\infty.
\end{equation*}
Consequently,
\begin{equation*}
\mathbf K(x;\lambda)
=
(\mathbb I + o(1))
\begin{bmatrix}
e^{-i\lambda x} & 0\\
0 & e^{i\lambda x}
\end{bmatrix}\quad \text{as}\quad |\lambda|\to\infty.
\end{equation*}
Using the scattering relations, we then have
\begin{equation*}
    \mathbf{K}^+(x;\lambda)=\mathbf{K}^-(x;\lambda)\begin{bmatrix}
        \frac{1}{A(\lambda)} & -\frac{B(\lambda)}{A(\lambda)}\\
        \frac{b(\lambda)}{A(\lambda)} & \frac{1}{A(\lambda)}
    \end{bmatrix}.
\end{equation*}
The determinant of the jump matrix is not unity:
\[
\det \begin{bmatrix}
\frac{1}{A(\lambda)} & -\frac{B(\lambda)}{A(\lambda)}\\
\frac{b(\lambda)}{A(\lambda)} & \frac{1}{A(\lambda)}
\end{bmatrix} = \frac{1}{A^2(\lambda)} + \frac{b(\lambda)B(\lambda)}{A^2(\lambda)} \neq 1.
\] 
We factor the jump matrix as
\begin{equation*}
    \begin{bmatrix}
        \frac{1}{A(\lambda)} & -\frac{B(\lambda)}{A(\lambda)}\\
        \frac{b(\lambda)}{A(\lambda)} & \frac{1}{A(\lambda)}
    \end{bmatrix}=\begin{bmatrix}
        1 & 0\\
        0 & \frac{1}{A(\lambda)}
    \end{bmatrix}\begin{bmatrix}
        \frac{1}{a(\lambda)A(\lambda)} & -\frac{B(\lambda)}{A(\lambda)}\\
        \frac{b(\lambda)}{a(\lambda)} & 1
    \end{bmatrix}\begin{bmatrix}
        a(\lambda) & 0\\
        0 & 1
    \end{bmatrix},
\end{equation*}
where $a(\lambda)$ is analytic in $\mathbb{C}^+$ and $A(\lambda)$ is analytic in $\mathbb{C}^-$. This factorization separates the analytic contributions in the upper and lower half-planes and will be useful in constructing an \gls{rhp} with analytic jump factors.
Next, define 
\begin{equation*}
    \mathbf{L}(x;\lambda)=\begin{cases}\begin{bmatrix}
        \frac{\mathbf{m}_{1,-}(x;\lambda)}{a(\lambda)} & \mathbf{m}_{2,+}(x;\lambda)   
    \end{bmatrix},\quad\lambda\in\mathbb{C}^+,\\[5pt]
    \begin{bmatrix}
        \mathbf{m}_{2,-}(x;\lambda) & \frac{\mathbf{m}_{1,+}(x;\lambda)}{A(\lambda)}
    \end{bmatrix},\quad\lambda\in\mathbb{C}^-.
    \end{cases}
\end{equation*}
Then
\[
\mathbf{L}^+(x;\lambda) = \mathbf{L}^-(x;\lambda) 
\begin{bmatrix}
\frac{1}{a(\lambda)A(\lambda)} & -\frac{B(\lambda)}{A(\lambda)}\\
\frac{b(\lambda)}{a(\lambda)} & 1
\end{bmatrix},\quad
\mathbf{L}(x;\lambda)=\left(\mathbb{I}+o(1)\right)
\begin{bmatrix}
e^{-i\lambda x} & 0\\
0 & e^{i\lambda x}
\end{bmatrix}\quad \text{as}\quad |\lambda|\to\infty.
\]
Using $\text{det}(\mathbf{S}(\lambda))=1$, we have
\begin{equation*}
    \frac{1}{a(\lambda)A(\lambda)}=1-\frac{b(\lambda)B(\lambda)}{a(\lambda)A(\lambda)},\quad\text{det}\begin{bmatrix}
        \frac{1}{a(\lambda)A(\lambda)} & -\frac{B(\lambda)}{A(\lambda)}\\
        \frac{b(\lambda)}{a(\lambda)} & 1
    \end{bmatrix}=1.
\end{equation*}
Finally, to adjust the asymptotic condition to the identity, define
\begin{equation*}
    \mathbf{T}(x;\lambda)=\mathbf{L}(x;\lambda)\begin{bmatrix}
        e^{i\lambda x} & 0\\
        0 & e^{-i\lambda x}
    \end{bmatrix}.
\end{equation*}
We then define the jump matrix
\begin{equation*}
    \mathbf{G}(x;\lambda)=\begin{bmatrix}
        1-\rho_1(\lambda)\rho_2(\lambda) & -\rho_2(\lambda)e^{-2i\lambda x}\\
        \rho_1(\lambda)e^{2i\lambda x} & 1
    \end{bmatrix},
\end{equation*}
where $\rho_1(\lambda)=b(\lambda)/a(\lambda)$ and $\rho_2(\lambda)=B(\lambda)/A(\lambda)$ are the reflection coefficients. This formulation corresponds to the right scattering problem. For the left scattering problem, the definitions of the reflection coefficients and the resulting jump matrix must be modified. These changes are discussed in Section~\ref{sec:LeftScattering}.

Consider the following homogeneous \gls{rhp}. 
\begin{rhp}\label{rhp:3}
    Find $\mathbf{T}(x;\diamond):\mathbb{C}\setminus\mathbb{R}\rightarrow\mathbb{C}^{2\times2}$ analytic such that for $\lambda\in\mathbb{R}$,
    \begin{align*}
        \mathbf{T}^+(x;\lambda)&=\mathbf{T}^-(x;\lambda)\mathbf{G}(x;\lambda),
        \end{align*}
        and
        \begin{align*}
        \mathbf{T}(x;\lambda) &= \mathbb{I}+o(1)\quad\text{as}\quad |\lambda|\rightarrow\infty.
    \end{align*}
\end{rhp}
The fact that $\det(\mathbf{G}(x;\lambda))=1$ implies that $\det(\mathbf{T}^+(x;\lambda))=\det(\mathbf{T}^-(x;\lambda))$ and $\det(\mathbf{T}(x;\infty))=1$, so that $\det(\mathbf{T}(x;\lambda))=1$. We now consider the following inhomogeneous \gls{rhp}. 
\begin{rhp}\label{rhp:4}
    Find $\mathbf{V}(x;\diamond):\mathbb{C}\setminus\mathbb{R}\rightarrow\mathbb{C}^{2\times2}$ analytic such that for $\lambda\in\mathbb{R}$,
    \begin{align*}
        \mathbf{V}^+(x;\lambda)&=\mathbf{V}^-(x;\lambda)\mathbf{G}(x;\lambda)+\mathbf{J}(x;\lambda),
        \end{align*}
        and
        \begin{align*}
        \mathbf{V}(x;\lambda)&=o(1)\quad\text{as}\quad|\lambda|\rightarrow\infty.
    \end{align*}
\end{rhp}
At this stage, the jump inhomogeneity $\mathbf{J}(x;\lambda)$ is left unspecified. It will be constructed subsequently so that the corresponding solution $\mathbf{V}(x;\lambda)$ encodes the desired transform. We can factor $\mathbf{G}(x;\lambda)=[\mathbf{T}^{-1}]^{-}(x;\lambda)\mathbf{T}^+(x;\lambda)$ from RHP~\ref{rhp:3}, allowing us to rewrite the jump condition from RHP~\ref{rhp:4} as
\begin{equation*}
    [\mathbf{V}\mathbf{T}^{-1}]^+(x;\lambda)=[\mathbf{V}\mathbf{T}^{-1}]^-(x;\lambda)+\mathbf{J}(x;\lambda)[\mathbf{T}^{-1}]^{+}(x;\lambda).
\end{equation*}
The conditions of RHPs~\ref{rhp:3} and~\ref{rhp:4} imply that $[\mathbf{V}\mathbf{T}^{-1}](x;\lambda)=o(1)$ as $|\lambda|\rightarrow\infty$. 
Assuming RHPs~\ref{rhp:3} and~\ref{rhp:4} admit solutions with the prescribed jump and normalization conditions, the Plemelj lemma implies
\begin{equation*}
    [\mathbf{V}\mathbf{T}^{-1}](x;\lambda)=\frac{1}{2\pi i}\int_{-\infty}^\infty \frac{\mathbf{J}(x;\lambda')[\mathbf{T}^{-1}]^+(x;\lambda')}{\lambda'-\lambda}\mathrm{d}\lambda'.
\end{equation*}
Multiplying by $\lambda$ and taking a limit as $|\lambda|\rightarrow\infty$ results in
\begin{equation}\label{eq:InhomoRHP}
    \lim_{|\lambda|\rightarrow\infty}\lambda\mathbf{V}(x;\lambda)\mathbf{T}^{-1}(x;\lambda)=\lim_{|\lambda|\rightarrow\infty}\lambda\mathbf{V}(x;\lambda)=-\frac{1}{2\pi i}\int_{-\infty}^\infty\mathbf{J}(x;\lambda')[\mathbf{T}^{-1}]^+(x;\lambda')\mathrm{d}\lambda',
\end{equation}
where the last equality follows from yet another application of the dominated convergence theorem. Recall from~\eqref{eq:InvTrans} that
\begin{align*}
    \mathbf{f}(x)=
    -\frac{\sigma_3}{2\pi}\int_{-\infty}^\infty \left(c_+(\lambda)\mathbf{m}_{2,+}(x;\lambda)+c_-(\lambda)\mathbf{m}_{2,-}(x;\lambda)\right)\mathrm{d}\lambda.
\end{align*}
This tells us we can expect to be able to recover $\mathbf{f}(x)$ if we choose $\mathbf{J}(x;\lambda)$ such that $\mathbf{J}(x;\lambda)[\mathbf{T}^{-1}]^+(x;\lambda)$ contains $c_+(\lambda)\mathbf{m}_{2,+}(x;\lambda)$ and $c_-(\lambda)\mathbf{m}_{2,-}(x;\lambda)$, respectively.

We have 
\begin{equation*}
    [\mathbf{T}^{-1}]^+(x;\lambda)=\begin{bmatrix}
        m_{2,+}^{(2)}(x;\lambda)e^{-i\lambda x} & -m_{2,+}^{(1)}(x;\lambda)e^{-i\lambda x}\\
        -\frac{m_{1,-}^{(2)}(x;\lambda)}{a(\lambda)}e^{i\lambda x} & \frac{m_{1,-}^{(1)}(x;\lambda)}{a(\lambda)}e^{i\lambda x}
    \end{bmatrix}.
\end{equation*}
To recover \(c_+(\lambda)\mathbf{m}_{2,+}(x;\lambda)\), define $\mathbf{j}_1(x;\lambda)=\begin{bmatrix}
    c_+(\lambda)e^{i\lambda x} & 0
\end{bmatrix}$. Then we find that
\begin{equation*}
    \mathbf{j}_1(x;\lambda)[\mathbf{T}^{-1}]^+(x;\lambda)=c_+(\lambda)\begin{bmatrix}
        m_{2,+}^{(2)}(x;\lambda) & -m_{2,+}^{(1)}(x;\lambda)
    \end{bmatrix},
\end{equation*}
giving us access to $c_+(\lambda)\mathbf{m}_{2,+}(x;\lambda)$, as desired. We now want to make an analogous choice for $\mathbf{j}_2(x;\lambda)$ to be able to extract $c_-(\lambda)\mathbf{m}_{2,-}(x;\lambda)$. This is not possible using $[\mathbf{T}^{-1}]^+(x;\lambda)$ in its current form. We instead need to rewrite $[\mathbf{T}^{-1}]^+(x;\lambda)$ using RHP~\ref{rhp:3}, resulting in $[\mathbf{T}^{-1}]^+(x;\lambda)=\mathbf{G}^{-1}(x;\lambda)[\mathbf{T}^{-1}]^-(x;\lambda)$.
Choosing $\mathbf{j}_2(x;\lambda)=\begin{bmatrix}
    -\rho_1(\lambda)c_-(\lambda)e^{i\lambda x} & -c_-(\lambda)e^{-i\lambda x}
\end{bmatrix}$ then results in
\begin{equation*}
    \mathbf{j}_2(x;\lambda)[\mathbf{T}^{-1}]^+(x;\lambda) = c_-(\lambda)\begin{bmatrix}
        m_{2,-}^{(2)}(x;\lambda) & -m_{2,-}^{(1)}(x;\lambda)
    \end{bmatrix}.
\end{equation*}
Let 
\begin{equation*}
    \mathbf{J}(x;\lambda)=\begin{bmatrix}
        \mathbf{j}_1(x;\lambda)\\
        \mathbf{j}_2(x;\lambda)
    \end{bmatrix}.
\end{equation*}
Then
\begin{equation*}
    \mathbf{J}(x;\lambda)[\mathbf{T}^{-1}]^+(x;\lambda)=\mathbf{A}(x;\lambda)=\begin{bmatrix}
        c_+(\lambda)m_{2,+}^{(2)}(x;\lambda) & -c_+(\lambda)m_{2,+}^{(1)}(x;\lambda)\\
        c_-(\lambda)m_{2,-}^{(2)}(x;\lambda) & -c_-(\lambda)m_{2,-}^{(1)}(x;\lambda)
    \end{bmatrix}=\begin{bmatrix}
        A^{(11)} & A^{(12)}\\
        A^{(21)} & A^{(22)}
    \end{bmatrix}.
\end{equation*}
It is then clear that the inverse transform can be formulated via
\begin{equation*}
    \mathbf{f}(x) = -\frac{\sigma_3}{2\pi}\int_{-\infty}^\infty\begin{bmatrix}
        -(A^{(12)}+A^{(22)})\\
        A^{(11)} + A^{(21)}
    \end{bmatrix}\mathrm{d}\lambda=-\frac{\sigma_3}{2\pi}\int_{-\infty}^\infty\left(\text{diagvec}\left(\begin{bmatrix}
        -1 & -1\\
        1 & 1
    \end{bmatrix}\mathbf{A}(x;\lambda)\sigma_1\right)\right)\mathrm{d}\lambda.
\end{equation*}
Here \(\operatorname{diagvec}(\mathbf{M}):=\begin{bmatrix}
    M^{(11)} & M^{(22)}
\end{bmatrix}^T\). We can then use~\eqref{eq:InhomoRHP} to write $\mathbf{f}(x)$ in terms of the solution to RHP~\ref{rhp:4} as follows:
\begin{equation}\label{eq:InverseTransInhomo}
    \mathbf{f}(x)=i\sigma_3\left(\lim_{\substack{\lambda=i\eta\\ \eta\rightarrow+\infty}}\left[\lambda\cdot\text{diagvec}\left(\begin{bmatrix}
        -1 & -1\\
        1 & 1
    \end{bmatrix}\mathbf{V}(x;\lambda)\sigma_1\right)\right]\right).
\end{equation}
While solving \glspl{ode} naturally produces a forward transform, the formulation of a systematic inverse-transform framework remains comparatively underdeveloped.
The alternate representation of the inverse transform given by~\eqref{eq:InverseTransInhomo} shows that solving an inhomogeneous \gls{rhp} leads to the inverse transform. 

\begin{remark}
    Although the inverse transform can be formulated via an inhomogeneous \gls{rhp}, we do not employ this representation in our numerical implementation. This representation is included for its potential theoretical implications.
\end{remark}

\subsection{Accounting for Poles}\label{sec:DiracPoles}

Up to this point, we have assumed $a(\lambda)\neq 0$ and $A(\lambda)\neq 0$, corresponding to the self-adjoint setting ($\tau=1$), where no discrete spectrum arises. In the non-self-adjoint case ($\tau=-1$), the scattering coefficients may admit zeros in $\mathbb{C}^\pm$, giving rise to discrete spectral contributions. Such zeros introduce poles into our \gls{rhp}, and we must account for them to properly recover the solution. Let $\{z_{j,+}\}_{j=1}^{n}$ denote the zeros of $a(\lambda)$ in $\mathbb{C}^+$, and $\{z_{j,-}\}_{j=1}^{n}$ denote the zeros of $A(\lambda)$ in $\mathbb{C}^-$. We assume a finite number of simple zeros.  In the specific examples we consider below, we have exponentially decaying potential functions, guaranteeing that these sets are finite. Moreover, $a(\lambda), A(\lambda)\neq0$ for all $\lambda\in\mathbb{R}$.

For $\lambda\in\mathbb{C}^+$, the residues of $\mathbf{v}(x;\lambda)$ at the zeros of $a(\lambda)$ are
\begin{align*}
    \mathrm{Res}_{\lambda=z_{j,+}}\mathbf{v}(x;\lambda) = \mathbf{w}_{j,+}(x;z_{j,+}) 
    &= \frac{1}{a'(z_{j,+})}\mathbf{m}_{2,+}(x;z_{j,+}) 
    \int_{-\infty}^{\infty} \begin{bmatrix}-m_{1,-}^{(2)}(s;z_{j,+}) & m_{1,-}^{(1)}(s;z_{j,+})\end{bmatrix} \mathbf{f}(s)\, \mathrm{d}s\\
    &=\frac{a(z_{j,+})c_+(z_{j,+})}{a'(z_{j,+})}\mathbf{m}_{2,+}(x;z_{j,+}),
\end{align*}
and similarly, for $\lambda\in\mathbb{C}^-$, the residues of $\mathbf{v}(x;\lambda)$ at the zeros of $A(\lambda)$ are
\begin{align*}
    \mathrm{Res}_{\lambda=z_{j,-}}\mathbf{v}(x;\lambda) = \mathbf{w}_{j,-}(x;z_{j,-}) 
    &= \frac{1}{A'(z_{j,-})}\mathbf{m}_{2,-}(x;z_{j,-}) 
    \int_{-\infty}^{\infty} \begin{bmatrix} m_{1,+}^{(2)}(s;z_{j,-}) & -m_{1,+}^{(1)}(s;z_{j,-}) \end{bmatrix} \mathbf{f}(s)\, \mathrm{d}s\\
    &=-\frac{A(z_{j,-})c_-(z_{j,-})}{A'(z_{j,-})}\mathbf{m}_{2,-}(x;z_{j,-}).
\end{align*}
Hence, $\mathbf{v}(x;\lambda)$ can be expressed as
\begin{equation*}
    \mathbf{v}(x;\lambda)=\frac{1}{2\pi i}\int_{-\infty}^\infty\frac{c_+(\lambda')\mathbf{m}_{2,+}(x;\lambda')+c_-(\lambda')\mathbf{m}_{2,-}(x;\lambda')}{\lambda'-\lambda}\mathrm{d}\lambda' + \sum_{j=1}^{n}\frac{\mathbf{w}_{j,+}(x;z_{j,+})}{\lambda-z_{j,+}} + \sum_{j=1}^{n}\frac{\mathbf{w}_{j,-}(x;z_{j,-})}{\lambda-z_{j,-}}.
\end{equation*}
Using Lemma~\ref{lemma:Recovery}, the recovery formula for $\mathbf{f}(x)$ including contributions from the poles is
\begin{align}\label{eq:RecoveryPoles}
    \mathbf{f}(x)&=
    i\sigma_3\lim_{\substack{\lambda=i\eta\\ \eta\rightarrow + \infty}}\lambda\mathbf{v}(x;\lambda)\\\nonumber
    &=-\frac{\sigma_3}{2\pi}\int_{-\infty}^\infty \left(c_+(\lambda)\mathbf{m}_{2,+}(x;\lambda)+c_-(\lambda)\mathbf{m}_{2,-}(x;\lambda)\right)\mathrm{d}\lambda + i\sigma_3\sum_{j=1}^{n}\frac{a(z_{j,+})c_+(z_{j,+})}{a'(z_{j,+})}\mathbf{m}_{2,+}(x;z_{j,+})\\
    &\quad\quad - i\sigma_3\sum_{j=1}^{n}\frac{A(z_{j,-})c_-(z_{j,-})}{A'(z_{j,-})}\mathbf{m}_{2,-}(x;z_{j,-}).\nonumber
\end{align}

We note that a zero of $a(\lambda)$ at $\lambda=z_{j,+}$ implies that $\mathbf{m}_{1,-}(x;z_{j,+})=b_{j,+}\mathbf{m}_{2,+}(x;z_{j,+})$ for some constant $b_{j,+}$. Therefore, 
\begin{equation*}
    \text{Res}_{\lambda=z_{j,+}}\frac{\mathbf{m}_{1,-}(x;\lambda)}{a(\lambda)}=\frac{\mathbf{m}_{1,-}(x;z_{j,+})}{a'(z_{j,+})}=\frac{b_{j,+}}{a'(z_{j,+})}\mathbf{m}_{2,+}(x;z_{j,+}).
\end{equation*}
Similarly, for a constant $b_{j,-}$ satisfying $\mathbf{m}_{1,+}(x;z_{j,-})=b_{j,-}\mathbf{m}_{2,-}(x;z_{j,-})$,
\begin{equation*}
    \text{Res}_{\lambda=z_{j,-}}\frac{\mathbf{m}_{1,+}(x;\lambda)}{A(\lambda)}=\frac{\mathbf{m}_{1,+}(x;z_{j,-})}{A'(z_{j,-})}=\frac{b_{j,-}}{A'(z_{j,-})}\mathbf{m}_{2,-}(x;z_{j,-}).
\end{equation*}
This leads us to define the norming constants
\begin{equation}\label{eq:normingConstants}
    c_{j,+}=\frac{b_{j,+}}{a'(z_{j,+})},\quad c_{j,-}=\frac{b_{j,-}}{A'(z_{j,-})}.
\end{equation}


\subsection{Computation}\label{sec:DiracComputation}

Again using our Fourier transform example as a guide, we break down the computation of generalized transform pairs as follows:
\begin{enumerate}
    \item (\emph{Forward transform}) Solve the associated \glspl{ode} numerically using the \gls{uclm} to compute the scattering data, comprising both continuous spectral functions $c_\pm(\lambda)$ and discrete spectral data (poles), with the latter detailed in Section~\ref{sec:PolesComputation}.
    \item (\emph{Inverse transform}) Approximate a solution to the associated \gls{rhp} by solving the equivalent singular integral equation numerically. Use this solution to reconstruct the generalized eigenfunctions as a function of the spectral variable, and evaluate the inverse transform.
\end{enumerate}

\textbf{Step 1:} We obtain computationally convenient expressions for $c_+(\lambda)$ and $c_-(\lambda)$ using the Levin-type ideas described in Section~\ref{sec:Fourier}.
Let $\mathbf{p}_\pm:\mathbb{R}\rightarrow\mathbb{C}^{2\times 1}$ satisfy
\begin{align*}
    &\mathbf{p}_\pm'(x)+\left(i\lambda\sigma_3-\mathbf{Q}(x)\right)\mathbf{p}_\pm(x)=\mathbf{f}(x),\quad\lambda\in\mathbb{R},\\
    &\mathbf{p}_\pm(\pm\infty)=\mathbf{0}.
\end{align*}
The boundary value problems for $\mathbf{p}_\pm$ are solved on a truncated interval $[0,L]$ (or $[L,0]$ for $L<0$) via a Chebyshev spectral discretization. The decay condition at $\pm\infty$ is enforced at the finite endpoint through homogeneous Dirichlet conditions.
We find that
\begin{equation*}
    \mathbf{p}_{\pm}(x;\lambda)=
    \begin{cases}
        \frac{1}{a(\lambda)}\mathbf{M}_+(x;\lambda)\begin{bmatrix}
            \int_{\pm\infty}^x\begin{bmatrix}
                m_{2,+}^{(2)}(s;\lambda) & -m_{2,+}^{(1)}(s;\lambda)
            \end{bmatrix}\mathbf{f}(s)\mathrm{d}s\\[.5em]
            \int_{\pm\infty}^x\begin{bmatrix}
                -m_{1,-}^{(2)}(s;\lambda) & m_{1,-}^{(1)}(s;\lambda)
            \end{bmatrix}\mathbf{f}(s)\mathrm{d}s
        \end{bmatrix},\quad\lambda\in\mathbb{C}^+,\\[2em]
        \frac{1}{A(\lambda)}\mathbf{M}_-(x;\lambda)\begin{bmatrix}
            \int_{\pm\infty}^x\begin{bmatrix}
                m_{1,+}^{(2)}(s;\lambda) & -m_{1,+}^{(1)}(s;\lambda)
            \end{bmatrix}\mathbf{f}(s)\mathrm{d}s\\[.5em]
            \int_{\pm\infty}^x\begin{bmatrix}
                -m_{2,-}^{(2)}(s;\lambda) & m_{2,-}^{(1)}(s;\lambda)
            \end{bmatrix}\mathbf{f}(s)\mathrm{d}s
        \end{bmatrix},\quad\lambda\in\mathbb{C}^-.
    \end{cases}
\end{equation*}
Then $\mathbf{p}_+$ and $\mathbf{p}_-$ satisfy
\begin{align*}
(\mathbf{p}_+-\mathbf{p}_-)(x;\lambda)=
\begin{cases}
\left(-\frac{1}{a(\lambda)}\int_{-\infty}^\infty
\begin{bmatrix}
m_{2,+}^{(2)}(s;\lambda) & -m_{2,+}^{(1)}(s;\lambda)
\end{bmatrix}\mathbf{f}(s)\mathrm{d}s\right)\mathbf{m}_{1,-}(x;\lambda)\\
\hspace{2em} -\,c_+(\lambda)\mathbf{m}_{2,+}(x;\lambda), \quad\quad\quad\quad\quad\quad\quad\quad\quad\quad\quad\quad\quad\quad\quad\quad\quad\quad\quad \lambda\in\mathbb{C}^+,\\[0.6em]
c_-(\lambda)\mathbf{m}_{2,-}(x;\lambda)\\
\hspace{2em} -\left(\frac{1}{A(\lambda)}\int_{-\infty}^\infty
\begin{bmatrix}
-m_{2,-}^{(2)}(s;\lambda) & m_{2,-}^{(1)}(s;\lambda)
\end{bmatrix}\mathbf{f}(s)\mathrm{d}s\right)\mathbf{m}_{1,+}(x;\lambda),
\quad \lambda\in\mathbb{C}^-.
\end{cases}
\end{align*}
Taking the Wronskian of $(\mathbf{p}_+-\mathbf{p}_-)(x;\lambda)$ with respect to 
$\mathbf{m}_{1,-}(x;\lambda)$ and 
$\mathbf{m}_{1,+}(x;\lambda)$ for $\lambda \in \mathbb{R}$ gives
\begin{align*}
    c_+(\lambda)&=-\frac{\text{W}((\mathbf{p}_+-\mathbf{p}_-)(x;\lambda),\mathbf{m}_{1,-}(x;\lambda))}{\text{W}(\mathbf{m}_{2,+}(x;\lambda),\mathbf{m}_{1,-}(x;\lambda))}=\frac{\text{W}((\mathbf{p}_+-\mathbf{p}_-)(x;\lambda),\mathbf{m}_{1,-}(x;\lambda))}{a(\lambda)},\\
    c_-(\lambda)&=\frac{\text{W}((\mathbf{p}_+-\mathbf{p}_-)(x;\lambda),\mathbf{m}_{1,+}(x;\lambda))}{\text{W}(\mathbf{m}_{2,-}(x;\lambda),\mathbf{m}_{1,+}(x;\lambda))}=\frac{\text{W}((\mathbf{p}_+-\mathbf{p}_-)(x;\lambda),\mathbf{m}_{1,+}(x;\lambda))}{A(\lambda)}.
\end{align*}
Thus, the forward transforms are expressed in terms of solutions to boundary-value problems, which are computationally more convenient than the original oscillatory integral forms.
To compute solutions to the boundary-value problems, we again use the the ultraspherical rectangular collocation method \cite{trogdon_ultraspherical_2024}, essentially amounting to a \gls{uclm} implementation. Figure~\ref{fig:DiracForward} shows the forward transforms $c_+(\lambda)$ and $c_-(\lambda)$ for $\mathbf{f}(x)=\begin{bmatrix}
    \exp(-x^2)&\exp(-x^2)
\end{bmatrix}^T$, $q(x)=\exp(-x^2)$, and $\tau=1$. Figure~\ref{fig:cp_err} shows the maximum absolute error between $c_+(\lambda)$ and a highly resolved computation of $c_+(\lambda)$ as a function of the number of nodes used.
\begin{figure}[htbp]
    \centering
    \begin{subfigure}[b]{0.49\textwidth}
        \centering
        \includegraphics[width=\textwidth]{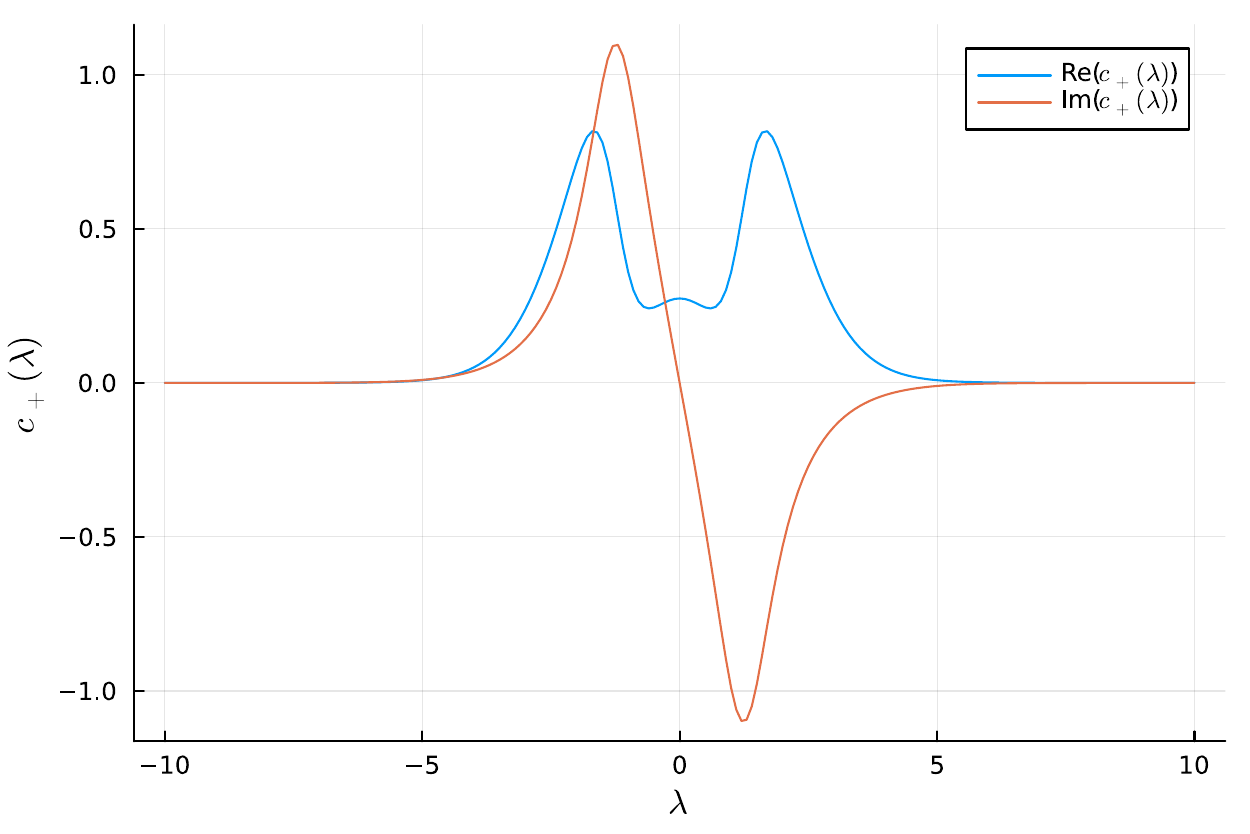}
        \caption{$c_+(\lambda)$}
        \label{fig:Dirac_cp}
    \end{subfigure}
    \hfill
    \begin{subfigure}[b]{0.49\textwidth}
        \centering
        \includegraphics[width=\textwidth]{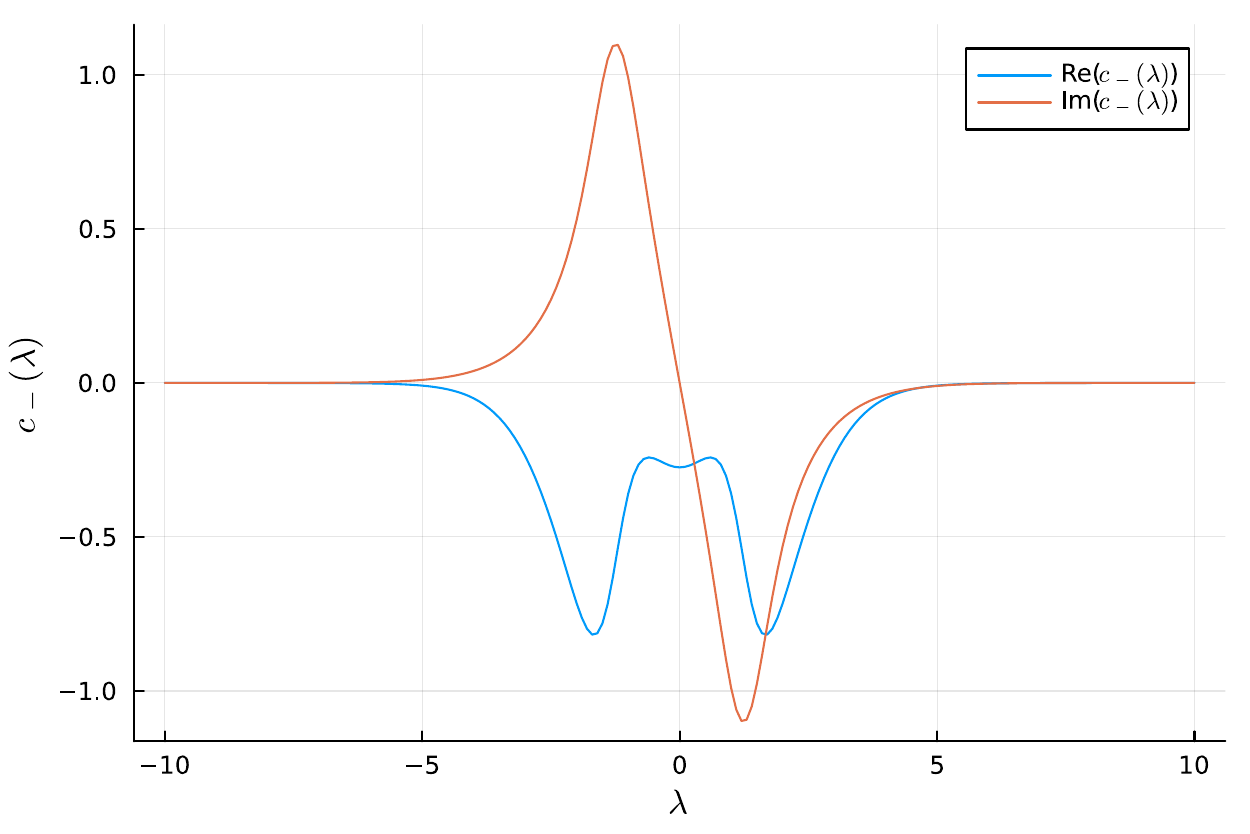}
        \caption{$c_-(\lambda)$}
        \label{fig:Dirac_cm}
    \end{subfigure}
    \caption{Forward transforms $c_+(\lambda)$ and $c_-(\lambda)$ of $f(x)=\begin{bmatrix}
    \exp(-x^2)&\exp(-x^2)
\end{bmatrix}^T$ and $q(x)=\exp(-x^2)$ computed using the \gls{uclm}. The real parts are negatives of one another, while the imaginary parts agree, reflecting an underlying conjugation symmetry between the forward transforms.}
    \label{fig:DiracForward}
\end{figure}

\begin{figure}
    \centering
    \includegraphics[width=0.5\linewidth]{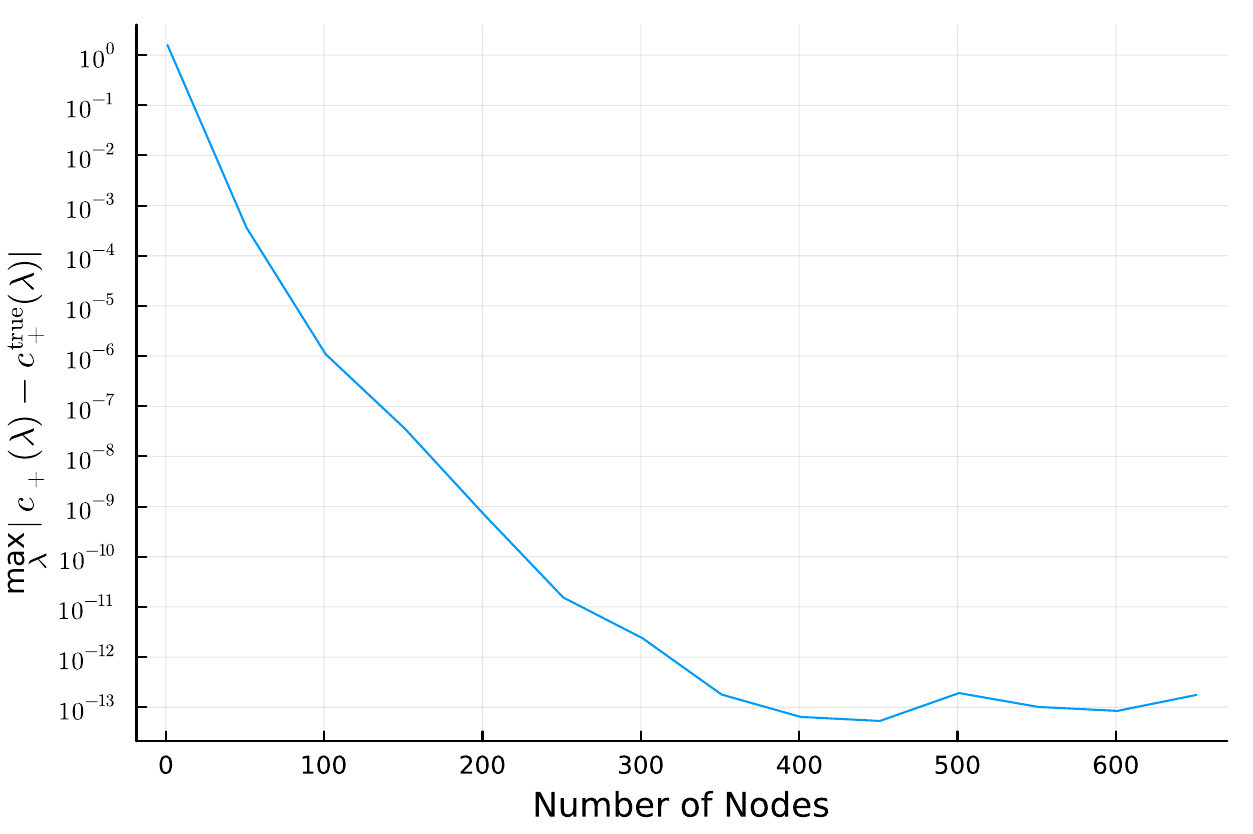}
    \caption{Max absolute error between $c_+(\lambda)$ and a reference solution computed via the same method using 550 coefficients, for which the coefficients decay to machine precision evaluated on a uniform grid in $\lambda \in [-10,10]$ as a function of the number of nodes used.}
    \label{fig:cp_err}
\end{figure}

\begin{remark}
    Note that we could alternatively take the Wronskian of $(\mathbf{p}_+-\mathbf{p}_-)(x;\lambda)$ with respect to $\mathbf{m}_{2,+}(x;\lambda)$ (using the $\mathbb{C}^+$ expression) and with respect to $\mathbf{m}_{2,-}(x;\lambda)$ (using the $\mathbb{C}^-$ expression) to solve for 
    \begin{equation*}
        -\int_{-\infty}^\infty\begin{bmatrix}
            m_{2,+}^{(2)}(s;\lambda) & -m_{2,+}^{(1)}(s;\lambda)
        \end{bmatrix}\mathbf{f}(s)\mathrm{d}s\quad\text{and}\quad -\int_{-\infty}^\infty\begin{bmatrix}
            -m_{2,-}^{(2)}(s;\lambda) & m_{2,-}^{(1)}(s;\lambda)
        \end{bmatrix}\mathbf{f}(s)\mathrm{d}s,
    \end{equation*}
    respectively. This alternative Wronskian approach is useful when computing $\hat{\mathbf{w}}_{j,\pm}$ in the recovery formula for left scattering (see Section~\ref{sec:LeftScatteringPoles}).
\end{remark}

\textbf{Step 2:} We now compute the inverse transform. RHP~\ref{rhp:3} is useful here because its solution is given by
\begin{equation}\label{eq:DiracRHPSol}
    \mathbf{T}^+(x;\lambda)=\begin{bmatrix}
        \frac{\mathbf{m}_{1,-}(x;\lambda)}{a(\lambda)} & \mathbf{m}_{2,+}(x;\lambda)
    \end{bmatrix}\quad\text{and}\quad\mathbf{T}^-(x;\lambda)=\begin{bmatrix}
        \mathbf{m}_{2,-}(x;\lambda) & \frac{\mathbf{m}_{1,+}(x;\lambda)}{A(\lambda)}
    \end{bmatrix}.
\end{equation}
We compute an approximate solution to RHP~\ref{rhp:3}, resulting in an approximation of $\mathbf{m}_{2,\pm}(x;\lambda)$ as a function of $\lambda$.
For an unknown function $\mathbf{U}(x;\lambda)$ set 
\begin{equation*}
    \mathbf{T}^\pm(x;\lambda)=\mathcal{C}^\pm_\mathbb{R}\mathbf{U}(x;\lambda)+\mathbb{I},
\end{equation*}
where the Cauchy operators $\mathcal{C}^\pm_\mathbb{R}$ are defined by
\begin{equation*}
    \mathcal{C}^\pm_{\mathbb{R}}g(x)=\mathcal{C}^\pm_\mathbb{R}\left[g\right](x)=\lim_{\epsilon\rightarrow0^+}\frac{1}{2\pi i}\int_\mathbb{R}\frac{g(s)}{s-(x\pm i\epsilon)}\mathrm{d}s.
\end{equation*}
It is known that if $g\in L^2(\mathbb{R},\mathbb{C})$ then this limit exists almost everywhere and is an $L^2(\mathbb{R},\mathbb{C})$ function that satisfies $\|\mathcal{C}^\pm_\mathbb{R}g\|_{L^2(\mathbb{R},\mathbb{C})}\leq\|g\|_{L^2(\mathbb{R},\mathbb{C})}$, see~\cite{deift_orthogonal_2000,titchmarsh_e_c_introduction_1948}. With this ansatz, the jump condition in RHP~\ref{rhp:3} becomes
\begin{align}\label{eq:DiracSingInt1}
    \mathcal{C}^+_\mathbb{R}\left[\mathbf{U}(x;\diamond)\right](\lambda)-\mathcal{C}^-_\mathbb{R}\left[\mathbf{U}(x;\diamond)\right](\lambda)\mathbf{G}(x;\lambda)=\mathbf{G}(x;\lambda)-\mathbb{I}.
\end{align}
Factoring $\mathbf{G}(x;\lambda)$ into upper- and lower-triangular parts,
\begin{equation*}
    \mathbf{G}(x;\lambda)=\begin{bmatrix}
        1-\rho_1(\lambda)\rho_2(\lambda) & -\rho_2(\lambda)e^{-2i\lambda x}\\
        \rho_1(\lambda)e^{2i\lambda x} & 1
    \end{bmatrix}=\begin{bmatrix}
        1 & -\rho_2(\lambda)e^{-2i\lambda x}\\
        0 & 1
    \end{bmatrix}\begin{bmatrix}
        1 & 0\\
        \rho_1(\lambda)e^{2i\lambda x} & 1
    \end{bmatrix}=\mathbf{G}_u(x;\lambda)\mathbf{G}_l(x;\lambda),
\end{equation*}
gives
\begin{equation*}
    \mathcal{C}^+_\mathbb{R}\left[\mathbf{U}(x;\diamond)\right](\lambda)\mathbf{G}_l^{-1}(x;\lambda)-\mathcal{C}^-_\mathbb{R}\left[\mathbf{U}(x;\diamond)\right](\lambda)\mathbf{G}_u(x;\lambda)=\mathbf{G}_u(x;\lambda)-\mathbf{G}_l^{-1}(x;\lambda).
\end{equation*}

Applying the fact that $\mathcal{C}^+_\mathbb{R}-\mathcal{C}^-_{\mathbb{R}}=\mathbb{I}$ for $L^2(\mathbb{R},\mathbb{C})$ functions gives
\begin{equation*}
\begin{aligned}
\mathbf{U}(x;\lambda)
&+ \mathcal{C}^+_\mathbb{R}\!\left[\mathbf{U}(x;\diamond)\right](\lambda)
\begin{bmatrix}
0 & 0\\
-\rho_1(\lambda)e^{2i\lambda x} & 0
\end{bmatrix} - \mathcal{C}^-_\mathbb{R}\!\left[\mathbf{U}(x;\diamond)\right](\lambda)
\begin{bmatrix}
0 & -\rho_2(\lambda)e^{-2i\lambda x}\\
0 & 0
\end{bmatrix} \\
&=
\begin{bmatrix}
0 & -\rho_2(\lambda)e^{-2i\lambda x}\\
\rho_1(\lambda)e^{2i\lambda x} & 0
\end{bmatrix}.
\end{aligned}
\end{equation*}
This results in the following two systems of singular integral equations
\begin{align*}
    &\begin{cases}
    U_{11}(x;\lambda) + \mathcal{C}^+_\mathbb{R}\left[U_{12}(x;\diamond)\right](\lambda)\left(-\rho_1(\lambda)e^{2i\lambda x}\right)=0,\\
    U_{12}(x;\lambda) + \mathcal{C}^-_\mathbb{R}\left[U_{11}(x;\diamond)\right](\lambda)\left(\rho_2(\lambda)e^{-2i\lambda x}\right)=-\rho_2(\lambda)e^{-2i\lambda x},
    \end{cases}\\
    &\begin{cases}
    U_{21}(x;\lambda)+\mathcal{C}^+_\mathbb{R}\left[U_{22}(x;\diamond)\right](\lambda)\left(-\rho_1(\lambda)e^{2i\lambda x}\right)=\rho_1(\lambda)e^{2i\lambda x},\\
    U_{22}(x;\lambda)+\mathcal{C}^-_\mathbb{R}\left[U_{21}(x;\diamond)\right](\lambda)\left(\rho_2(\lambda)e^{-2i\lambda x}\right)=0.
    \end{cases}
\end{align*}
The first two equations represent a system that can be solved for the first row of $\mathbf{U}(x;\lambda)$, $\mathbf{u}_1(x;\lambda)$, while the last two equations represent a system that can be solved for the second row of $\mathbf{U}(x;\lambda)$, $\mathbf{u}_2(x;\lambda)$. Define the operator $\mathcal{B}$ by
\begin{equation*}
    \mathcal{B}\mathbf{u}:= \mathbf{u}(x;\lambda)+\begin{bmatrix}
        \mathcal{C}^+_\mathbb{R}\left[u^{(2)}(x;\diamond)\right](\lambda)\left(-\rho_1(\lambda)e^{2i\lambda x}\right) & \mathcal{C}^-_{\mathbb{R}}\left[u^{(1)}(x;\diamond)\right](\lambda)\left(\rho_2(\lambda)e^{-2i\lambda x}\right)
    \end{bmatrix}.
\end{equation*}
To compute $\mathbf{U}(x;\lambda)$, we use the infinite-dimensional \gls{gmres} to solve
\begin{equation*}
    \mathcal{B}\mathbf{u}_1(x;\lambda)=\begin{bmatrix}
        0 & -\rho_2(\lambda)e^{-2i\lambda x}
    \end{bmatrix}\quad\text{and}\quad\mathcal{B}\mathbf{u}_2(x;\lambda)=\begin{bmatrix}
        \rho_1(\lambda)e^{2i\lambda x} & 0
    \end{bmatrix},
\end{equation*}
for $\mathbf{u}_1(x;\lambda)$ and $\mathbf{u}_2(x;\lambda)$, respectively. The oscillatory functions $-\rho_1(\lambda)e^{2i\lambda x}$ and $-\rho_2(\lambda)e^{-2i\lambda x}$ are projected onto the oscillatory rational basis functions~\cite{trogdon_application_2015}, $R_{j,\alpha}(\lambda)$, described in Appendix~\ref{sec:RationalBasis}.  We use the so-called infinite-dimensional \gls{gmres} framework, in which the operator equation is interpreted on the coefficient space induced by the oscillatory rational basis $\{R_{j,\alpha}\}$. Functions are expanded as infinite series in the basis, and \gls{gmres} is applied formally in this infinite-dimensional sequence space, with all computations carried out on truncated expansions. 
Inner products of such functions are readily computed. After each operation in the iteration, coefficients with magnitude below $10^{-15}$ are discarded (``chopped'').  The output of \gls{gmres}, at a given value of $x$, is an approximation of the form
\begin{equation*}
    \mathbf{u}(x;\lambda)\approx\sum_{\alpha'}\sum_{j=-N}^Nd_j(\alpha')R_{j,\alpha'}(\lambda),
\end{equation*}
where the outer sum is taken over a finite collection of parameters $\{\alpha'\}$. Recall that 
\begin{align*}
    \mathbf{T}^+(x;\lambda)&=\mathcal{C}^\pm_\mathbb{R}\left[\mathbf{U}(x;\diamond)\right](\lambda)+\mathbb{I}=\begin{bmatrix}
        \frac{\mathbf{m}_{1,-}(x;\lambda)}{a(\lambda)} & \mathbf{m}_{2,+}(x;\lambda)
    \end{bmatrix},\\
    \mathbf{T}^-(x;\lambda)&=\mathcal{C}^-_\mathbb{R}\left[\mathbf{U}(x;\diamond)\right](\lambda)+\mathbb{I}=\begin{bmatrix}
        \mathbf{m}_{2,-}(x;\lambda) & \frac{\mathbf{m}_{1,+}(x;\lambda)}{A(\lambda)}
    \end{bmatrix}.
\end{align*}
The original function $\mathbf{f}(x)$ can then be recovered via
\begin{equation}\label{eq:DiracCompInvTrans}
    \mathbf{f}(x)=-\frac{\sigma_3}{2\pi}\int_{-\infty}^\infty \left(c_+(\lambda)\mathbf{T}^+_2(x;\lambda)+c_-(\lambda)\mathbf{T}^-_1(x;\lambda)\right)\mathrm{d}\lambda,
\end{equation}
where $\mathbf{T}^\pm_n$ denotes the $n$-th column of $\mathbf{T}^\pm$. 

We now approximate the integral~\eqref{eq:DiracCompInvTrans}. Our approach is to expand the integrand in the oscillatory rational basis $\{R_{j,\alpha}\}$, exploit that this basis is closed under the action of Cauchy operators and multiplication, and then evaluate the resulting integral using explicit formulas for the integrals of the basis functions (see Appendix~\ref{sec:RationalBasis} for details).

The integrand of~\eqref{eq:DiracCompInvTrans} is approximated as a linear combination of basis functions:
\begin{equation*}
    c_+(\lambda)\mathbf{T}^+_2(x;\lambda)+c_-(\lambda)\mathbf{T}^-_1(x;\lambda)\approx\sum_{\alpha',j}c_j(\alpha')R_{j,\alpha'}(\lambda).
\end{equation*}
Then $\mathbf{f}(x)$ is computed via
\begin{align*}
    \mathbf{f}(x)&\approx\int_{-\infty}^\infty\sum_{\alpha',j}c_j(\alpha')R_{j,\alpha'}(\lambda)\mathrm{d}\lambda=\sum_{\alpha',j}c_j(\alpha')\int_{-\infty}^\infty R_{j,\alpha'}(\lambda)\mathrm{d}\lambda\\
    &=\sum_{\alpha',j}c_j(\alpha')\cdot\begin{cases}
        0,\quad &\mathrm{sign}(j)=\mathrm{sign}(\alpha'),\\
        -2\pi|j|,\quad &\alpha'=0,\\
        -4\pi e^{-|\alpha'|}L_{|j|-1}^{(1)}(2|\alpha'|),\quad &\mathrm{otherwise}.
    \end{cases}
\end{align*}

\subsubsection{Accounting for Poles}\label{sec:PolesComputation}

We describe a method to construct the discrete scattering data. To find the zeros of $a(\lambda)$, denoted $z_{j,+}$, $j=1,\dots,n$, we first map the real line to the unit circle via
\begin{equation*}
    z=\frac{\lambda -i}{\lambda + i},\quad \lambda=-i\frac{z+1}{z-1},
\end{equation*}
and parameterize the unit circle by $z=e^{i\theta}$, $\theta\in[0,2\pi)$.
Sampling $a(\lambda(\theta))$ at equispaced values of $\theta$, we compute its Fourier coefficients using the \gls{fft}, yielding the approximation
\begin{equation*}
    a(\lambda(\theta))\approx\sum_{k=-N}^Nc_ke^{ik\theta}=\sum_{k=-N}^Nc_kz^k.
\end{equation*}
From these coefficients, we extract a polynomial approximation in $z$, whose coefficients define a companion matrix. In practice, we retain only the coefficients corresponding to nonnegative Fourier modes to construct the polynomial. The eigenvalues of this matrix give the roots in the mapped $z$-plane, and applying the inverse transformation yields the zeros, $z_{j,+}$, of $a(\lambda)$. 
An analogous procedure applies to $A(\lambda)$ to find $z_{j,-}=\overline{z_{j,+}}$, $j=1,\dots,n$. 

To compute derivatives $a'(z_{j,+})$ and $A'(z_{j,-})$, we differentiate the truncated expansions of $a(\lambda)$ and $A(\lambda)$ in the oscillatory rational basis. Using the recurrence relation for the derivatives of the basis functions (see Appendix~\ref{sec:Appendix_DiffMult}), the derivative of a basis function $R_{j,\alpha}(\lambda)$ is expressed as a linear combination of $R_{j+1}(\lambda)$, $R_j(\lambda)$, and $R_{j-1}(\lambda)$. This structure allows for efficient evaluation of $a'(\lambda)$ and $A'(\lambda)$ at the discrete eigenvalues using Horner's method.

In exact arithmetic, $b_{j,\pm}$ are determined from the proportionality relations
\begin{equation*}
    \mathbf{m}_{1,-}(0;z_{j,+})=b_{j,+}\mathbf{m}_{2,+}(0;z_{j,+})\quad\text{and}\quad \mathbf{m}_{1,+}(0;z_{j,-})=b_{j,-}\mathbf{m}_{2,-}(0;z_{j,-}),
\end{equation*}
so that $b_{j,\pm}$ could in principle be computed by dividing one component of the vectors by the corresponding component of the other. However, due to numerical error, these vectors are not exactly proportional in finite precision arithmetic. We therefore compute $b_{j,\pm}$ as the least-squares solution to
\[
\min_{b \in \mathbb{C}} \| \mathbf{m}_{1,\mp}(0;z_{j,\pm}) - b\,\mathbf{m}_{2,\pm}(0;z_{j,\pm}) \|_2^2,
\]
which yields
\begin{equation*}
    b_{j,+}=\frac{\mathbf{m}_{1,-}(0;z_{j,+})\cdot\mathbf{m}_{2,+}(0;z_{j,+})}{\mathbf{m}_{2,+}(0;z_{j,+})\cdot\mathbf{m}_{2,+}(0;z_{j,+})},\quad
    b_{j,-}=\frac{\mathbf{m}_{1,+}(0;z_{j,-})\cdot\mathbf{m}_{2,-}(0;z_{j,-})}{\mathbf{m}_{2,-}(0;z_{j,-})\cdot\mathbf{m}_{2,-}(0;z_{j,-})},
\end{equation*}
where $\cdot$ denotes the Euclidean inner product.
We evaluate $\mathbf{m}_{\ell,\pm}(x;\lambda)$, $\ell\in\{1,2\}$, at $x=0$ because $b_{j,\pm}$ are constant in $x$. The norming constants $c_{j,\pm}$ follow from~\eqref{eq:normingConstants}. 

To illustrate pole contributions, we consider RHP~\ref{rhp:5}, following the formulation in~\cite{trogdon_scattering_2021}, see also \cite{ablowitz_solitons_1991}.

\begin{rhp}\label{rhp:5}
    Find $\mathbf{Y}(x;\diamond):\mathbb{R}\rightarrow\mathbb{C}^{2\times 2}$ and $\mathbf{Y}_{j,\pm}(x)\in\mathbb{C}^{2\times 2}$, $j=1,2,...,n$, such that
    \begin{align*}
        \mathbf{T}(x;\lambda)=\mathbb{I}+&\frac{1}{2\pi i}\int_{-\infty}^\infty\frac{\mathbf{Y}(x;\lambda')}{\lambda'-\lambda}\mathrm{d}\lambda' + \sum_{j=1}^{n}\frac{\mathbf{Y}_{j,+}(x)}{\lambda-z_{j,+}} + \sum_{j=1}^{n}\frac{\mathbf{Y}_{j,-}(x)}{\lambda-z_{j,-}},\\
        \mathbf{T}^+(x;\lambda)&=\mathbf{T}^-(x;\lambda)\begin{bmatrix}
            1-\rho_1(\lambda)\rho_2(\lambda) & -\rho_2(\lambda)e^{-2i\lambda x}\\
            \rho_1(\lambda)e^{2i\lambda x} & 1
        \end{bmatrix},
    \end{align*}
    and
    \begin{align*}
        \mathrm{Res}_{\lambda=z_{j,+}}\mathbf{T}(x;\lambda)=\lim_{\lambda\rightarrow z_{j,+}}\mathbf{T}(x;\lambda)\begin{bmatrix}
            0 & 0\\
            c_{j,+}e^{2iz_{j,+}x} & 0
        \end{bmatrix},\\
        \mathrm{Res}_{\lambda=z_{j,-}}\mathbf{T}(x;\lambda)=\lim_{\lambda\rightarrow z_{j,-}}\mathbf{T}(x;\lambda)\begin{bmatrix}
            0 & c_{j,-}e^{-2iz_{j,-}x}\\
            0 & 0
        \end{bmatrix}.
    \end{align*}
\end{rhp}

Now, with $\rho_1=\rho_2=0$, we have RHP~\ref{rhp:6}.

\begin{rhp}\label{rhp:6}
    Find $\mathbf{Y}_{j,\pm}(x)\in\mathbb{C}^{2\times 2}$, $j=1,2,...,n$ such that
    \begin{equation*}
        \mathbf{T}_d(x;\lambda)=\mathbb{I}+\sum_{j=1}^{n}\frac{\mathbf{Y}_{j,+}(x)}{\lambda-z_{j,+}}+\sum_{j=1}^{n}\frac{\mathbf{Y}_{j,-}(x)}{\lambda-z_{j,-}},
    \end{equation*}
    satisfies
    \begin{equation*}
        \mathrm{Res}_{\lambda=z_{j,+}}\mathbf{T}_d(x;\lambda)=\lim_{\lambda\rightarrow z_{j,+}}\mathbf{T}_d(x;\lambda)\begin{bmatrix}
            0 & 0\\
            c_{j,+}e^{2iz_{j,+}x} & 0
        \end{bmatrix},
    \end{equation*}
    and
    \begin{equation*}
        \mathrm{Res}_{\lambda=z_{j,-}}\mathbf{T}_d(x;\lambda)=\lim_{\lambda\rightarrow z_{j,-}}\mathbf{T}_d(x;\lambda)\begin{bmatrix}
            0 & c_{j,-}e^{-2iz_{j,-}x}\\
            0 & 0
        \end{bmatrix}.
    \end{equation*}
\end{rhp}
The matrices $\mathbf{Y}_{j,\pm}(x)\in\mathbb{C}^{2\times2}$ can be obtained row-by-row. The residue conditions imply that
\begin{equation*}
    \mathbf{Y}_{j,+}(x)=\begin{bmatrix}
        Y_{j,+}^{(11)}(x) & 0\\
        Y_{j,+}^{(21)}(x) & 0
    \end{bmatrix},\quad \mathbf{Y}_{j,-}(x)=\begin{bmatrix}
        0 & Y_{j,-}^{(12)}(x)\\
        0 & Y_{j,-}^{(22)}(x)
    \end{bmatrix}.
\end{equation*}
Define $\mathbf{Z}$ entrywise as
\begin{equation*}
    Z_{jk} = \frac{1}{z_{j,+}-z_{k,-}}.
\end{equation*}
Construct diagonal matrices
\begin{align*}
    &\quad\mathbf{C}_+=\text{diag}(\mathbf{c}_+),\quad\mathbf{C}_-=\text{diag}(\mathbf{c}_-),\\
    \mathbf{c}_+=&\begin{bmatrix}
        c_{1,+}e^{2iz_{1,+}x}\\
        \vdots\\
        c_{n,+}e^{2iz_{n,+}x}
    \end{bmatrix},\quad\mathbf{c}_-=\begin{bmatrix}
        c_{1,-}e^{-2iz_{1,-}x}\\
        \vdots\\
        c_{n,-}e^{-2iz_{n,-}x}
    \end{bmatrix}.
\end{align*}
The residue conditions give the block linear system
\begin{equation*}
    \left[
    \begin{array}{c|c}
    \mathbb{I} & -\mathbf{C}_+\mathbf{Z} \\
    \hline
    \mathbf{C}_-\mathbf{Z}^T & \mathbb{I}
    \end{array}
    \right]\left[
    \begin{array}{c|c}
    \begin{array}{c}
    Y_{1,+}^{(11)}(x) \\
    \vdots\\
    Y_{n,+}^{(11)}(x)
    \end{array} &
    \begin{array}{c}
    Y_{1,+}^{(21)}(x) \\
    \vdots\\
    Y_{n,+}^{(21)}(x)
    \end{array} \\
    \hline
    \begin{array}{c}
    Y_{1,-}^{(12)}(x) \\
    \vdots\\
    Y_{n,-}^{(12)}(x)
    \end{array} &
    \begin{array}{c}
    Y_{1,-}^{(22)}(x) \\
    \vdots\\
    Y_{n,-}^{(22)}(x)
    \end{array}
    \end{array}
    \right]=\left[
    \begin{array}{c|c}
    \mathbf{0} & \mathbf{c}_+ \\
    \hline
    \mathbf{c}_- & \mathbf{0}
    \end{array}
    \right].
\end{equation*}
This reduces solving for $\mathbf{T}_d(x;\lambda)$ to pure linear algebra. 

Define $\mathbf{T}_0(x;\lambda)=\mathbf{T}(x;\lambda)\mathbf{T}_d(x;\lambda)^{-1}$. Then $\mathbf{T}_0(x;\lambda)$ satisfies an \gls{rhp} with no residue conditions, which we state in a different form.

\begin{rhp}\label{rhp:7}
    Find $\mathbf{Y}_0(x;\diamond):\mathbb{R}\rightarrow\mathbb{C}^{2\times 2}$, such that
    \begin{align*}
        \mathbf{T}_0(x;\lambda)&=\mathbb{I}+\frac{1}{2\pi i}\int_{-\infty}^\infty\frac{\mathbf{Y}_0(x;\lambda')}{\lambda'-\lambda}\mathrm{d}\lambda',\\
        \mathbf{T}_0^+(x;\lambda)&=\mathbf{T}_0^-(x;\lambda)\mathbf{T}_d(x;\lambda)\begin{bmatrix}
            1-\rho_1(\lambda)\rho_2(\lambda) & -\rho_2(\lambda)e^{-2i\lambda x}\\
            \rho_1(\lambda)e^{2i\lambda x} & 1
        \end{bmatrix}\mathbf{T}_d^{-1}(x;\lambda).
    \end{align*}
\end{rhp}
Factor as
\begin{equation*}
    \mathbf{T}_d(x;\lambda)\begin{bmatrix}
        1-\rho_1(\lambda)\rho_2(\lambda) & -\rho_2(\lambda)e^{-2i\lambda x}\\
        \rho_1(\lambda)e^{2i\lambda x} & 1
    \end{bmatrix}\mathbf{T}_d(x;\lambda)^{-1}=\mathbf{U}_d(x;\lambda)\mathbf{L}_d(x;\lambda)^{-1},
\end{equation*}
where
\begin{align*}
    \mathbf{U}_d(x;\lambda)&=\mathbf{T}_d(x;\lambda)\begin{bmatrix}
        1 & -\rho_2(\lambda)e^{-2i\lambda x}\\
        0 & 1
    \end{bmatrix}\mathbf{T}_d(x;\lambda)^{-1},\\
    \mathbf{L}_d(x;\lambda)&=\mathbf{T}_d(x;\lambda)\begin{bmatrix}
        1 & 0\\
        -\rho_1(\lambda)e^{2i\lambda x} & 1
    \end{bmatrix}\mathbf{T}_d(x;\lambda)^{-1}.
\end{align*}
The singular integral equation for $\mathbf{Y}_0(x;\lambda)$ takes the form
\begin{equation*}
    \mathcal{C}^+_\mathbb{R}\left[\mathbf{Y}_0(x;\diamond)\right](\lambda)\mathbf{L}_d(x;\lambda)-\mathcal{C}^-_{\mathbb{R}}\left[\mathbf{Y}_0(x;\diamond)\right](\lambda)\mathbf{U}_d(x;\lambda)=\mathbf{U}_d(x;\lambda)-\mathbf{L}_d(x;\lambda).
\end{equation*}

Each row of $\mathbf{Y}_0$ can be solved for independently using infinite-dimensional \gls{gmres} as described in Section~\ref{sec:DiracComputation}. We can then compute
\begin{equation*}
    \mathbf{T}_0^\pm(x;\lambda) = \mathcal{C}^\pm_{\mathbb{R}} \mathbf{Y}_0(x;\lambda) + \mathbb{I},\qquad
    \mathbf{T}(x;\lambda) = \mathbf{T}_0(x;\lambda) \mathbf{T}_d(x;\lambda).
\end{equation*}

\subsubsection{Modifications for Left Scattering}\label{sec:LeftScattering}

The method in Section~\ref{sec:DiracComputation} is efficient for right scattering ($x\geq0$). The modifications introduced here for left scattering extend this efficiency to $x<0$, as reflected in the \gls{gmres} iteration counts: for symmetric potentials, the iterations are symmetric in $x$ and decrease as $|x|$ increases (see Section~\ref{sec:Results}). We develop an analogous approach for left scattering ($x<0$), using a hat to represent the modified quantities. The left scattering formulation differs from the right scattering case in that the placement of the Jost columns in the complex $\lambda$-plane is reversed. We exploit this symmetry by conjugating with $\sigma_1$ and renormalizing by the transmission coefficients. This transformation restores the same determinant normalization and half-plane analyticity structure as in the right scattering case, allowing the computational framework of Section~\ref{sec:DiracComputation} to be applied.

We define $\hat{\mathbf{L}}(x;\lambda)$ for $x<0$ as
\begin{align*}
    \hat{\mathbf{L}}(x;\lambda) &= \begin{cases}
        \mathbf{T}^+(x;\lambda)\sigma_1\begin{bmatrix}
            \frac{1}{a(\lambda)} & 0\\
            0 & a(\lambda)
        \end{bmatrix},\quad\lambda\in\mathbb{C}^+,\\[1.25em]
    \mathbf{T}^-(x;\lambda)\sigma_1\begin{bmatrix}
        A(\lambda) & 0\\
        0 & \frac{1}{A(\lambda)}
    \end{bmatrix},\quad\lambda\in\mathbb{C}^-,
    \end{cases}\\
    &=\begin{cases}
        \begin{bmatrix}
            \frac{\mathbf{m}_{2,+}(x;\lambda)}{a(\lambda)} & \mathbf{m}_{1,-}(x;\lambda)
        \end{bmatrix},\quad\lambda\in\mathbb{C}^+,\\[0.5em]
        \begin{bmatrix}
            \mathbf{m}_{1,+}(x;\lambda) & \frac{\mathbf{m}_{2,-}(x;\lambda)}{A(\lambda)}
        \end{bmatrix},\quad\lambda\in\mathbb{C}^-.
    \end{cases}
\end{align*}
The jump relation for $\hat{\mathbf{L}}(x;\lambda)$ is
\begin{equation*}
    \hat{\mathbf{L}}^+(x;\lambda)=\hat{\mathbf{L}}^-(x;\lambda)\begin{bmatrix}
        \frac{1}{a(\lambda)A(\lambda)} & \frac{b(\lambda)}{A(\lambda)}\\
        -\frac{B(\lambda)}{a(\lambda)} & 1
    \end{bmatrix},\quad\hat{\mathbf{L}}(x;\lambda)=\left(\mathbb{I}+o(1)\right)\begin{bmatrix}
        0 & e^{-i\lambda x}\\
        e^{i\lambda x} & 0
    \end{bmatrix}\quad\text{as}\quad |\lambda|\rightarrow\infty.
\end{equation*}
The asymptotic condition follows from the large-$|\lambda|$ behavior of the Jost solutions. Arguing as above, 
\[
\mathbf m_{\ell,-}(x;\lambda)=e^{-i\lambda x}(\mathbf e_1+o(1)),\quad
\mathbf m_{\ell,+}(x;\lambda)=e^{i\lambda x}(\mathbf e_2+o(1)),\quad\ell\in\{1,2\},
\]
uniformly for $x \in \mathbb R$ and $a(\lambda),A(\lambda)= 1+\mathcal{O}(1/|\lambda|)$ as $|\lambda|\to\infty$. Thus, from the definition of $\hat{\mathbf L}$, the stated asymptotics follow.
Define $\acute{\mathbf{L}}(x;\lambda)=\sigma_1\hat{\mathbf{L}}(x;\lambda)$. The jump matrix has determinant 1. To match the asymptotic condition in RHP~\ref{rhp:3}, set
\begin{equation*}
    \hat{\mathbf{T}}(x;\lambda) = \acute{\mathbf{L}}(x;\lambda)\begin{bmatrix}
        e^{-i\lambda x} & 0\\
        0 & e^{i\lambda x}
    \end{bmatrix}.
\end{equation*}
Define $\hat{\rho}_1(\lambda)=-B(\lambda)/a(\lambda)$ and $\hat{\rho}_2(\lambda)=-b(\lambda)/A(\lambda)$. Then
\begin{align*}
    \hat{\mathbf{T}}^+(x;\lambda) &= \hat{\mathbf{T}}^-(x;\lambda)\begin{bmatrix}
        1-\hat{\rho}_1(\lambda)\hat{\rho}_2(\lambda) & -\hat{\rho}_2(\lambda)e^{2i\lambda x}\\
        \hat{\rho}_1(\lambda)e^{-2i\lambda x} & 1
    \end{bmatrix}=\hat{\mathbf{T}}(x;\lambda)\hat{\mathbf{G}}(x;\lambda),\\
    \hat{\mathbf{T}}(x;\lambda)&=\mathbb{I}+o(1)\quad\text{as}\quad|\lambda|\rightarrow\infty.
\end{align*}
From here we can repeat the procedure above from Step 2 onwards in Section~\ref{sec:DiracComputation}, replacing $\mathbf{T}^\pm(x;\lambda)$ and $\mathbf{G}(x;\lambda)$ with $\hat{\mathbf{T}}^\pm(x;\lambda)$ and $\hat{\mathbf{G}}(x;\lambda)$, respectively.

\subsubsection{Modifications for Left Scattering with Poles}\label{sec:LeftScatteringPoles}

We find that as $x\rightarrow-\infty$, for $\lambda\in\mathbb{C}^+$,
\begin{equation*}
    \text{Res}_{\lambda=z_{j,+}}\mathbf{v}(x;\lambda)=\hat{\mathbf{w}}_{j,+}(x;z_{j,+})=-\frac{1}{a'(z_{j,+})}\mathbf{m}_{1,-}(x;z_{j,+})\int_{-\infty}^\infty\begin{bmatrix}
        m_{2,+}^{(2)}(s;z_{j,+}) & -m_{2,+}^{(1)}(s;z_{j,+})
    \end{bmatrix}\mathbf{f}(s)\mathrm{d}s,
\end{equation*}
and for $\lambda\in\mathbb{C}^-$,
\begin{equation*}
    \text{Res}_{\lambda=z_{j,-}}\mathbf{v}(x;\lambda)=\hat{\mathbf{w}}_{j,-}(x;z_{j,-})=-\frac{1}{A'(z_{j,-})}\mathbf{m}_{1,+}(x;z_{j,-})\int_{-\infty}^\infty\begin{bmatrix}
        -m_{2,-}^{(2)}(s;z_{j,-}) & m_{2,-}^{(1)}(s;z_{j,-})
    \end{bmatrix}\mathbf{f}(s)\mathrm{d}s.
\end{equation*}
We modify the recovery formula in~\eqref{eq:RecoveryPoles} by replacing $\mathbf{w}_{j,\pm}$ with $\hat{\mathbf{w}}_{j,\pm}$. 

A zero of $a(\lambda)$ at $\lambda=z_{j,+}$ implies that $\mathbf{m}_{1,-}(x;z_{j,+})=b_{j,+}\mathbf{m}_{2,+}(x;z_{j,+})$ for some $b_{j,+}.$ Therefore, 
\begin{equation*}
    \mathrm{Res}_{\lambda=z_{j,+}}\frac{\mathbf{m}_{2,+}(x;\lambda)}{a(\lambda)}=\frac{\mathbf{m}_{2,+}(x;z_{j,+})}{a'(z_{j,+})}=\frac{1}{b_{j,+}a'(z_{j,+})}\mathbf{m}_{1,-}(x;z_{j,+}).
\end{equation*}
Similarly, if $\mathbf{m}_{1,+}(x; z_{j,-}) = b_{j,-} \mathbf{m}_{2,-}(x; z_{j,-})$, then
\begin{equation*}
    \mathrm{Res}_{\lambda=z_{j,-}}\frac{\mathbf{m}_{2,-}(x;\lambda)}{A(\lambda)}=\frac{\mathbf{m}_{2,-}(x;z_{j,-})}{A'(z_{j,-})}=\frac{1}{b_{j,-}A'(z_{j,-})}\mathbf{m}_{1,+}(x;z_{j,-}).
\end{equation*}
Define
\begin{equation*}
    \hat{c}_{j,+}=\frac{1}{b_{j,+}a'(z_{j,+})},\quad \hat{c}_{j,-}=\frac{1}{b_{j,-}a'(z_{j,-})}.
\end{equation*}
We then repeat the workflow presented in Section~\ref{sec:DiracPoles}, 
replacing any quantities with their corresponding hat quantities.

\subsection{Results}\label{sec:Results}

We now demonstrate the computation of the Dirac transform pair. For each example, we compute the forward transform, reconstruct the solution via the inverse transform, and report the associated error and \gls{gmres} iteration counts. Unless otherwise noted, all computations achieve near machine precision. Code used to produce all figures in this paper can be found at~\cite{Code}.

\subsubsection{Example 1: Gaussian potential, Gaussian data, and \texorpdfstring{$\tau=1$}{tau=1}}\label{sec:Ex1}

Let
\begin{equation}\label{eq:Ex1}
    \mathbf{f}(x)=\begin{bmatrix} e^{-x^2} \\ e^{-x^2} \end{bmatrix},\quad q(x)=e^{-x^2},\quad \tau=1.
\end{equation}
Observe that the forward transforms, $c_+(\lambda)$ and $c_-(\lambda)$, and the max absolute error of $c_+(\lambda)$ with this choice of $\mathbf{f}$ and $q$ have already been plotted in Figures~\ref{fig:DiracForward} and~\ref{fig:cp_err}. This choice of $\tau$ does not produce a discrete spectrum.  
Figure~\ref{fig:DiracInverse} shows the recovered function and the absolute difference from the original.  
Since both components of $\mathbf{f}(x)$ are identical, we plot only one component.
\begin{figure}[htbp]
    \centering
    \begin{subfigure}[b]{0.49\textwidth}
        \centering
        \includegraphics[width=\textwidth]{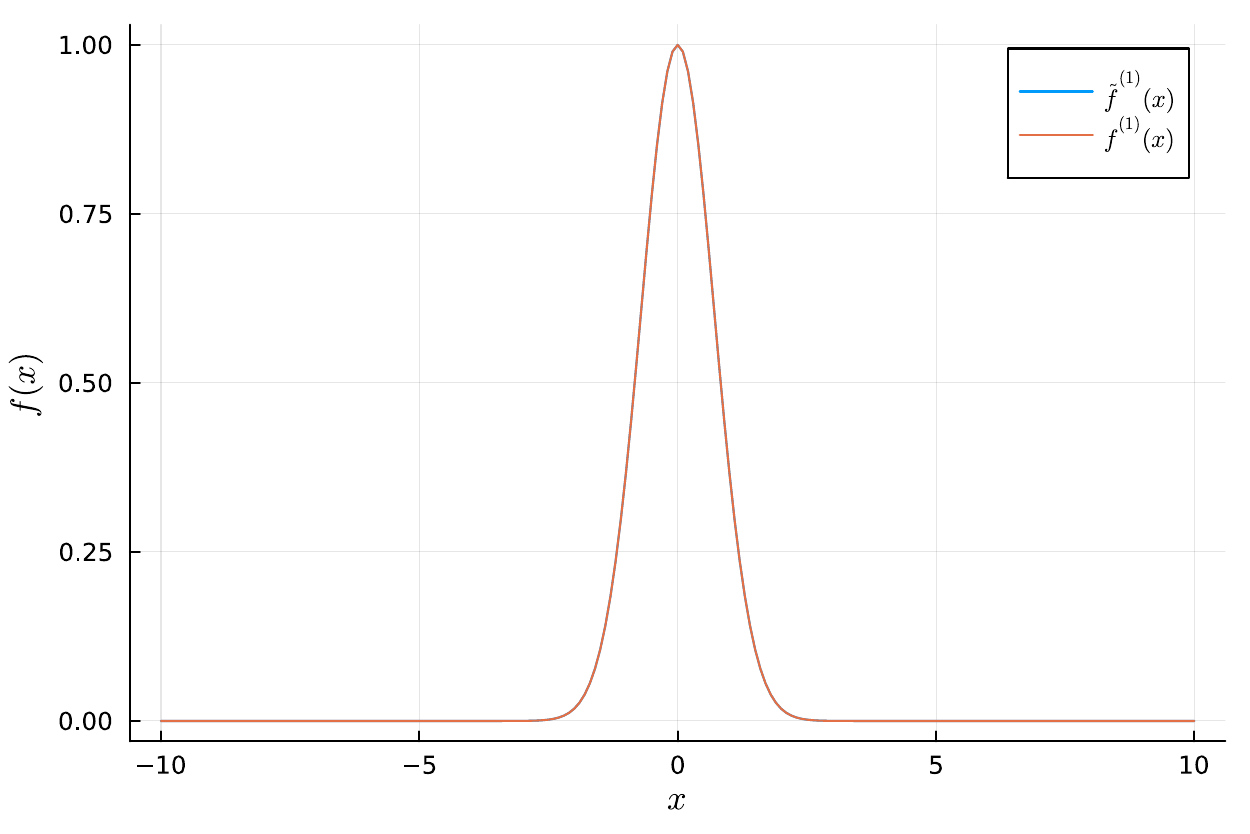}
        \caption{Inverse transform}
        \label{fig:DiracInversePlot1}
    \end{subfigure}
    \hfill
    \begin{subfigure}[b]{0.49\textwidth}
        \centering
        \includegraphics[width=\textwidth]{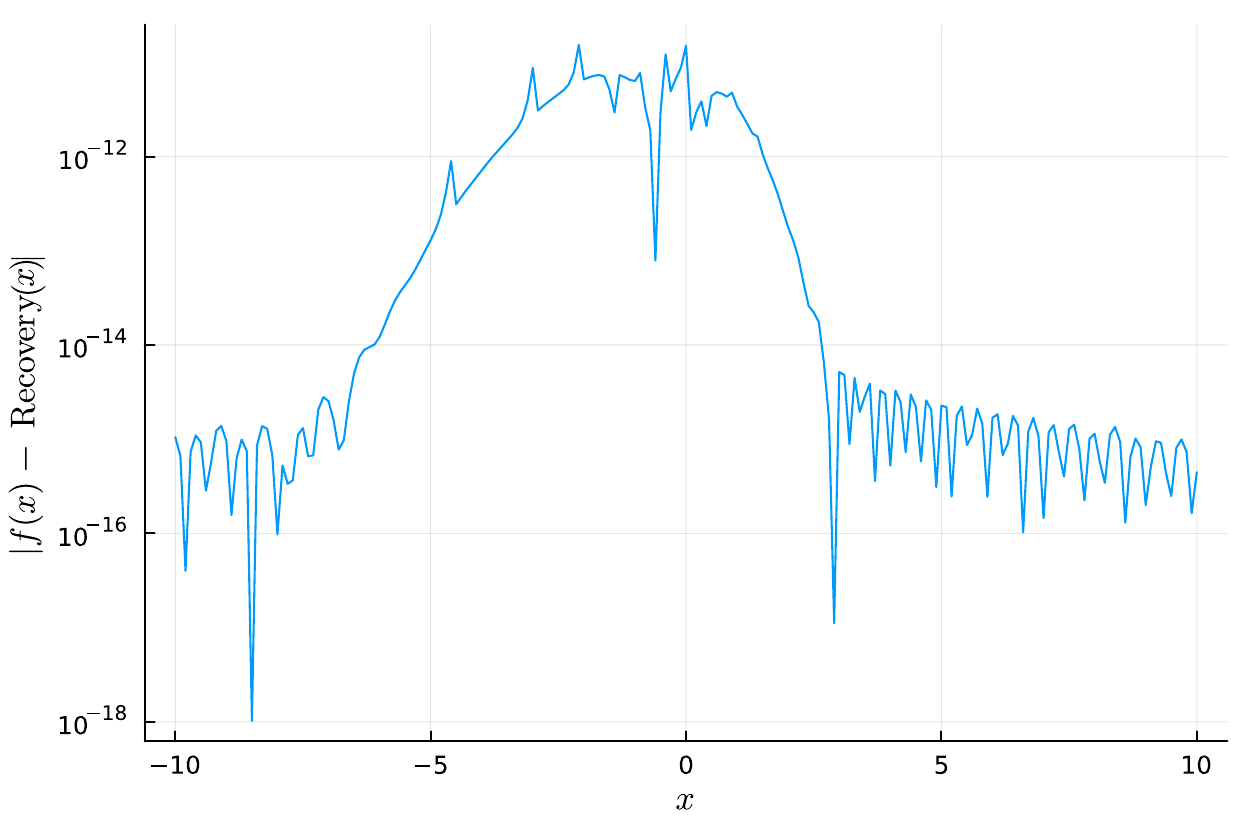}
        \caption{Error}
        \label{fig:DiracError}
    \end{subfigure}
    \caption{(a) Recovered first component of $\mathbf{f}(x)$, $\tilde{f}^{(1)}(x)$, and the original $f^{(1)}(x)$ with data and potential given by~\eqref{eq:Ex1}. (b) Absolute difference between the recovery and the original. The \glspl{ode} for $\mathbf{m}_{\ell,\pm}(x;\lambda)$, $\ell\in\{1,2\}$, were solved using 300 collocation nodes in $x$, while those for $\mathbf{p}_\pm(x;\lambda)$ used 500 nodes. The reflection coefficients were represented using 270 coefficients, $c_\pm(\lambda)$ were expanded with 498 coefficients, and $\mathbf{m}_{\ell,\pm}(x;\lambda)$ were expanded using 232 coefficients.}
    \label{fig:DiracInverse}
\end{figure}
The reconstruction (Figure~\ref{fig:DiracInverse}) achieves a worst-case error of approximately $10^{-11}$, improving as $|x|$ increases due to the $e^{-\alpha}$ factor in our Cauchy integral (Appendix~\ref{sec:RationalBasis}). Figure~\ref{fig:DiracGMRES} shows the number of \gls{gmres} iterations needed to reach a residual below $10^{-10}$.
\begin{figure}
    \centering
    \includegraphics[width=0.5\linewidth]{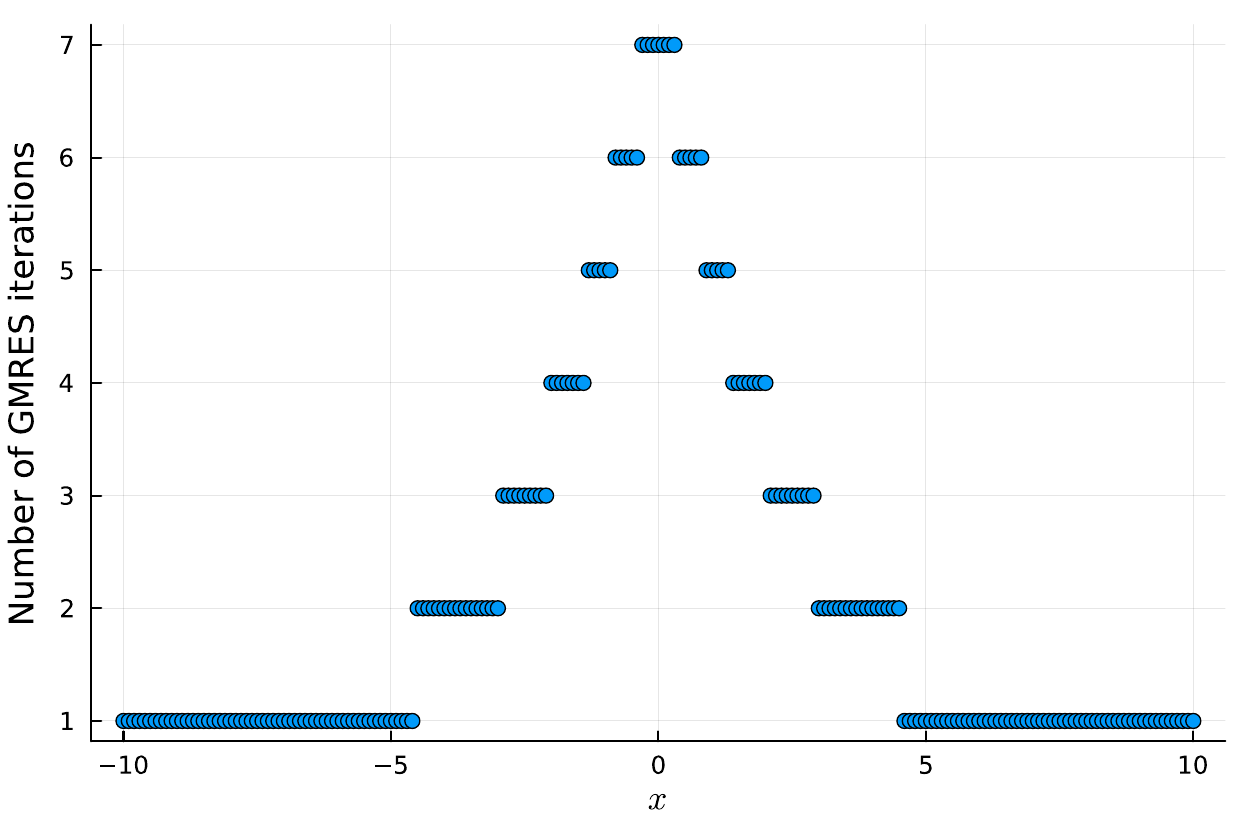}
    \caption{Number of \gls{gmres} iterations required to achieve a residual less than $10^{-10}$ for data and potential given by~\eqref{eq:Ex1}.}
    \label{fig:DiracGMRES}
\end{figure}
The required \gls{gmres} iterations are small and decrease as $|x|$ increases, so both accuracy and speed improve with $|x|$. 

\subsubsection{Example 2: Gaussian potential, Gaussian data, and \texorpdfstring{$\tau=1$}{tau=-1}}

Let 
\begin{equation}\label{eq:Ex2}
    \mathbf{f}(x)=\begin{bmatrix} e^{-x^2} \\ e^{-x^2} \end{bmatrix},\quad q(x)=e^{-x^2},\quad \tau=-1.
\end{equation} 
This potential produces the eigenvalues
\begin{equation*}
    z_{1,\pm}\approx \pm 0.13331628147293i,
\end{equation*}
with norming constants
\begin{equation*}
    c_{1,\pm}=\hat{c}_{1,\pm}=-0.77741091603772i.
\end{equation*}

\begin{figure}[htbp]
    \centering
    \begin{subfigure}[b]{0.49\textwidth}
        \centering
        \includegraphics[width=\textwidth]{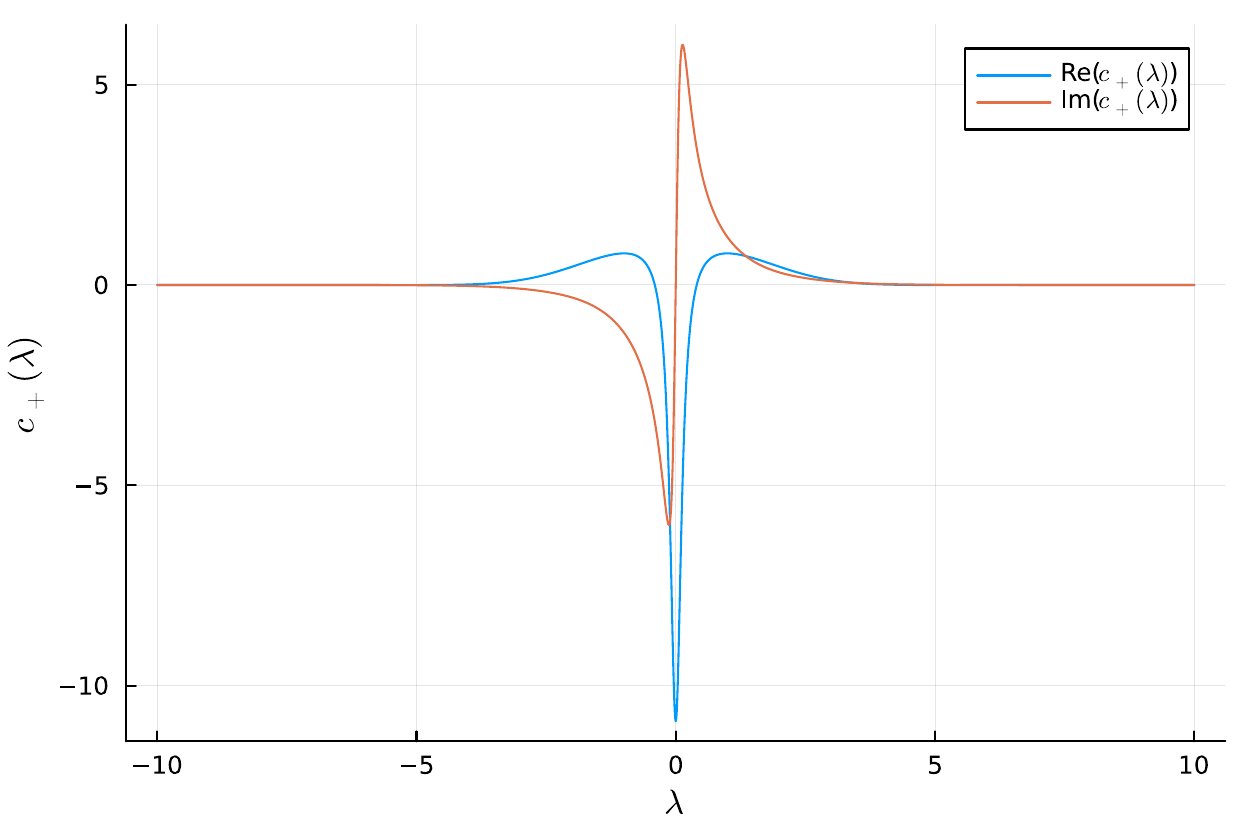}
        \caption{$c_+(\lambda)$}
    \end{subfigure}
    \hfill
    \begin{subfigure}[b]{0.49\textwidth}
        \centering
        \includegraphics[width=\textwidth]{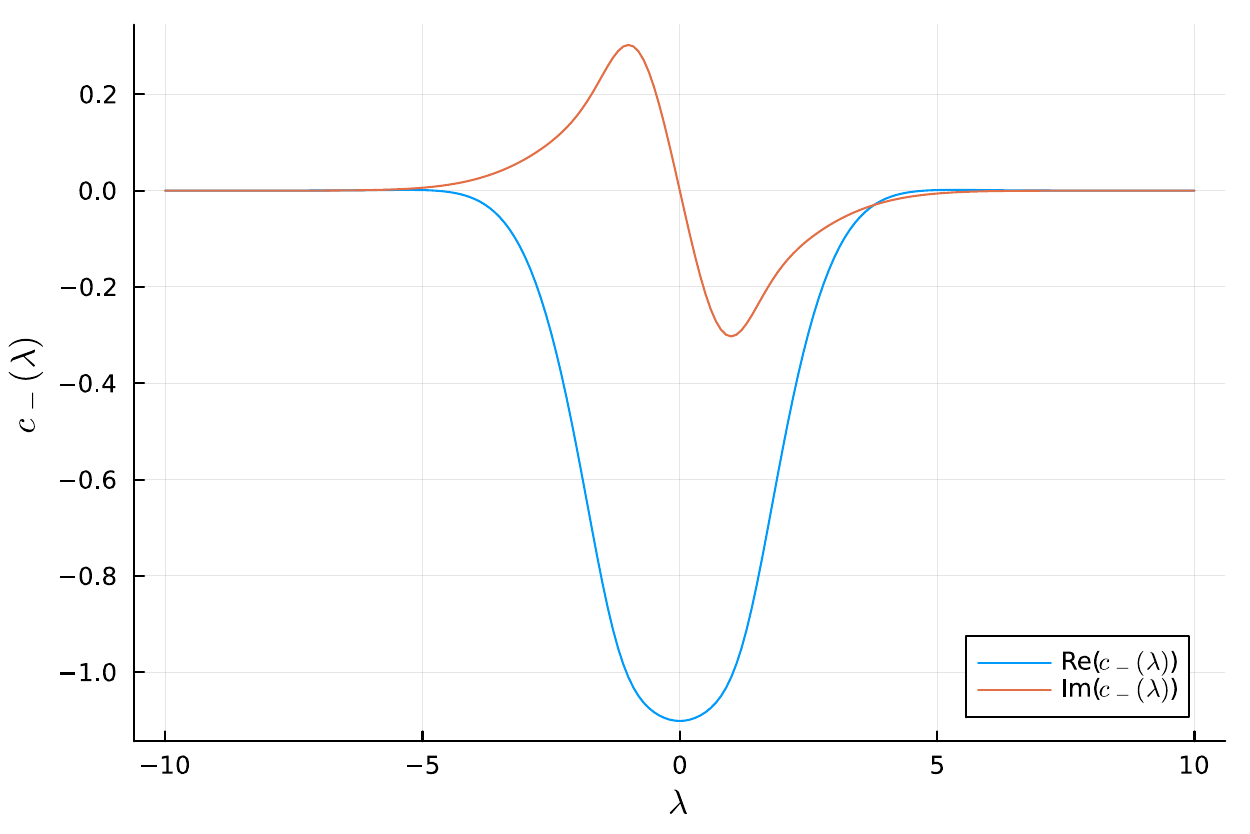}
        \caption{$c_-(\lambda)$}
    \end{subfigure}
    \caption{Forward transforms $c_+(\lambda)$ and $c_-(\lambda)$ with data and potential functions given by~\eqref{eq:Ex2} computed using the \gls{uclm}.}
    \label{fig:PolesForward}
\end{figure}
\begin{figure}[htbp]
    \centering
    \begin{subfigure}[b]{0.49\textwidth}
        \centering
        \includegraphics[width=\textwidth]{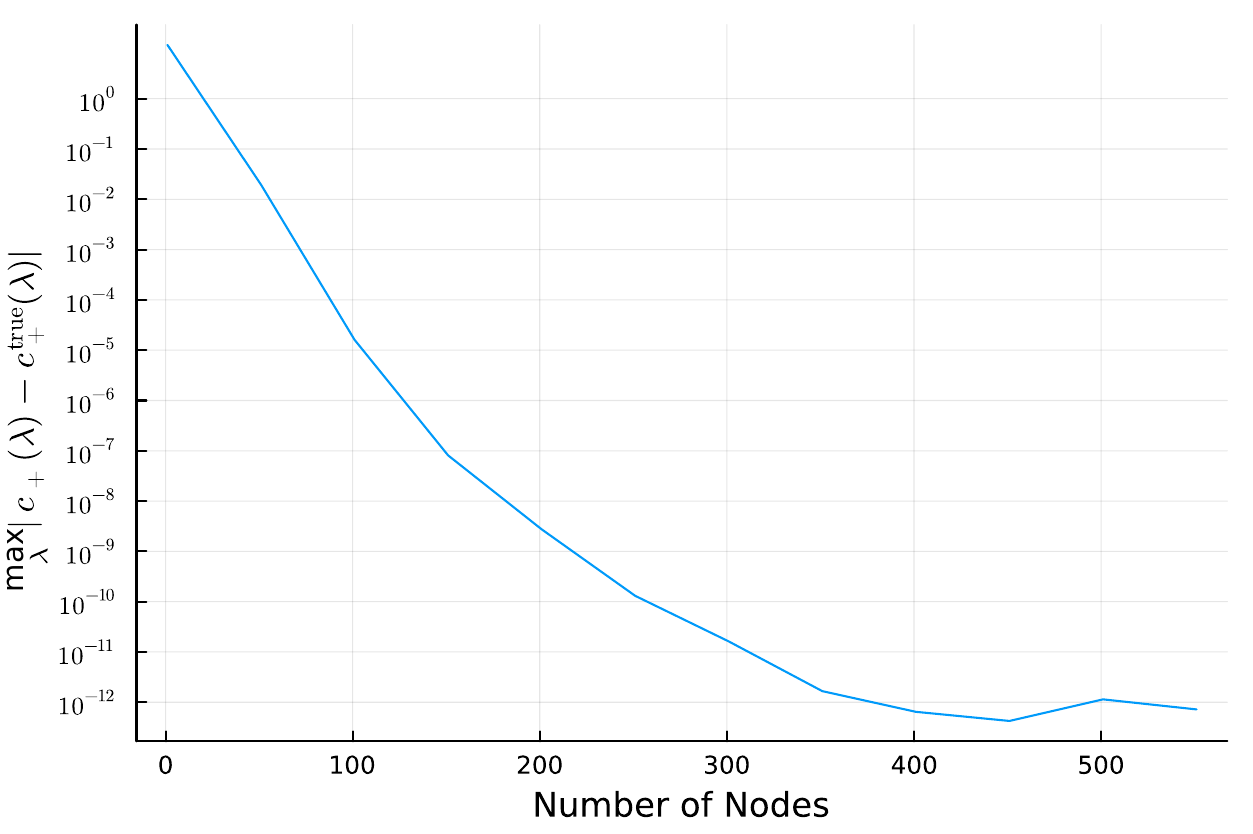}
        \caption{$c_+(\lambda)$}
    \end{subfigure}
    \hfill
    \begin{subfigure}[b]{0.49\textwidth}
        \centering
        \includegraphics[width=\textwidth]{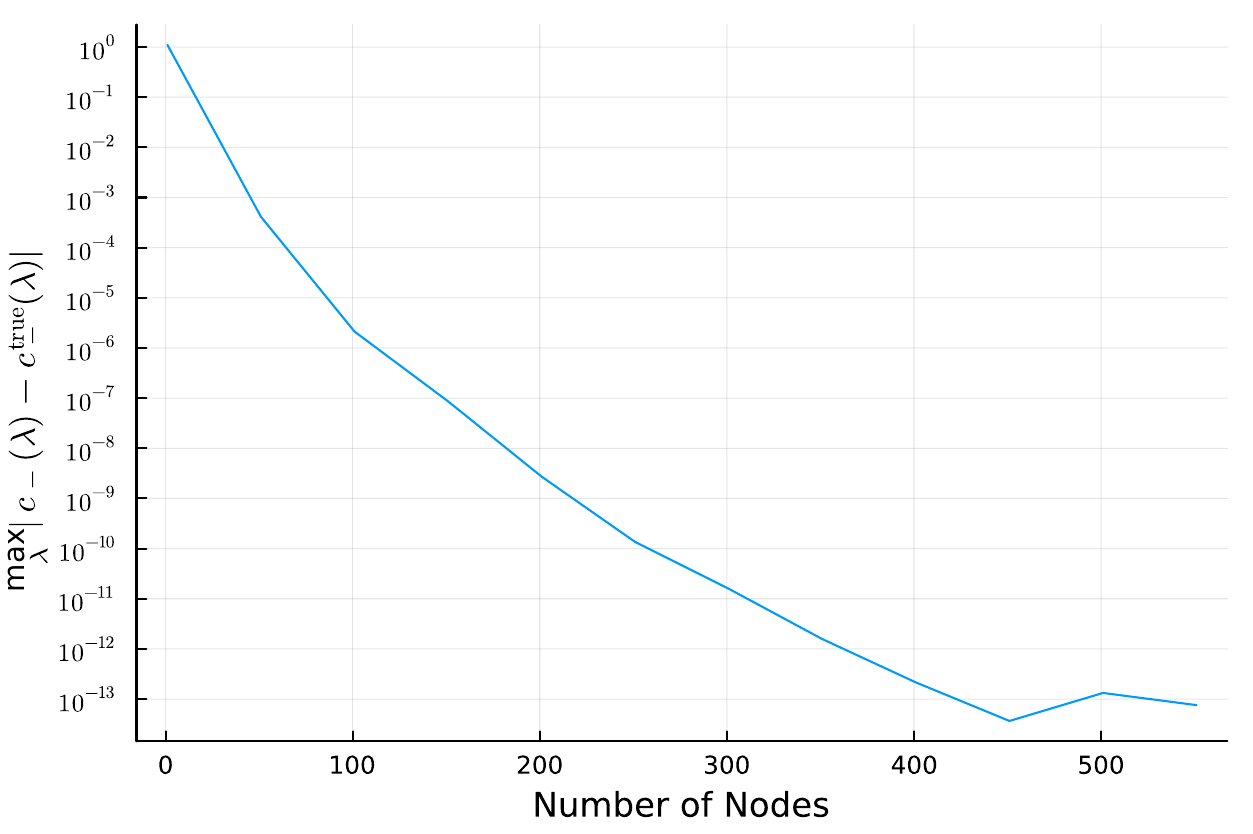}
        \caption{$c_-(\lambda)$}
    \end{subfigure}
    \caption{Max absolute error between $c_\pm(\lambda)$ for data and potential given by~\eqref{eq:Ex2} and highly resolved solution evaluated on a uniform grid in $\lambda \in [-10,10]$ as a function of the number of nodes used.}
    \label{fig:TBD}
\end{figure}

\begin{figure}[htbp]
    \centering
    \begin{subfigure}[b]{0.49\textwidth}
        \centering
        \includegraphics[width=\textwidth]{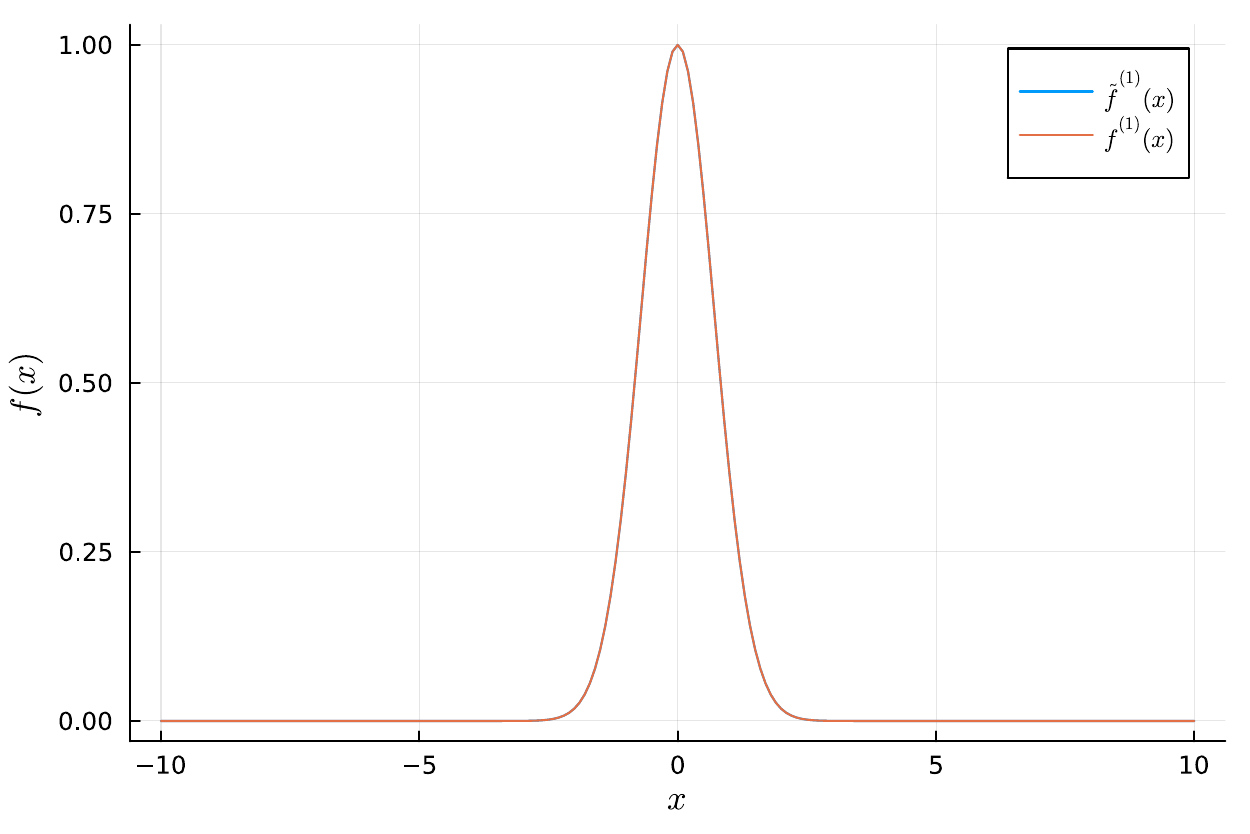}
        \caption{Inverse transform}
        \label{fig:DiracInversePlot}
    \end{subfigure}
    \hfill
    \begin{subfigure}[b]{0.49\textwidth}
        \centering
        \includegraphics[width=\textwidth]{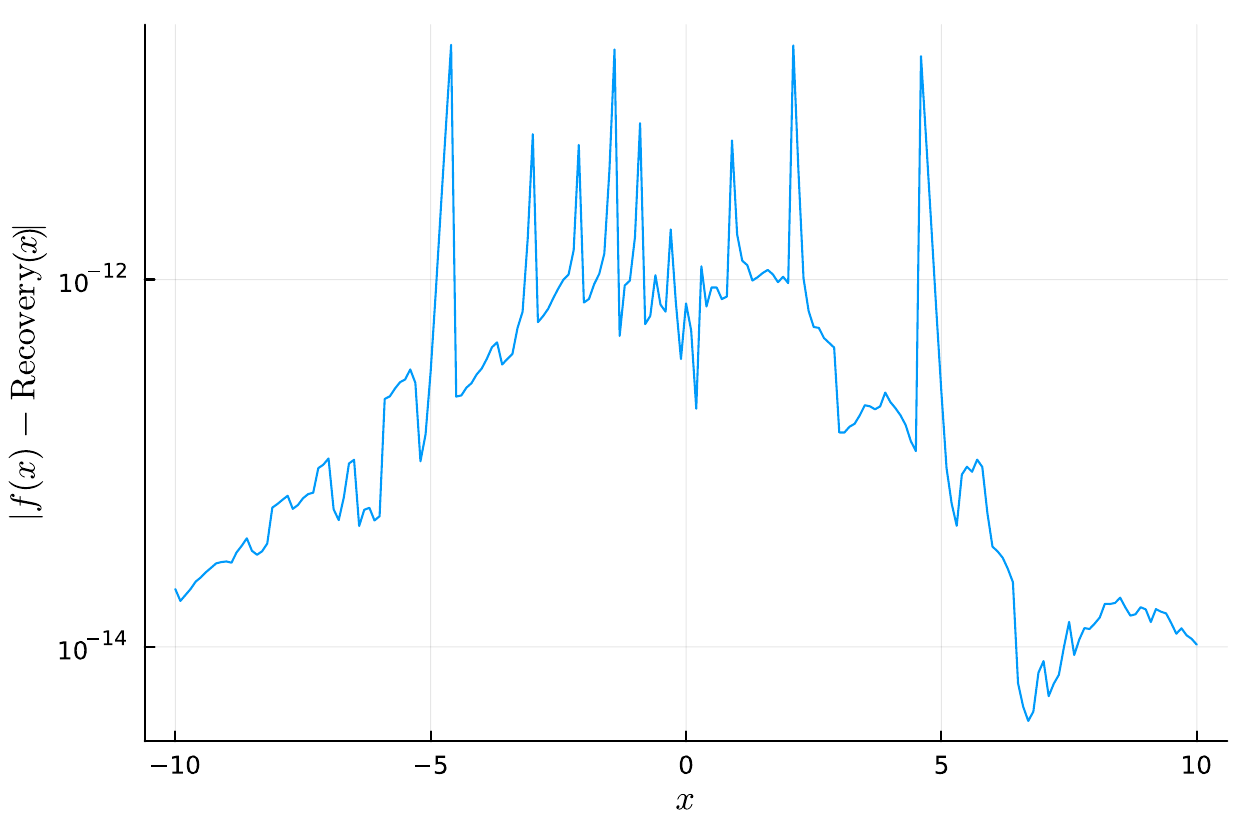}
        \caption{Error}
        \label{fig:DiracError1}
    \end{subfigure}
    \caption{(a) Recovered first component of $\mathbf{f}(x)$, $\tilde{f}^{(1)}(x)$, and the original $f^{(1)}(x)$ with data and potential given by~\eqref{eq:Ex2}. (b) Absolute difference between the recovery and the original. The \glspl{ode} for $\mathbf{m}_{\ell,\pm}(x;\lambda)$, $\ell\in\{1,2\}$, were solved using 300 collocation nodes in $x$, while those for $\mathbf{p}_\pm(x;\lambda)$ used 500 nodes. The reflection coefficients were represented using 292 coefficients, $c_\pm(\lambda)$ were expanded with 546 coefficients, and $\mathbf{m}_{\ell,\pm}(x;\lambda)$ were expanded using 330 coefficients.}
    \label{fig:DiracInversePoles}
\end{figure}

Despite the presence of discrete spectrum, the reconstruction maintains the same level of accuracy as in Example~\ref{sec:Ex1}. The \gls{gmres} iteration counts exhibit similar behavior, indicating that the inclusion of poles does not significantly affect convergence. However, in practice, each iteration becomes substantially more expensive when poles are included, leading to a noticeable slowdown in overall solver performance. Identifying the source of this increased cost and improving the efficiency of transform pair computations in this setting remains an area for future work.
\begin{figure}
    \centering
    \includegraphics[width=0.5\linewidth]{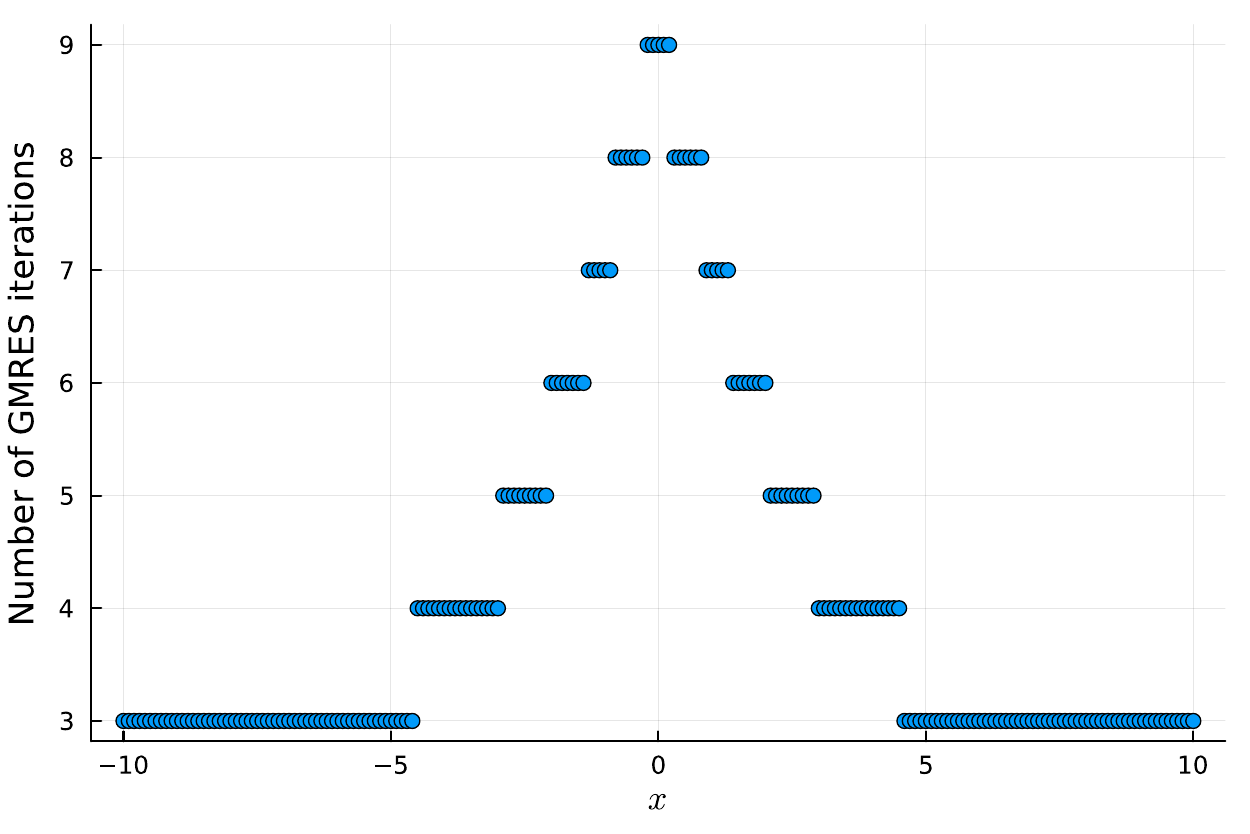}
    \caption{Number of \gls{gmres} iterations required to achieve a residual less than $10^{-10}$ for data and potential given by~\eqref{eq:Ex2}.}
    \label{fig:DiracGMRESPoles}
\end{figure}

\subsubsection{Example 3: Discontinuous potential, Gaussian data, and \texorpdfstring{$\tau=1$}{tau=1}}

Let 
\begin{equation}\label{eq:Ex3}
    \mathbf{f}(x)=\begin{bmatrix} e^{-x^2} \\ e^{-x^2} \end{bmatrix},\quad q(x)=\begin{cases}
        0,\quad &x\leq0,\\
        e^{-x},\quad &\text{otherwise}, 
    \end{cases}\quad \tau=1.
\end{equation}
This choice of $\tau$ produces no discrete spectrum. In contrast to the previous smooth examples, the discontinuity in $q(x)$ leads to two notable changes in behavior. First, the forward transform exhibits algebraic decay (Figure~\ref{fig:DiscontinuousPotentialDecay}), rather than the exponential decay obsered in Example~\ref{sec:Ex1}. Despite this slower decay, the inverse transform remains accurate to near machine precision (Figure~\ref{fig:DiscontinuousPotentialInverse}), demonstrating the robustness of the method. Second, the \gls{gmres} iteration counts are no longer symmetric in $x$ (Figure~\ref{fig:DiscontinuousPotentialGMRES}). In particular, \gls{gmres} converges in a single iteration for all $x\leq 0$. This is because $q(x)=0$ on $(-\infty,0]$, so the problem is significantly simplified in this region. For $x>0$, the nonzero potential produces a nontrivial operator leading to an increased iteration count.

\begin{figure}[htbp]
    \centering
    \begin{subfigure}[b]{0.49\textwidth}
        \centering
        \includegraphics[width=\textwidth]{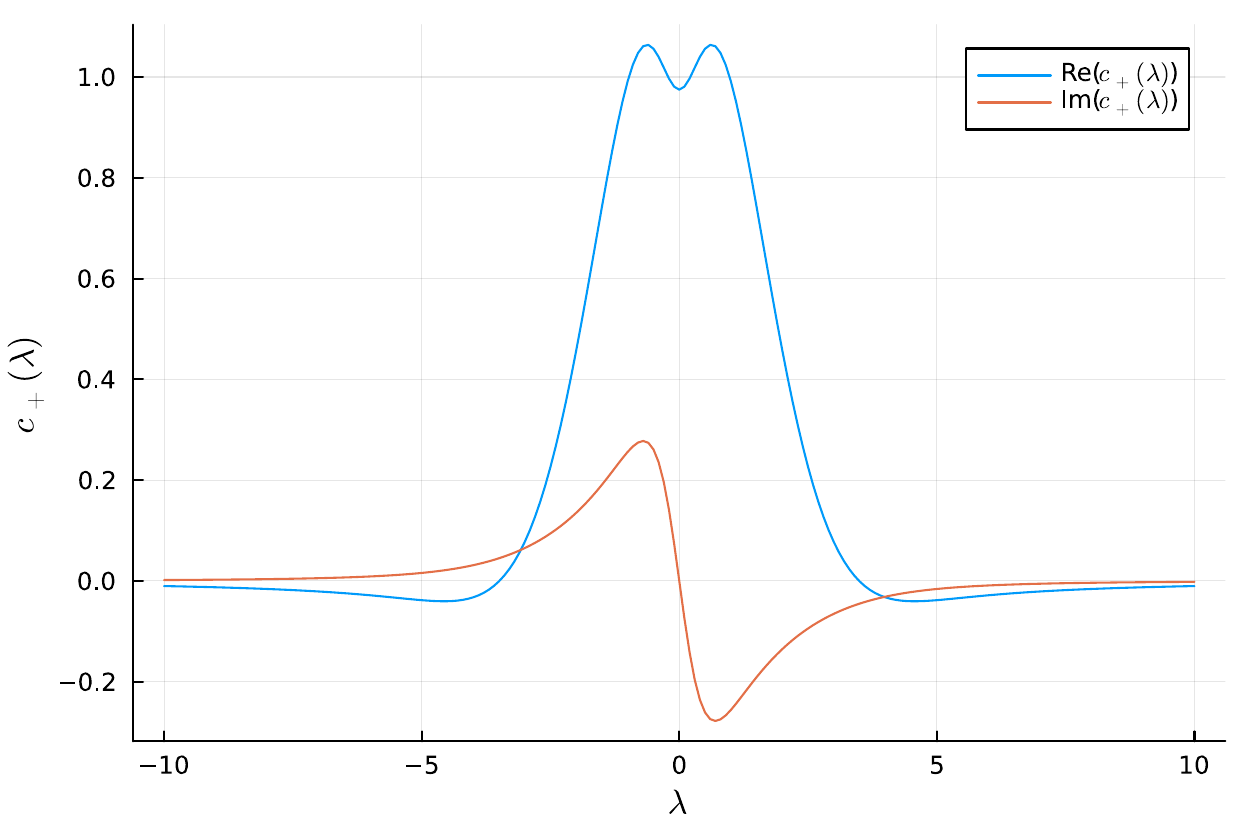}
        \caption{$c_+(\lambda)$}
        \label{fig:DiscontinuousPotential_cp}
    \end{subfigure}
    \hfill
    \begin{subfigure}[b]{0.49\textwidth}
        \centering
        \includegraphics[width=\textwidth]{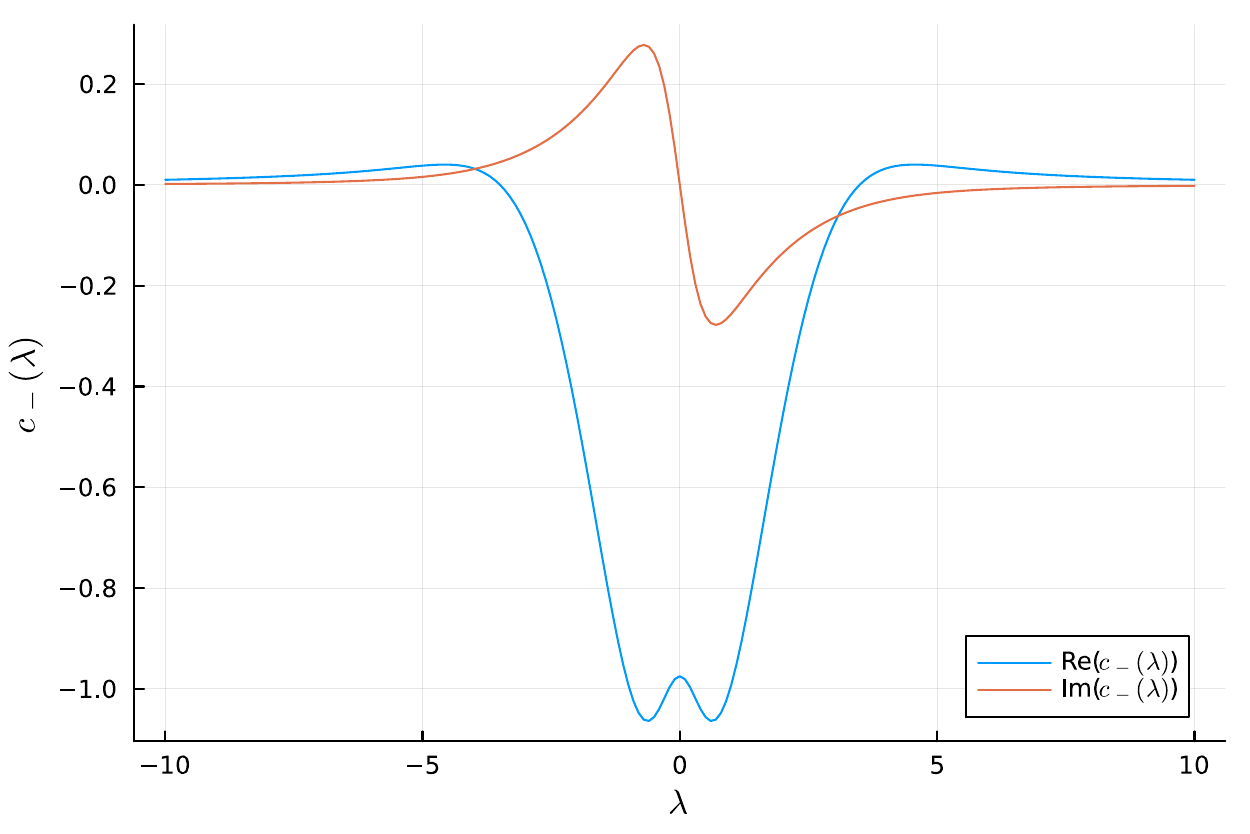}
        \caption{$c_-(\lambda)$}
        \label{fig:DiscontinuousPotential_cm}
    \end{subfigure}
    \caption{Forward transforms $c_+(\lambda)$ and $c_-(\lambda)$ with data and potential given by~\eqref{eq:Ex3} computed using the \gls{uclm}.}
    \label{fig:DiscontinuousPotentialForward}
\end{figure}
\begin{figure}
    \centering
    \includegraphics[width=0.5\linewidth]{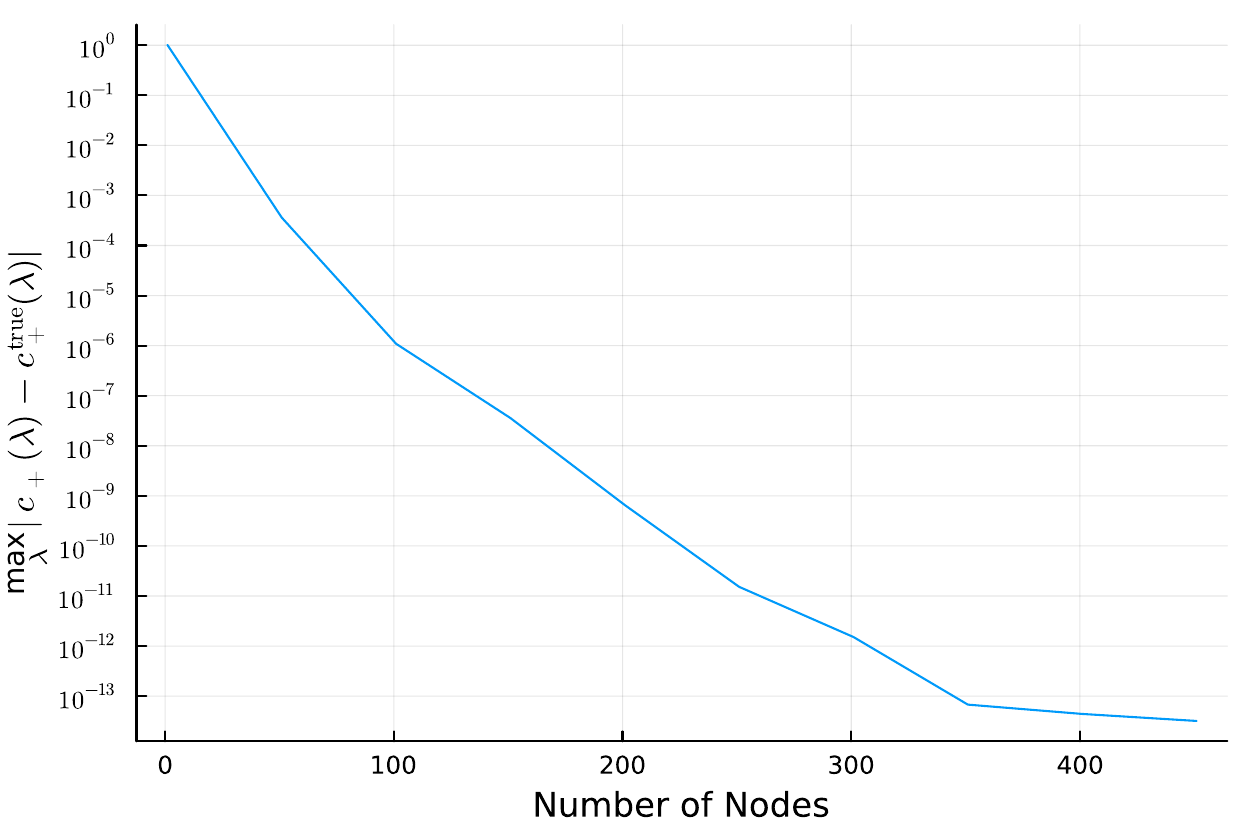}
    \caption{Max absolute error between $c_+(\lambda)$ for data and potential given by~\eqref{eq:Ex3} and highly resolved solution evaluated on a uniform grid in $\lambda \in [-35,35]$ as a function of the number of nodes used. Note that the error for $c_-(\lambda)$ looks extremely similar.}
    \label{fig:cp_err_discQ_contF}
\end{figure}
\begin{figure}
    \centering
    \includegraphics[width=0.5\linewidth]{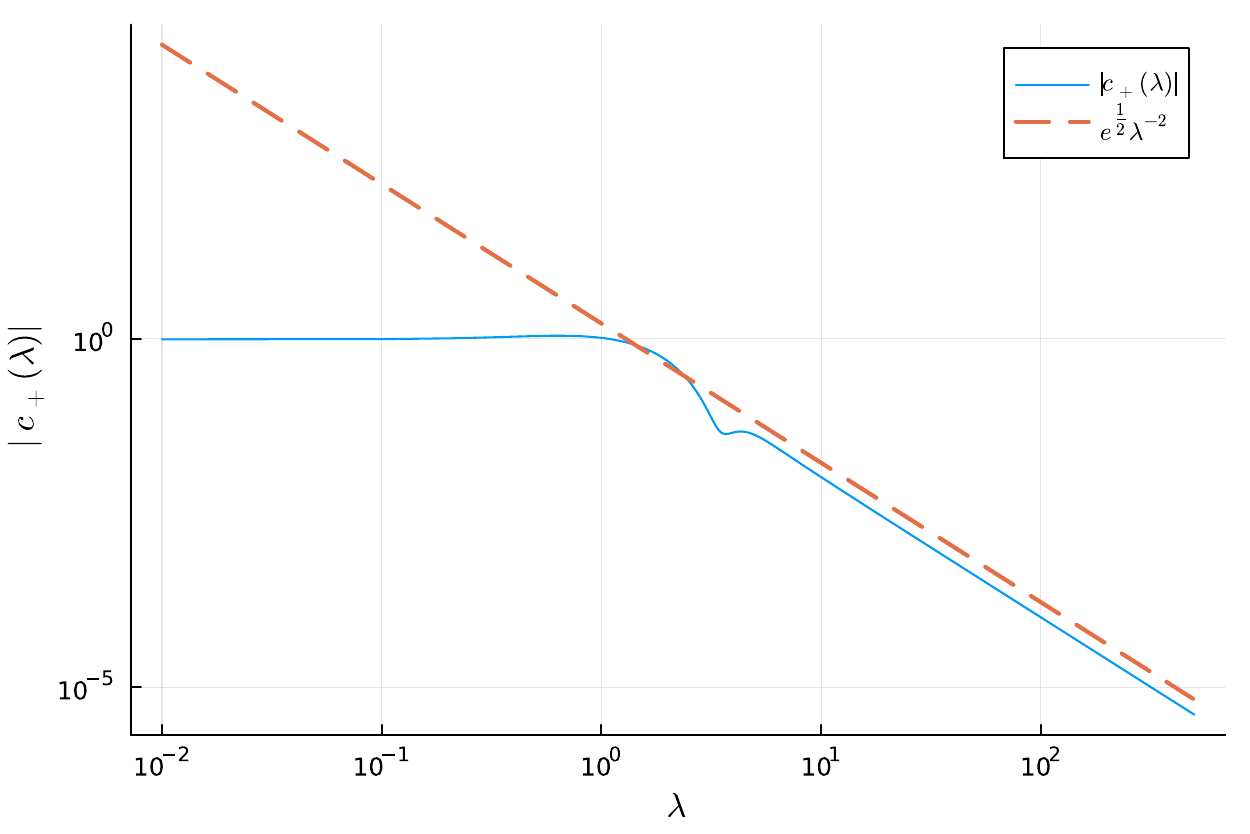}
    \caption{Magnitude of forward transform $|c_+(\lambda)|$ for data and potential given by~\eqref{eq:Ex3}. This was plotted on a log-log scale. We can see that the decay of the transform is algebraic.}
    \label{fig:DiscontinuousPotentialDecay}
\end{figure}
\begin{figure}[htbp]
    \centering
    \begin{subfigure}[b]{0.49\textwidth}
        \centering
        \includegraphics[width=\textwidth]{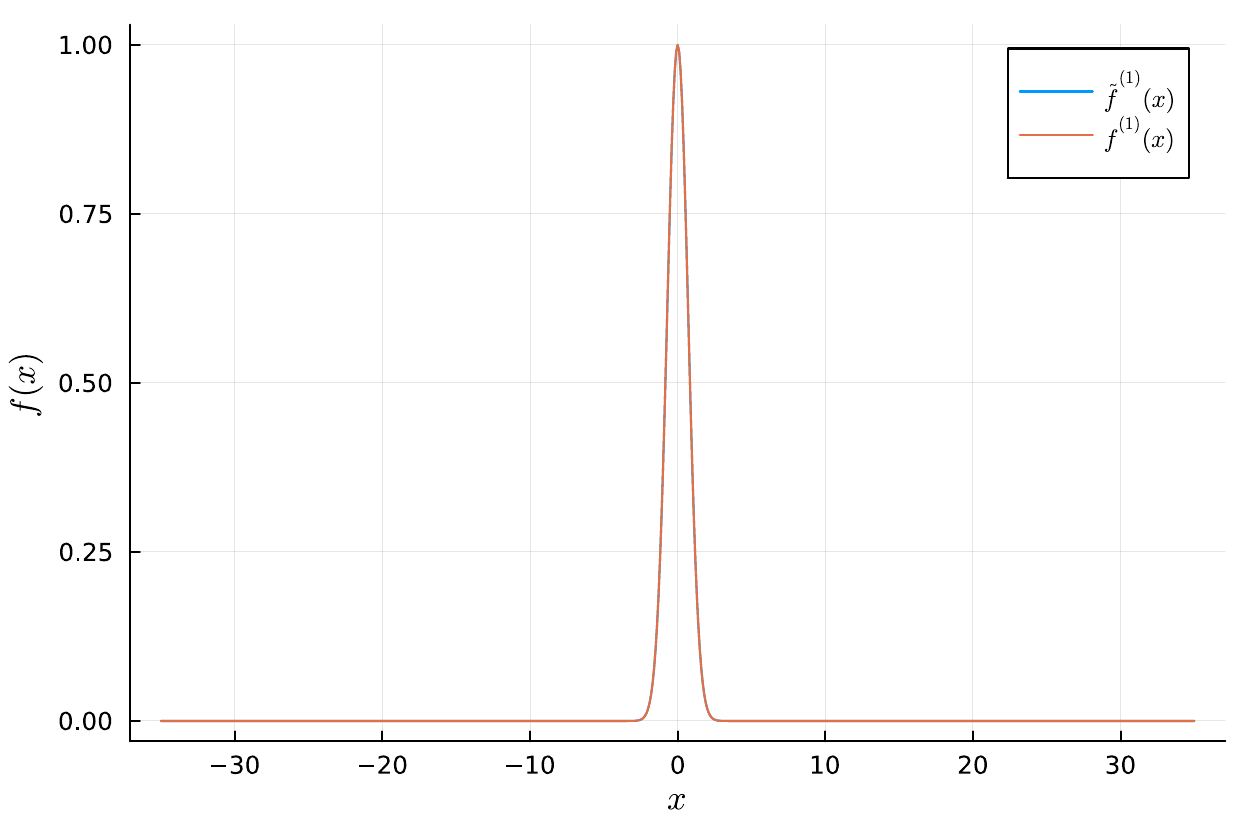}
        \caption{Inverse transform}
        \label{fig:DiscontinuousPotentialInversePlot}
    \end{subfigure}
    \hfill
    \begin{subfigure}[b]{0.49\textwidth}
        \centering
        \includegraphics[width=\textwidth]{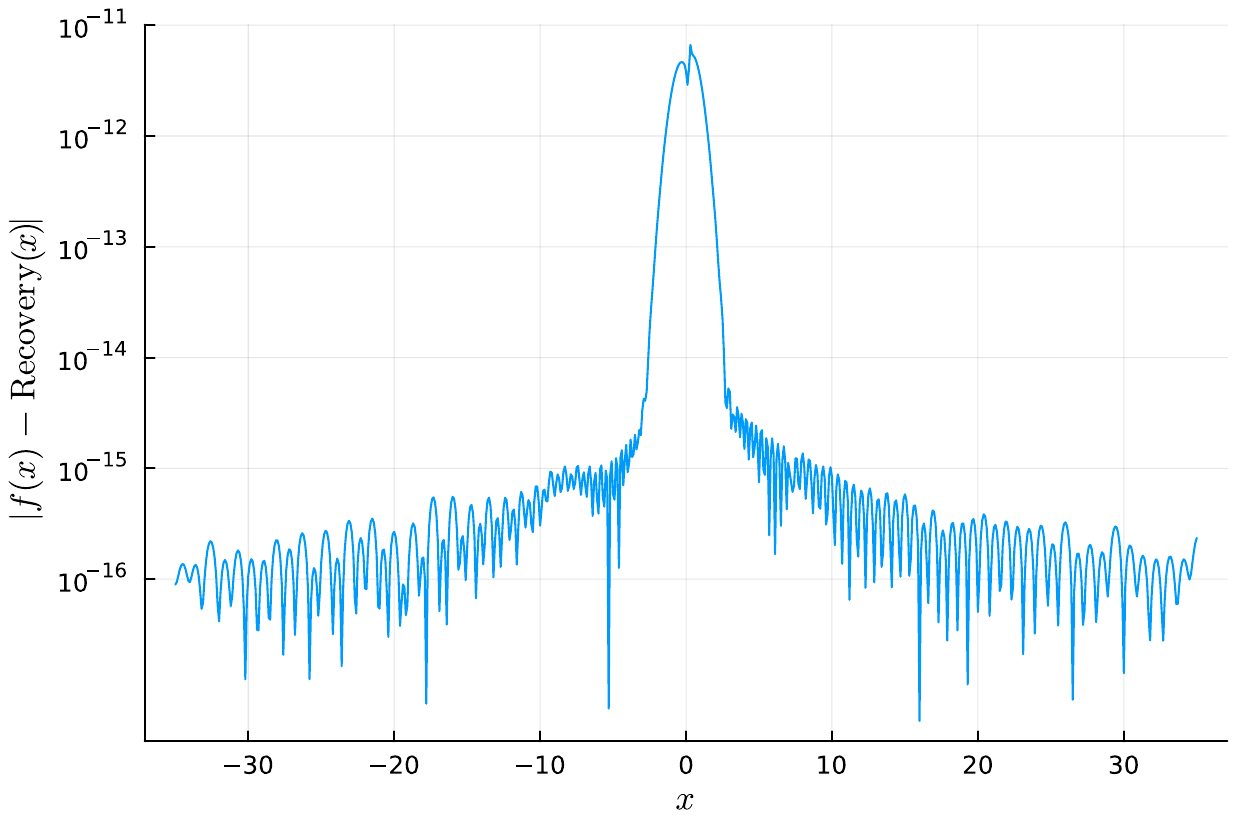}
        \caption{Error}
        \label{fig:DiscontinuousPotentialError}
    \end{subfigure}
    \caption{(a) Recovered first component of $\mathbf{f}(x)$, $\tilde{f}^{(1)}(x)$, and the original $f^{(1)}(x)$ with data and potential given by~\eqref{eq:Ex3}.  (b) Absolute difference between the recovery and the original. The \glspl{ode} for $\mathbf{m}_{\ell,\pm}(x;\lambda)$, $\ell\in\{1,2\}$, were solved using 300 collocation nodes in $x$, while those for $\mathbf{p}_\pm(x;\lambda)$ used 500 nodes. The reflection coefficients were represented using 270 coefficients, $c_\pm(\lambda)$ were expanded with 498 coefficients, and $\mathbf{m}_{\ell,\pm}(x;\lambda)$ were expanded using 232 coefficients.}
    \label{fig:DiscontinuousPotentialInverse}
\end{figure}
\begin{figure}
    \centering
    \includegraphics[width=0.5\linewidth]{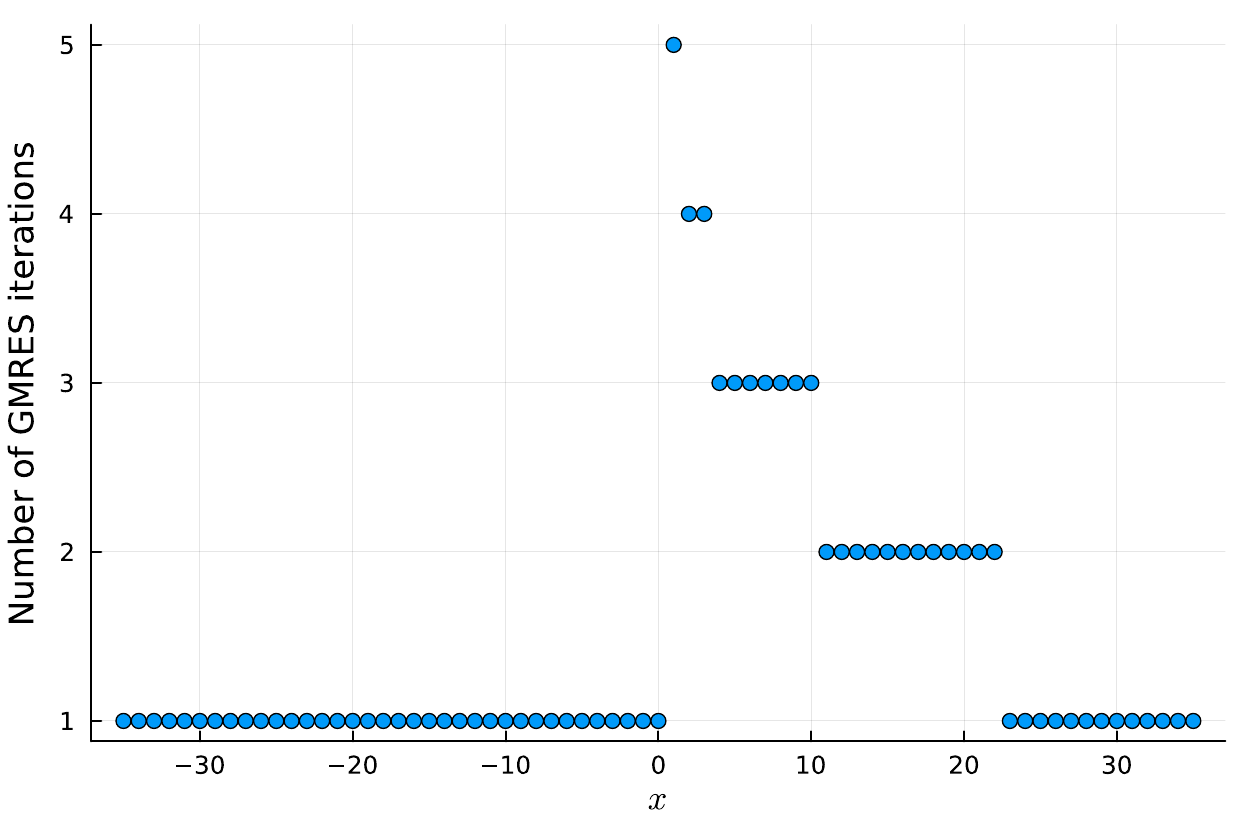}
    \caption{Number of \gls{gmres} iterations required to achieve a residual less than $10^{-10}$ for data and potential given by~\eqref{eq:Ex3}.}
    \label{fig:DiscontinuousPotentialGMRES}
\end{figure}

\subsubsection{Example 4: Gaussian potential, discontinuous data, and \texorpdfstring{$\tau=1$}{tau=1}}

Let 
\begin{equation}\label{eq:Ex4}
    \mathbf{f}(x)=\begin{bmatrix} f^{(1)}(x) \\ f^{(2)}(x) \end{bmatrix},\quad f^{(1)}(x)=f^{(2)}(x)=\begin{cases}
        0,\quad &x\leq 0,\\
        e^{-x},\quad &x>0,
    \end{cases}\quad q(x)=e^{-x^2},\quad \tau=1.
\end{equation}
This choice of $\tau$ produces no discrete spectrum. The discontinuity in the data produces algebraic decay in the transform (Figure~\ref{fig:DiscontinuousRHSDecay}), in contrast to the exponential decay observed in Example~\ref{sec:Ex1}. Despite this, the reconstruction remains accurate to near machine precision, and the \gls{gmres} iteration counts behave similarly to the smooth case. This demonstrates that the method is robust to discontinuities in the input data.  

\begin{figure}[htbp]
    \centering
    \begin{subfigure}[b]{0.49\textwidth}
        \centering
        \includegraphics[width=\textwidth]{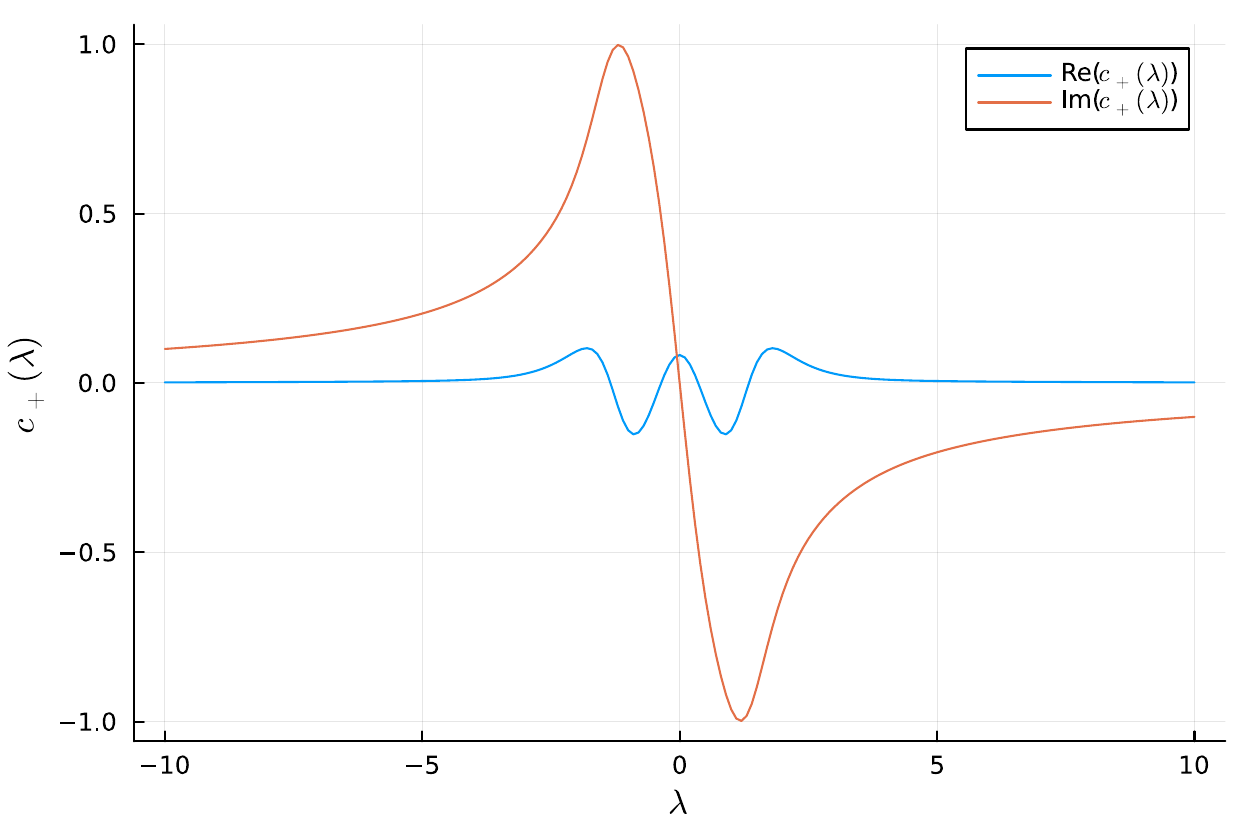}
        \caption{$c_+(\lambda)$}
        \label{fig:DiscontinuousRHS_cp}
    \end{subfigure}
    \hfill
    \begin{subfigure}[b]{0.49\textwidth}
        \centering
        \includegraphics[width=\textwidth]{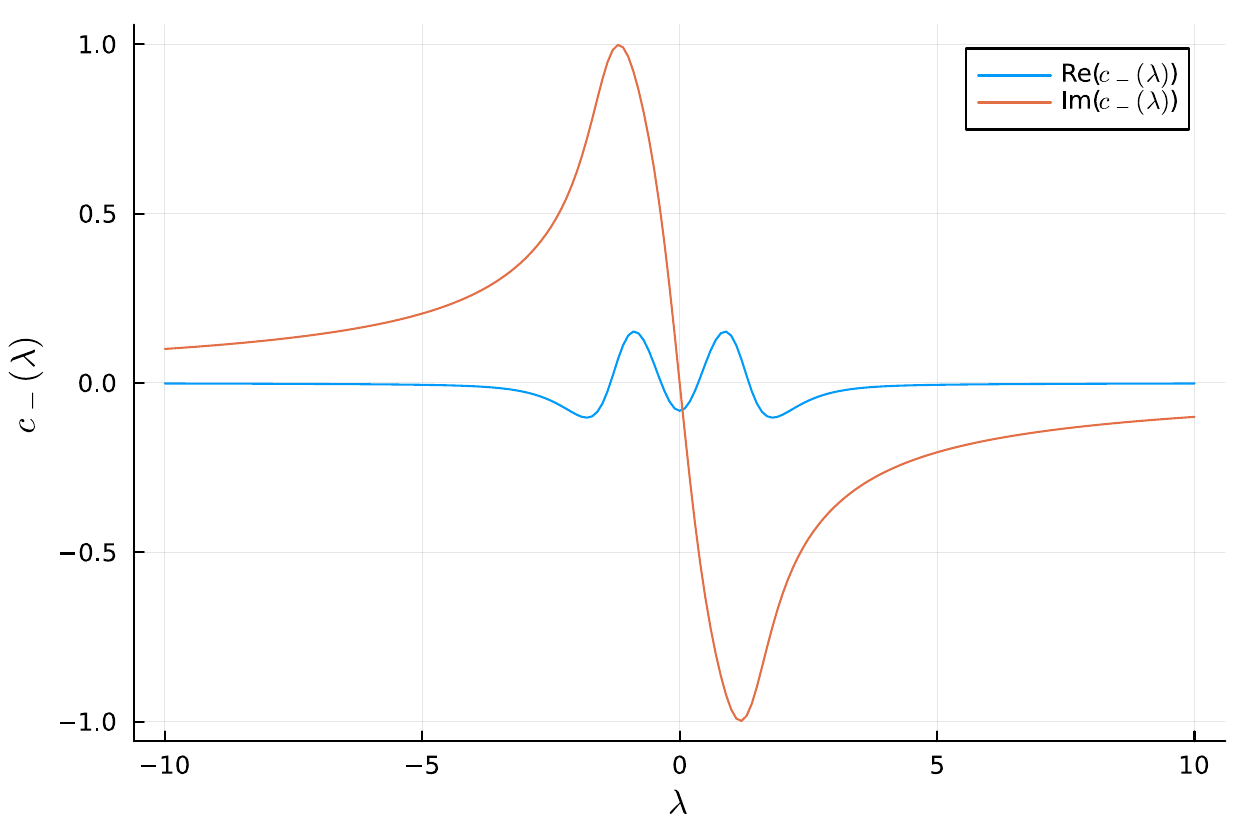}
        \caption{$c_-(\lambda)$}
        \label{fig:DiscontinuousRHS_cm}
    \end{subfigure}
    \caption{Forward transforms $c_+(\lambda)$ and $c_-(\lambda)$ with data and potential given by~\eqref{eq:Ex4} computed using the \gls{uclm}.}
    \label{fig:DiscontinuousRHSForward}
\end{figure}
\begin{figure}
    \centering
    \includegraphics[width=0.5\linewidth]{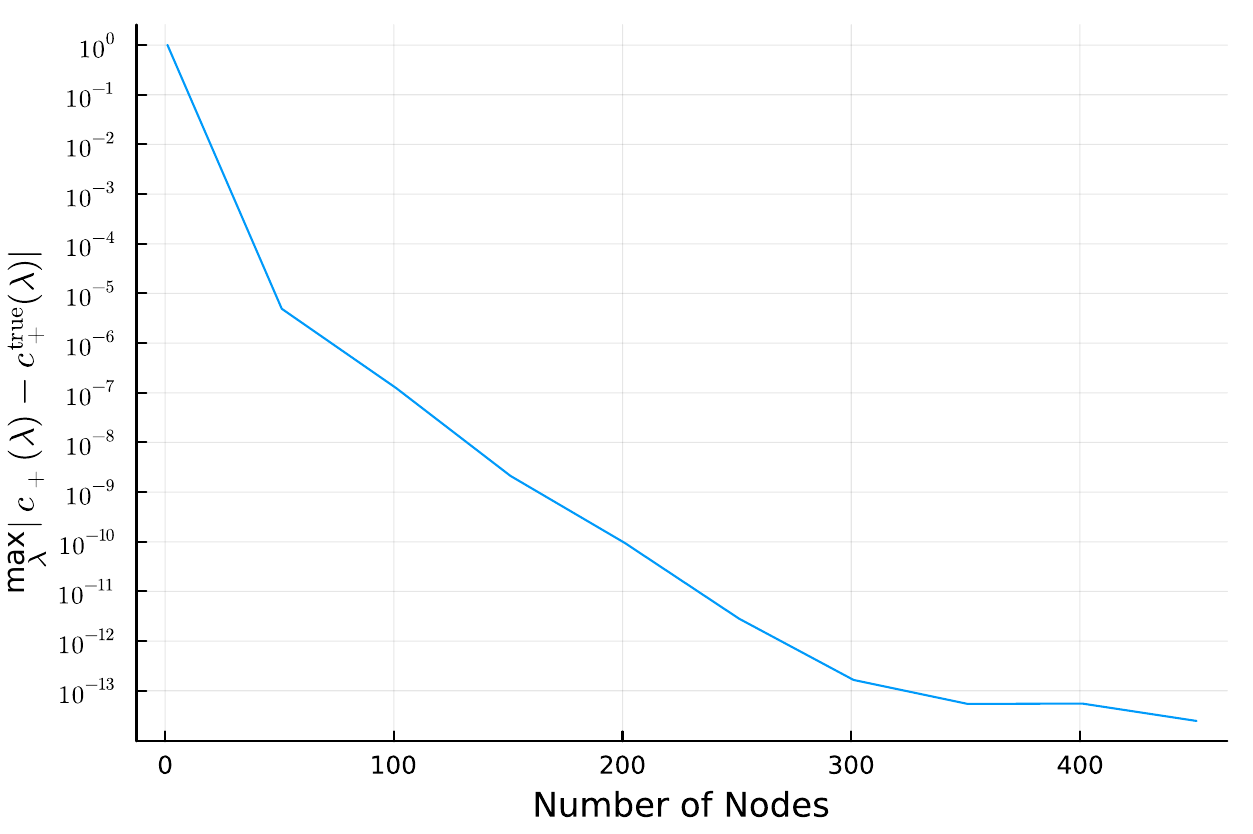}
    \caption{Max absolute error between $c_+(\lambda)$ for data and potential given by~\eqref{eq:Ex4} and highly resolved solution evaluated on a uniform grid in $\lambda \in [-35,35]$ as a function of the number of nodes used. Note that the error for $c_-(\lambda)$ looks extremely similar.}
    \label{fig:cp_errs_contPot_discF}
\end{figure}
\begin{figure}
    \centering
    \includegraphics[width=0.5\linewidth]{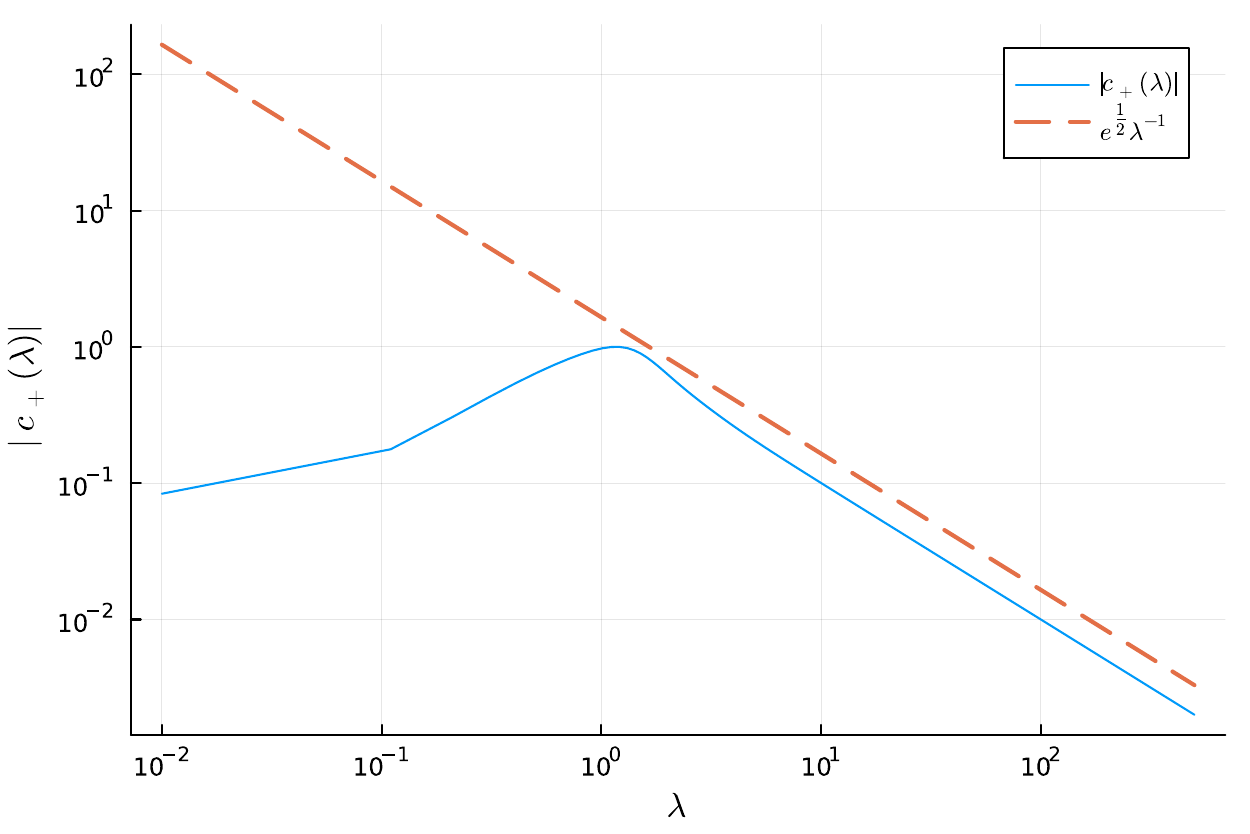}
    \caption{Magnitude of forward transform $|c_+(\lambda)|$ with data and potential given by~\eqref{eq:Ex4}. This was plotted on a log-log scale. We can see that the decay of the transform is algebraic.}
    \label{fig:DiscontinuousRHSDecay}
\end{figure}
\begin{figure}[htbp]
    \centering
    \begin{subfigure}[b]{0.49\textwidth}
        \centering
        \includegraphics[width=\textwidth]{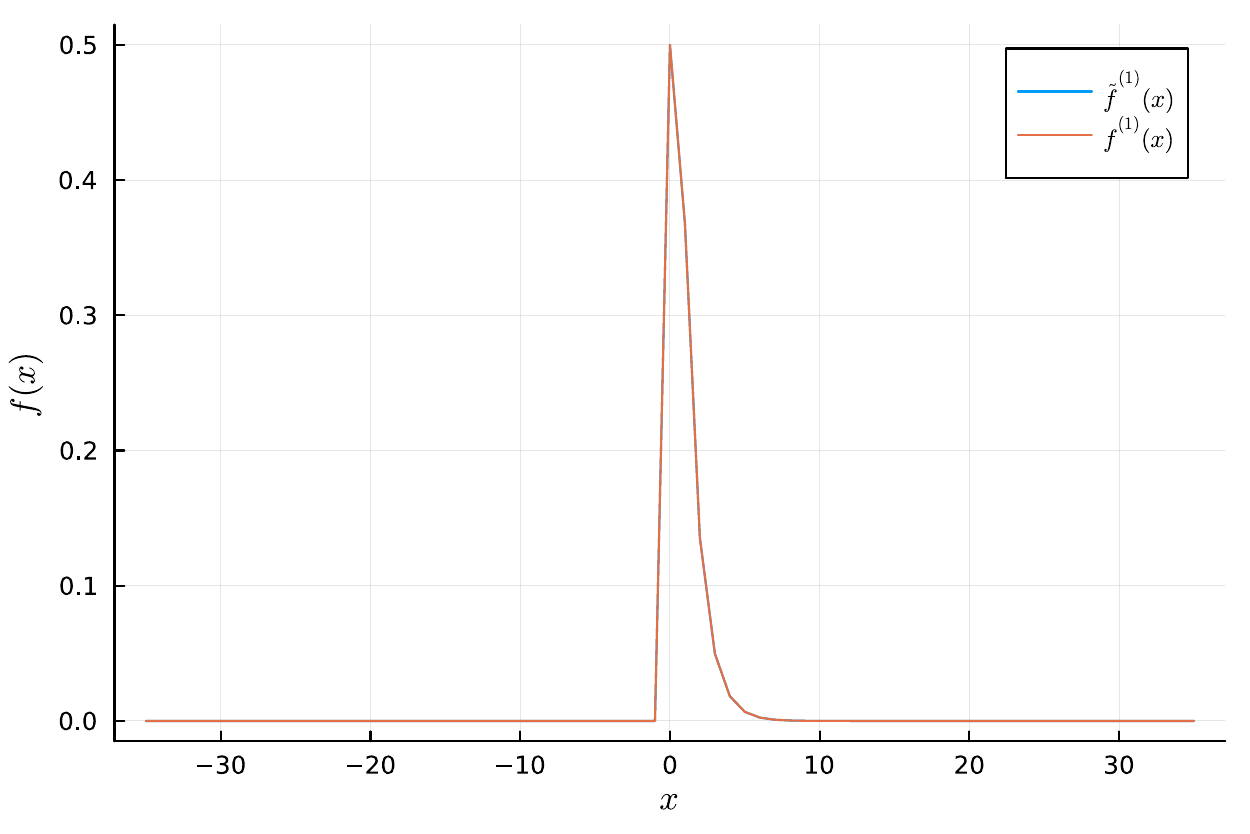}
        \caption{Inverse transform}
        \label{fig:DiscontinuousRHSInversePlot}
    \end{subfigure}
    \hfill
    \begin{subfigure}[b]{0.49\textwidth}
        \centering
        \includegraphics[width=\textwidth]{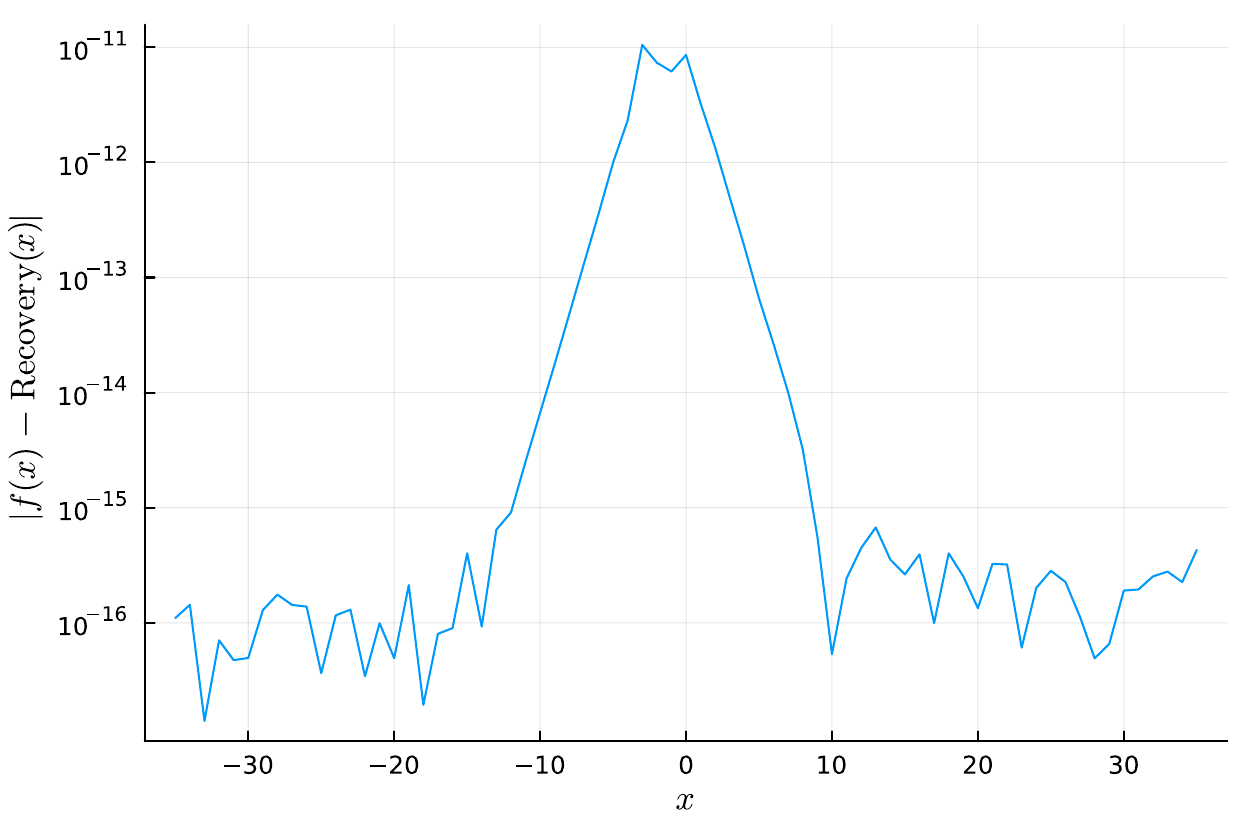}
        \caption{Error}
        \label{fig:DiscontinuousRHSError}
    \end{subfigure}
    \caption{(a) Recovered first component of $\mathbf{f}(x)$, $\tilde{f}^{(1)}(x)$, and the original $f^{(1)}(x)$ with data and potential given by~\eqref{eq:Ex4}. (b) Absolute difference between the recovery and the original. The \glspl{ode} for $\mathbf{m}_{\ell,\pm}(x;\lambda)$, $\ell\in\{1,2\}$, were solved using 300 collocation nodes in $x$, while those for $\mathbf{p}_\pm(x;\lambda)$ used 500 nodes. The reflection coefficients were represented using 270 coefficients, $c_\pm(\lambda)$ were expanded with 458 coefficients, and $\mathbf{m}_{\ell,\pm}(x;\lambda)$ were expanded using 232 coefficients.}
    \label{fig:DiscontinuousRHSInverse}
\end{figure}
\begin{figure}
    \centering
    \includegraphics[width=0.5\linewidth]{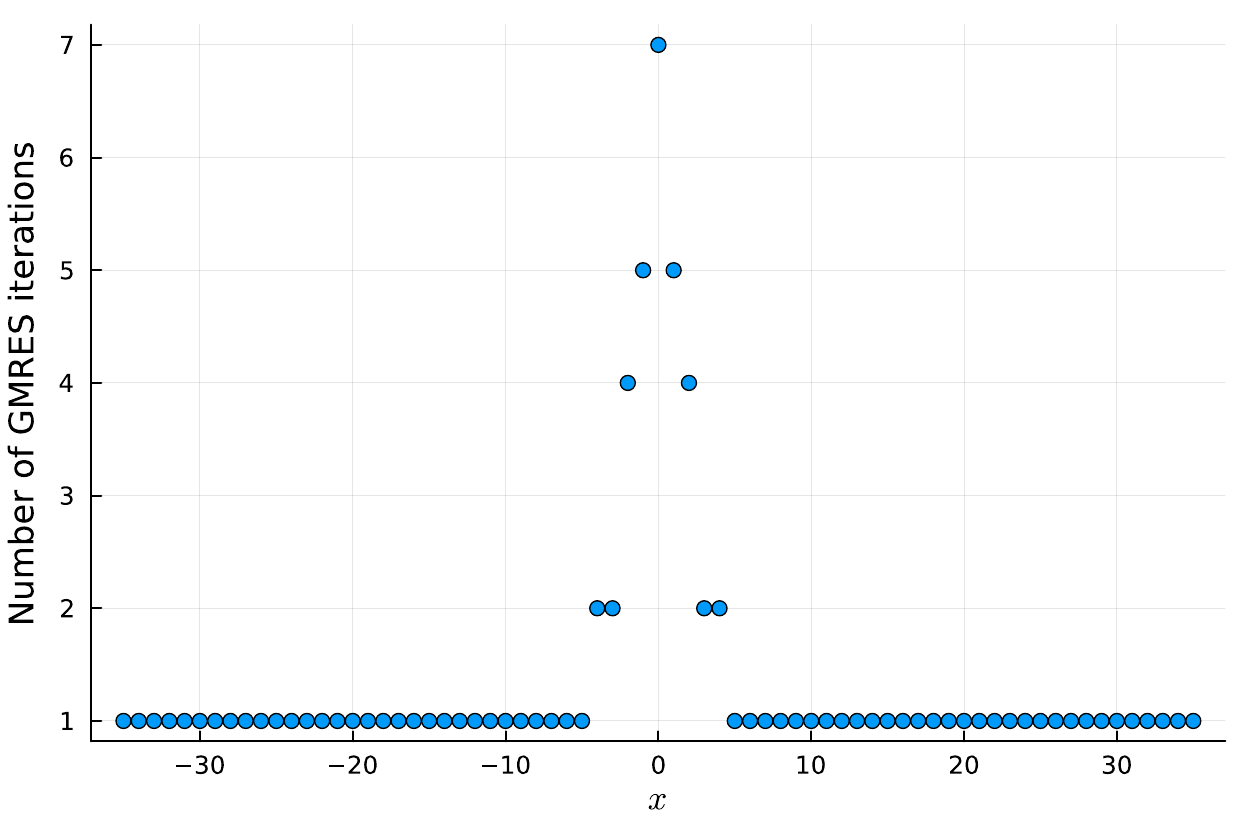}
    \caption{Number of \gls{gmres} iterations required to achieve a residual less than $10^{-10}$ for data and potential given by~\eqref{eq:Ex4}.}
    \label{fig:DiscontinuousRHSGMRES}
\end{figure}

\subsubsection{Example 5: Discontinuous potential, discontinuous data, \texorpdfstring{$\tau=1$}{tau=1}}

Let 
\begin{equation}\label{eq:Ex5}
    \mathbf{f}(x)=\begin{bmatrix} f^{(1)}(x) \\ f^{(2)}(x) \end{bmatrix},\quad f^{(1)}(x)=f^{(2)}(x)=q(x)=\begin{cases}
        0,\quad &x\leq 0,\\
        e^{-x},\quad &x>0,
    \end{cases}\quad \tau=1.
\end{equation}
This choice of $\tau$ produces no discrete spectrum. When both the potential and data are discontinuous, we again observe algebraic decay in the forward transform. Nevertheless, the reconstruction accuracy and \gls{gmres} performance remain comparable to the previous examples, indicating that the method is accurate even in the presence of simultaneous discontinuities. 

\begin{figure}[htbp]
    \centering
    \begin{subfigure}[b]{0.49\textwidth}
        \centering
        \includegraphics[width=\textwidth]{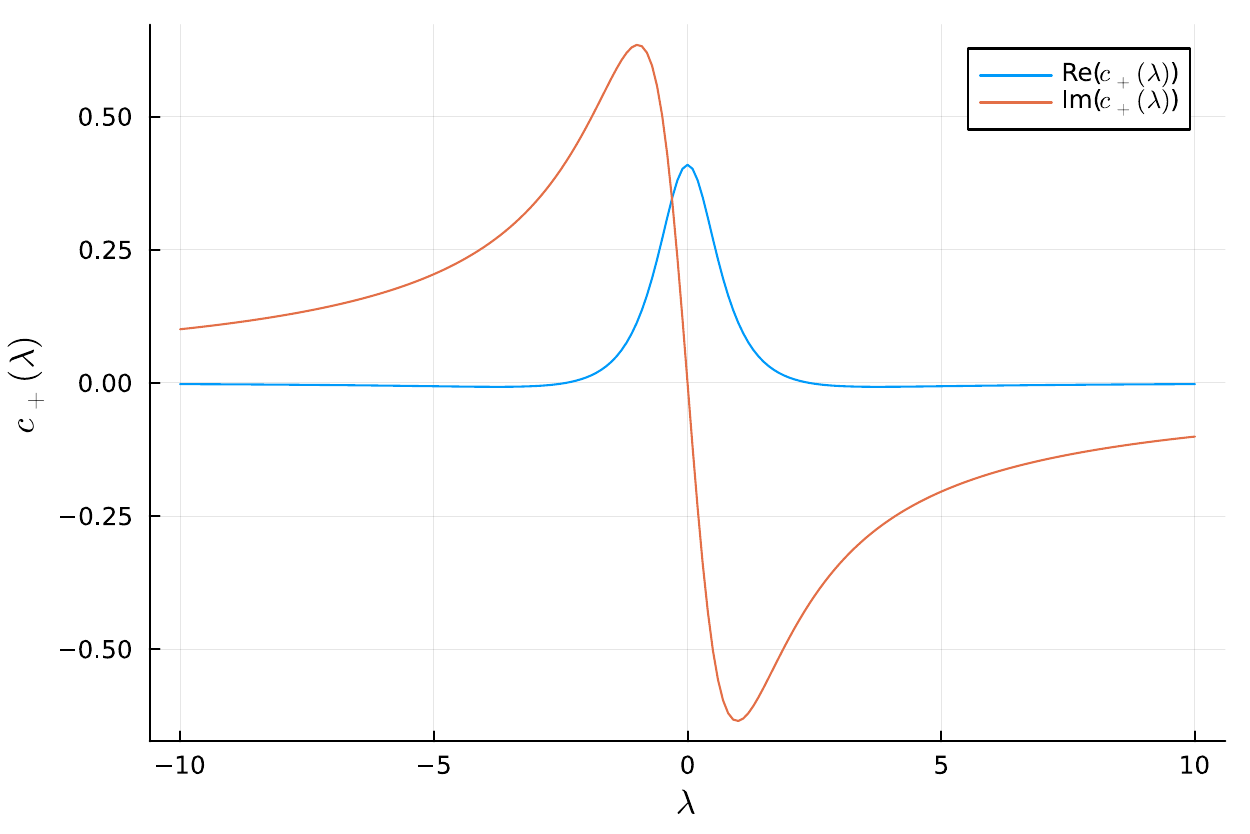}
        \caption{$c_+(\lambda)$}
        \label{fig:DiscontinuousPotentialRHSSame_cp}
    \end{subfigure}
    \hfill
    \begin{subfigure}[b]{0.49\textwidth}
        \centering
        \includegraphics[width=\textwidth]{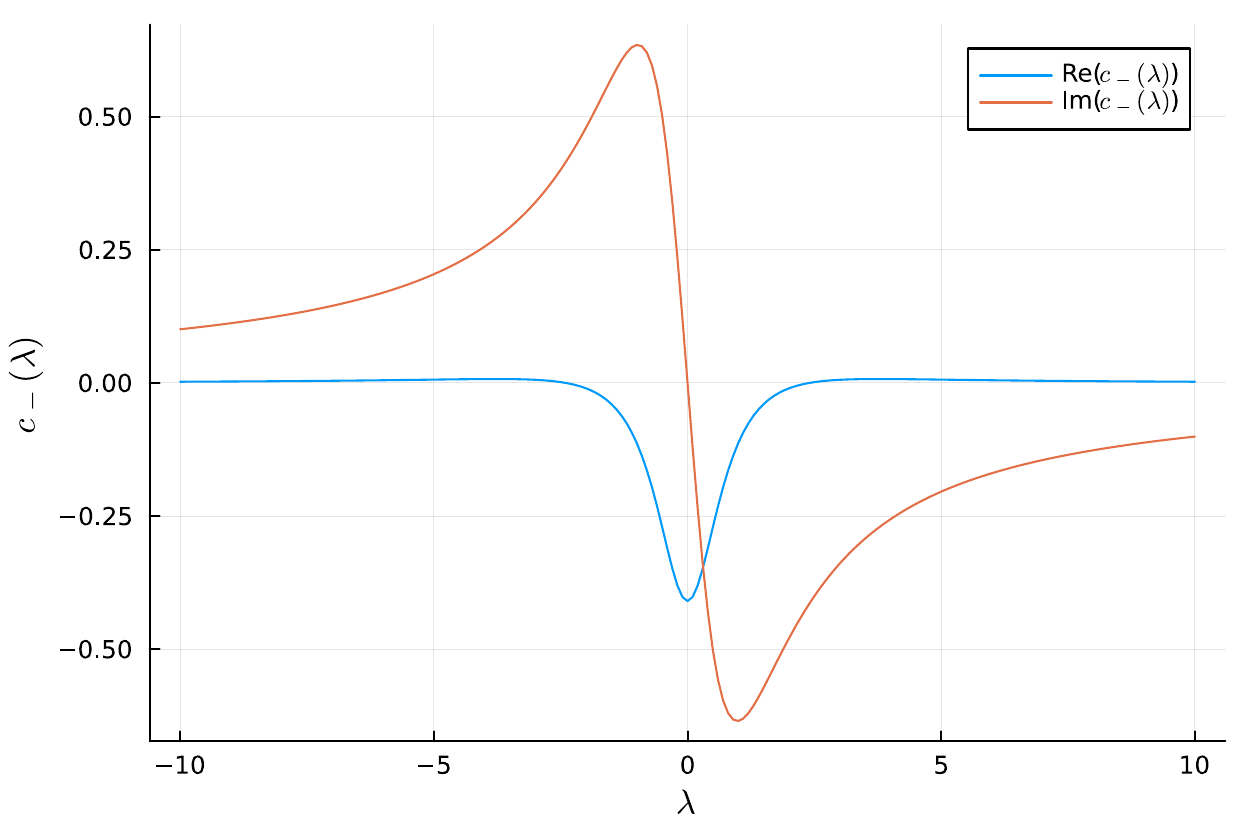}
        \caption{$c_-(\lambda)$}
        \label{fig:DiscontinuousPotentialRHSSame_cm}
    \end{subfigure}
    \caption{Forward transforms $c_+(\lambda)$ and $c_-(\lambda)$ with data and potential given by~\eqref{eq:Ex5} computed using the \gls{uclm}.}
    \label{fig:DiscontinuousPotentialRHSSameForward}
\end{figure}
\begin{figure}
    \centering
    \includegraphics[width=0.5\linewidth]{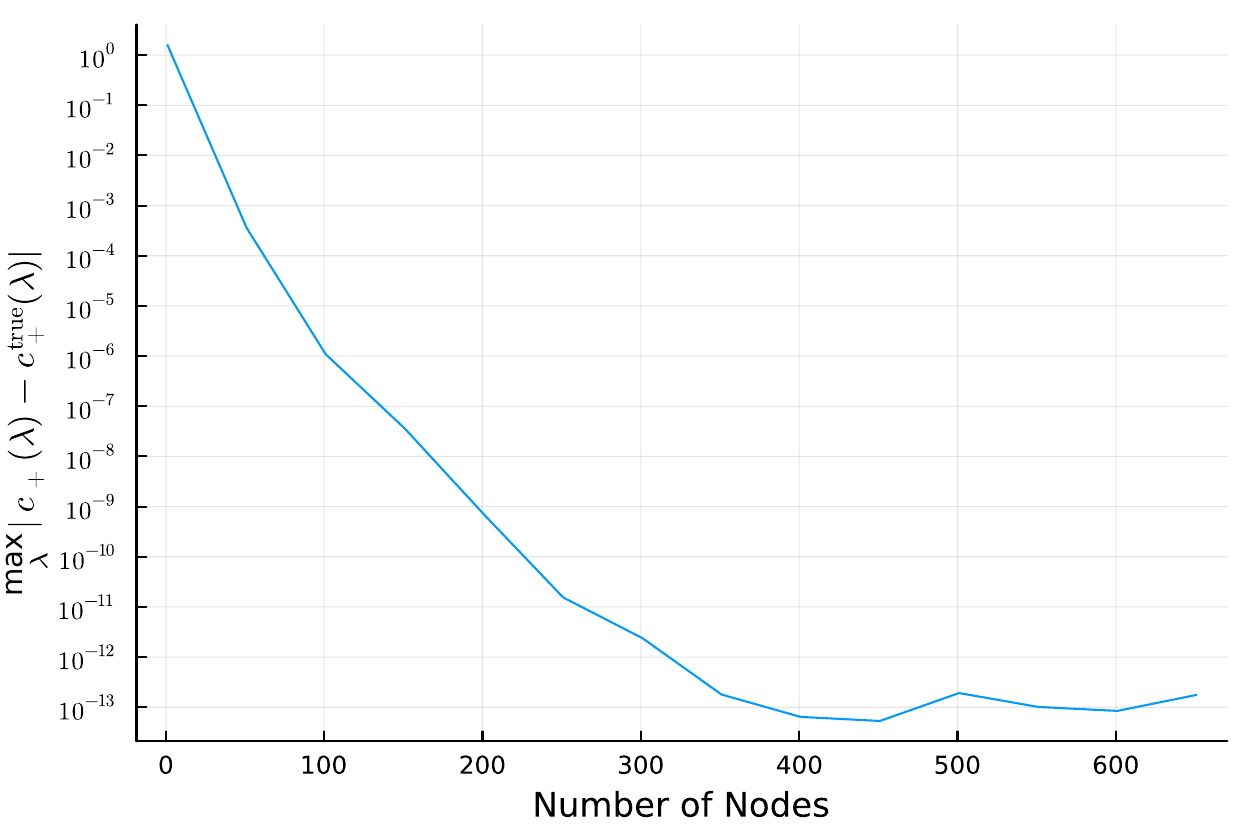}
    \caption{Max absolute error between $c_+(\lambda)$ with data and potential given by~\eqref{eq:Ex5} and highly resolved solution evaluated on a uniform grid in $\lambda \in [-35,35]$ as a function of the number of nodes used. Note that the error for $c_-(\lambda)$ looks extremely similar.}
    \label{fig:cp_err_discQF}
\end{figure}
\begin{figure}
    \centering
    \includegraphics[width=0.5\linewidth]{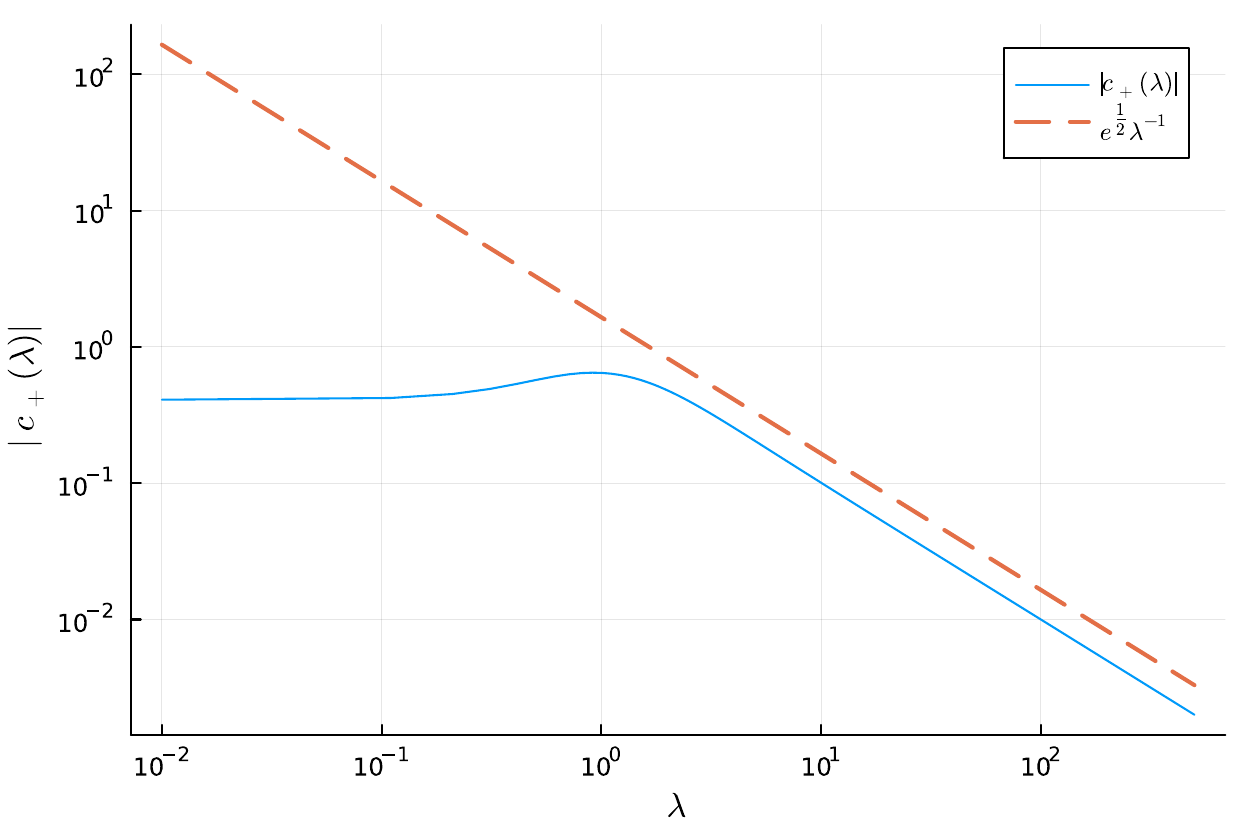}
    \caption{Magnitude of forward transform $|c_+(\lambda)|$ with data and potential given by~\eqref{eq:Ex5}. This was plotted on a log-log scale. We can see that the decay of the transform is algebraic.}
    \label{fig:DiscontinuousPotentialRHSSameDecay}
\end{figure}
\begin{figure}[htbp]
    \centering
    \begin{subfigure}[b]{0.49\textwidth}
        \centering
        \includegraphics[width=\textwidth]{Figures/DiscontinuousRHSInvTRans2.pdf}
        \caption{Inverse transform}
        \label{fig:DiscontinuousPotentialRHSSameInversePlot}
    \end{subfigure}
    \hfill
    \begin{subfigure}[b]{0.49\textwidth}
        \centering
        \includegraphics[width=\textwidth]{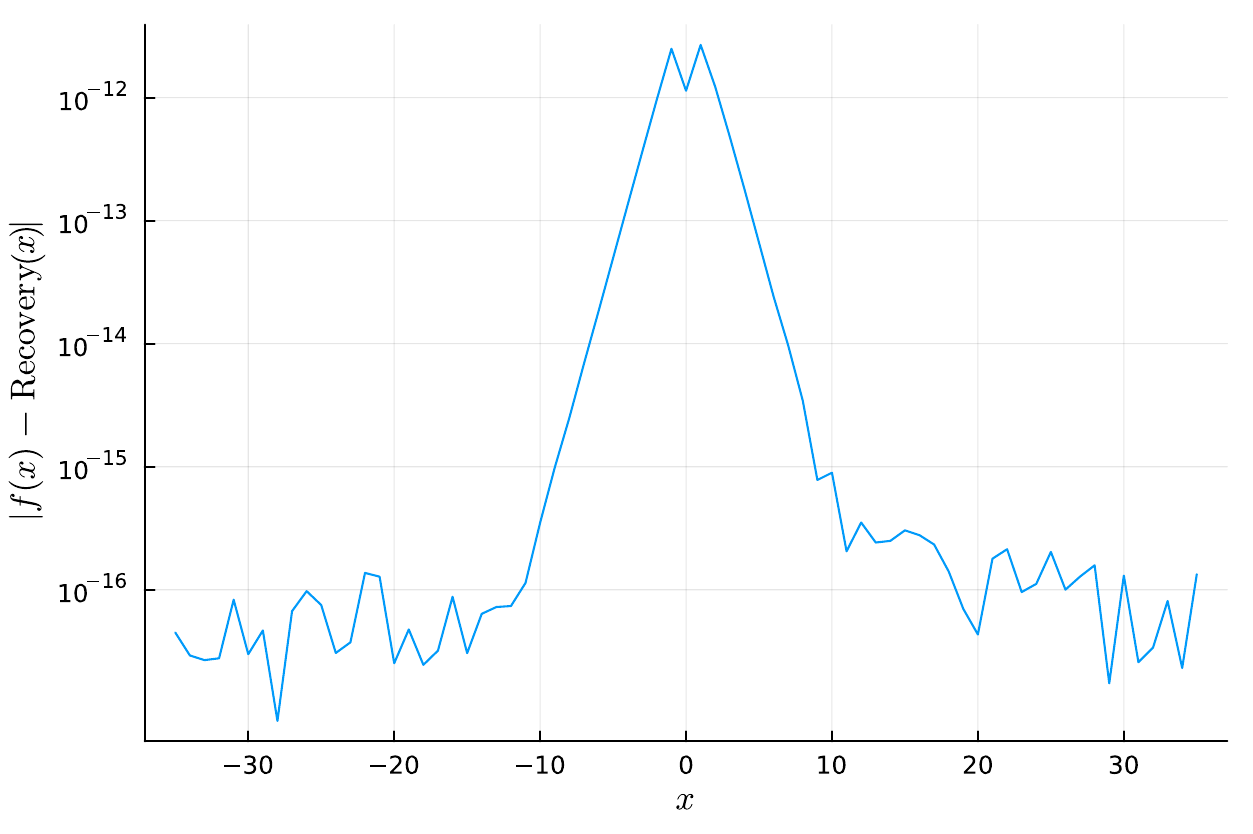}
        \caption{Error}
        \label{fig:DiscontinuousPotentialRHSSameError}
    \end{subfigure}
    \caption{(a) Recovered first component of $\mathbf{f}(x)$, $\tilde{f}^{(1)}(x)$, and the original $f^{(1)}(x)$ with data and potential given by~\eqref{eq:Ex5}. (b) Absolute difference between the recovery and the original. The \glspl{ode} for $\mathbf{m}_{\ell,\pm}(x;\lambda)$, $\ell\in\{1,2\}$, were solved using 300 collocation nodes in $x$, while those for $\mathbf{p}_\pm(x;\lambda)$ used 500 nodes. The reflection coefficients were represented using 270 coefficients, $c_\pm(\lambda)$ were expanded with 448 coefficients, and $\mathbf{m}_{\ell,\pm}(x;\lambda)$ were expanded using 232 coefficients. }
    \label{fig:DiscontinuousPotentialRHSSameInverse}
\end{figure}
\begin{figure}
    \centering
    \includegraphics[width=0.5\linewidth]{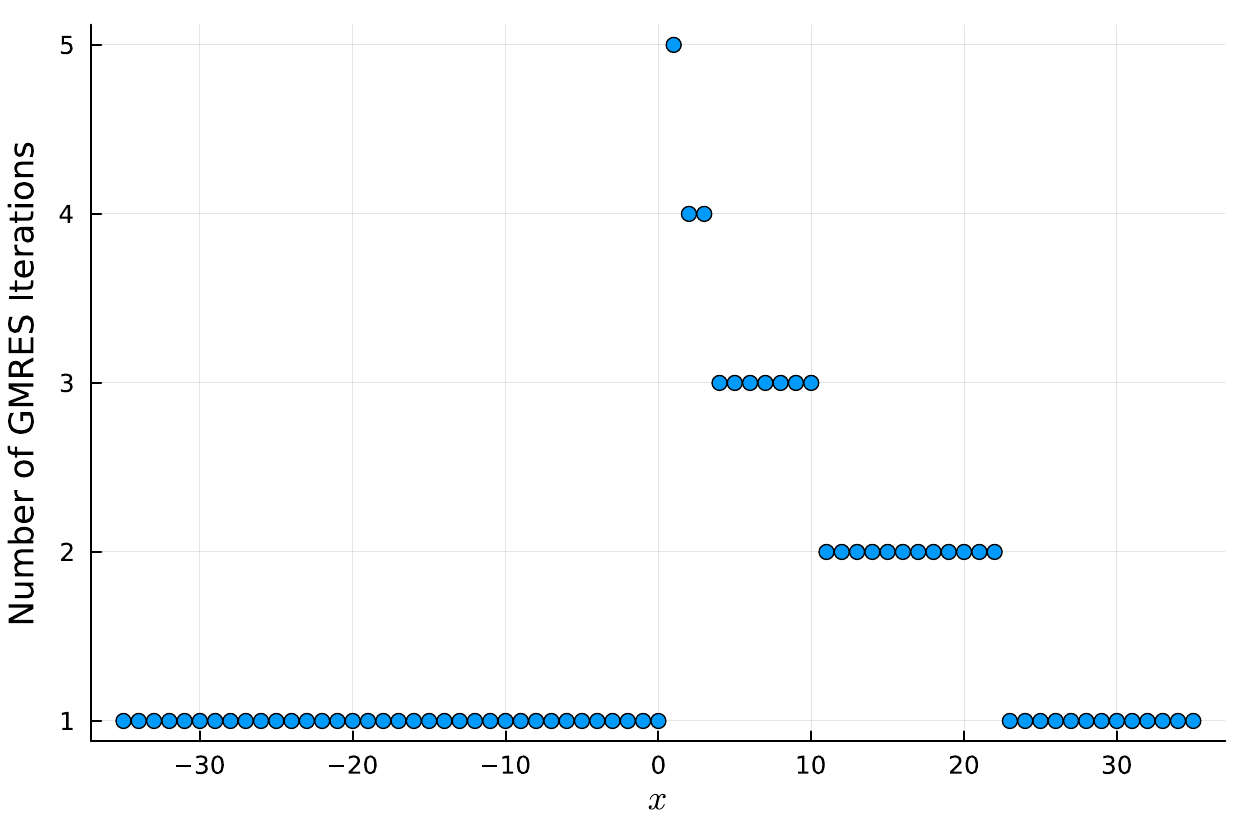}
    \caption{Number of \gls{gmres} iterations required to achieve a residual less than $10^{-10}$ for data and potential given by~\eqref{eq:Ex5}.}
    \label{fig:DiscontinuousPotentialRHSSameGMRES}
\end{figure}

\section*{Acknowledgments}

This material is based upon work supported by the National Science Foundation under Grant Nos. DGE-2140004 (KL) and DMS-2306438 (TT). Any opinions, findings, and conclusions or recommendations expressed in this material are those of the authors and do not necessarily reflect the views of the National Science Foundation. The work was also partially supported by the ARCS (Achievement Rewards for College Scientists) Foundation.

\section*{Appendix}

\appendix 

\section{Rational Functions and Oscillatory Cauchy Integrals}\label{sec:RationalBasis}

In this appendix we summarize the analytic and computational framework developed in \cite{trogdon_scattering_2021}, building on the rational approximation techniques introduced in \cite{trogdon_rational_2014}. The oscillatory rational basis underlying these constructions originates in \cite{weber_numerical_1980,11e5cc21-6b4a-3b97-8220-bac5264f5b1f,Olver2011ComputingTH}. None of results presented in this appendix are novel. We include this material to make the present paper self-contained.

Note that the oscillatory rational basis and associated operations, such as function approximation, multiplication, derivatives, integration, Cauchy integrals, and inner products have been incorporated into a Julia library, \texttt{OperatorApproximation.jl}, as a result of this work~\cite{OperatorApproximation}. Within this code base, we cycle between the presented basis, $R_{j,\alpha}(k)$, and a modified basis
\begin{equation*}
    \tilde{R}_{j,\alpha}(k)=e^{ik\alpha}\left(\frac{k-i}{k+i}\right)^j.
\end{equation*}
Swapping between these two bases is trivial, and swapping can be advantageous for certain operations. For example, the multiplication operation is considerably simpler using the $\tilde{R}_{j,\alpha}(k)$ basis, while the application of the Cauchy integral operator is simpler using $R_{j,\alpha}(k)$. For simplicity, we present everything in this work in terms of $R_{j,\alpha}(k)$.

\subsection{Rational Approximation}

We consider the problem of the rational approximation of $f:\mathbb R\rightarrow\mathbb C$, under suitable regularity conditions. We begin by discussing trigonometric interpolation of an associated continuous periodic function $F$. Define $\theta_j=2\pi j/n$ for $j=0,...,n-1$, $n\in\mathbb N$. Further define
\begin{equation*}
    n_+=\lfloor n/2\rfloor,\quad n_-=\lfloor (n-1)/2\rfloor.
\end{equation*}
\begin{definition}
    The discrete Fourier transform of order $n$ of a continuous function $F$ is the mapping
    \begin{align*}
        \mathcal{F}_nF = \left[\Tilde{F}_0,\Tilde{F}_1,...,\Tilde{F}_n\right]^T,\\
        \Tilde{F}_k=\frac{1}{n}\sum_{j=0}^{n-1}e^{-ik\theta_j}F(\theta_j).
    \end{align*}
\end{definition}
The \gls{fft} is an algorithm that implements the discrete Fourier transform in $\mathcal{O}(n\log n)$ floating point operations. To get the coefficients for the trigonometric interpolant of $F$ we use the following formula:
\begin{equation*}
    \mathcal{I}_nF(\theta)=\sum_{k=-n_-}^{n_+}e^{ik\theta}\Tilde{F}_k,\quad\mathcal{I}_n
F(\theta_j)=F(\theta_j),\quad j=0,...,n-1.
\end{equation*}
We now define the following M\"obius transformation
\begin{equation*}
    T(k)=\frac{k-i}{k+i},\quad T^{-1}(z)=\frac{1}{i}\frac{z+1}{z-1},
\end{equation*}
where $T$ maps the real axis onto the unit circle. There is a set of basis functions related to these transformations given by
\begin{equation*}
    R_{j,\alpha}(k)=e^{ik\alpha}\left[\left(\frac{k-i}{k+i}\right)^j-1\right].
\end{equation*}
These are referred to as the oscillatory rational basis functions when $\alpha\neq0$.

Assume $f$ is a smooth rapidly decaying or rational function on $\mathbb R$, decaying at infinity. Then $F(\theta)=f(T^{-1}(e^{i\theta}))$ maps $f$ to a smooth function on $[0,2\pi]$. Hence, the \gls{fft} can be applied to $F(\theta)$ to obtain an interpolant $\mathcal{I}_nF(\theta)$. Inverting the transformation $k=T^{-1}(e^{i\theta})$, we arrive at the following rational approximation of $f$:
\begin{equation*}
    \mathcal{R}_nf(k):=\mathcal{I}_nF(T^{-1}(k)).
\end{equation*}
This approximation has been shown to converge rapidly~\cite{trogdon_rational_2014}. We have that
\begin{equation*}
    \mathcal{R}_nf(k)=\sum_{j=-n_-}^{n_+}\Tilde{F}_jT^j(k)=\sum_{j=-n_-}^{n_+}\Tilde{F}_jR_{j,0}(k)
\end{equation*}
because $f(\infty)=0$ and we choose $\theta=0$ ($k=\infty)$ to be an interpolation point.
Define the oscillatory interpolation operator by
\begin{equation*}
    \mathcal{R}_{n,\alpha}f(k)=e^{i\alpha k}\mathcal{R}_n\left[f(\diamond)e^{-i\alpha(\diamond)}\right](k).
\end{equation*}

We also want to allow vector inputs. To do so, we overload notation. Given interpolation data
$\mathbf{c}=[c_0,...,c_{n-1}]^T\in\mathbb{R}^n$ at the nodes $k_j=T^{-1}(e^{i\theta_j})$, define $\mathcal{R}_{n,\alpha}\mathbf{c}$ to be $\mathcal{R}_{n,\alpha}g(k)$, where $g$ is any function on $\mathbb R$ satisfying $g(k_j)=c_j$ for $j=0,1,...,n-1$.

\subsection{Differentiation and Multiplication}
\label{sec:Appendix_DiffMult}

We want to derive an expression for $R_{j,\alpha}'(k)$. Differentiating $R_{j,\alpha}(k)$ results in
\begin{equation*}
    \frac{\mathrm{d}}{\mathrm{d}k}R_{j,\alpha}(k)=jR_{j-1}(k)\left(\frac{2i}{(k+i)^2}\right)+je^{ik\alpha}\left(\frac{2i}{(k+i)^2}\right)+i\alpha R_{j,\alpha}(k).
\end{equation*}
We now want to rewrite $2i/(k+i)^2$ in terms of oscillatory rational basis functions. It turns out that 
\begin{equation*}
    \frac{2i}{(k+i)^2}=\left[iR_{1,\alpha}(k)-\frac{i}{2}R_{2,\alpha}(k)\right]e^{-ik\alpha}.
\end{equation*}
Plugging that expression into our derivative and using the following fact,
\begin{equation*}
    R_{j,\alpha}(k)R_{\ell,\beta}(k)=R_{j+\ell,\alpha+\beta}(k)-R_{j,\alpha+\beta}(k)-R_{\ell,\alpha+\beta}(k)
\end{equation*}
results in the following tridiagonal operator for differentiation
\begin{equation*}
    R_{j,\alpha}'(k)=i\left[-\frac{j}{2}R_{j+1,\alpha}(k)+(j+\alpha)R_{j,\alpha}(k)-\frac{j}{2}R_{j-1,\alpha}(k)\right].
\end{equation*}

Suppose $g=\sum_jc_jR_{j,\alpha}(k)$ and consider the operator $\mathcal{M}_gf=gf$. Taking $f=R_{\ell,\beta}$,
\begin{align*}
    g(k)R_{\ell,\beta}(k)&=\sum_jc_jR_{j+\ell,\alpha+\beta}(k)-\sum_j c_jR_{j,\alpha+\beta}(k)-\sum_jc_jR_{\ell,\alpha+\beta}(k)\\
    &=\sum_jc_{j-\ell}R_{j,\alpha+\beta}(k)-\sum_jc_jR_{j,\alpha+\beta}(k)-\left[\sum_jc_j\right]R_{\ell,\alpha+\beta}(k).
\end{align*}
This implies that $\mathcal{M}_g$ has a bi-infinite matrix representation as
\begin{align*}
    \mathcal{M}(\mathbf{c})&:= \mathcal{T}(\mathbf{c})-\mathbf{c}\begin{bmatrix}
        \hdots & 1 & 1 & \hdots
    \end{bmatrix}-\left[\sum_jc_j\right]\mathcal{I},\\
    \mathbf{c}&=\begin{bmatrix}
        \hdots&c_{-1}&c_0&c_1&\hdots
    \end{bmatrix}^T.
\end{align*}
Here, $\mathcal{T}(\mathbf{c})$ is the Toeplitz operator with entry $(i,j)$ given by $c_{i-j}$.

\subsection{Cauchy Integrals}

Recall the notation for the Cauchy integral
\begin{align*}
    \mathcal{C}f(k)=\frac{1}{2\pi i}\int_{-\infty}^\infty\frac{f(k')}{k-k'}\mathrm{d}k,\quad k\in\mathbb{C}\setminus\mathbb{R},\\
    \mathcal{C}^\pm f(k)=\lim_{\epsilon\rightarrow 0^+}\mathcal{C}f(k\pm i\epsilon),\quad k\in\mathbb{R}.
\end{align*}
Our goal is to express $\mathcal{C}R_{j,\alpha}(k)$ in terms of $R_{j,\alpha}(k)$ and $R_{j,0}(k)$. We perform all of our computations in coefficient space. Given a function $f$ expressed as a finite linear combination of $R_{j,\alpha}$'s, we seek to compute the coefficients of $\mathcal{C}^\pm f$ in the same basis. Define for $j>0$, $n>0$,
\begin{equation*}
    \gamma_{j,n}(\alpha)=-e^{-|\alpha|}L_{|j|-n}^{(\beta)}(2|\alpha|\sigma),\quad \sigma=\mathrm{sign}(j),
\end{equation*}
where $L_n^{(\beta)}(x)$ is the generalized Laguerre polynomial of order $n$. Further define
\begin{align*}
    r_{j,\alpha}\left(\frac{-2i\sigma}{k+\sigma i}\right):=\mathrm{Res}_{k'=-\sigma i}\left\{R_{j,\alpha}(k')\frac{1}{k'-k}\right\}=\sum_{n-1}^{|j|}\gamma_{j,n}\left(\frac{-2i\sigma}{k+\sigma i}\right)^n.
\end{align*}
We then use the well-known relation for Laguere polynomials 
\begin{equation*}
    L_n^{(\beta)}(x)=L_n^{(\beta+1)}(x)-L_{n-1}^{(\beta+1)}(x),
\end{equation*}
to obtain the recurrence relation for $j\geq 1$
\begin{equation*}
    r_{j,\alpha}(z)=(1+z)r_{j-1,\alpha}(z)+z(L_{j-1}^{(1)}(2|k|)-L_{j-2}^{(1)}(2|k|)),
\end{equation*}
where $L_{-1}^{(1)}(x):=0$ and $r_{0,\alpha}(z):=0$. Note that $L_j^{(1)}(x)$ can be computed using its three-term recurrence relation~\cite{olver_nist_2010}. Hence, given a vector $\mathbf{k}$ of size $m$, with $\mathbf{z}=\frac{-2i\sigma}{\mathbf{k}+\sigma i}$, the matrix 
\begin{equation*}
    \mathbf{M}_\sigma(\mathbf{k})=\begin{bmatrix}
        \mathbf{r}_{1,\alpha}(\mathbf{z}) & \mathbf{r}_{2,\alpha}(\mathbf{z}) & \hdots & \mathbf{r}_{m,\alpha}(\mathbf{z})
    \end{bmatrix}
\end{equation*}
can be constructed in $\mathcal{O}(m^2)$ operations by building each column from the previous one. This means that to compute $\mathbf{M}_\sigma(\mathbf{k})\mathbf{c}$, it is not necessary to construct the matrix $\mathbf{M}_\sigma(\mathbf{k})$.

So for $\mathbf{f}_+(\mathbf{k})=\sum_{j=1}^m c_jR_{j,\alpha}(\mathbf{k})$ with $\mathbf{c}=\begin{bmatrix}
    c_1& \hdots &c_m
\end{bmatrix}^T$ we have 
\begin{align*}
    \mathcal{C}^+\mathbf{f}_+(\mathbf{k})&=\begin{cases}
        \mathbf{f}_+(\mathbf{k}),\quad &\alpha\geq0,\\
        -\mathbf{M}_{+1}(\mathbf{k})\mathbf{c},\quad &\alpha<0,
    \end{cases}\\
    \mathcal{C}^-\mathbf{f}_+(\mathbf{k})&=\begin{cases}
        \mathbf{0},\quad &\alpha\geq0,\\
        -\mathbf{f}_+(\mathbf{k})-\mathbf{M}_{+1}(\mathbf{k})\mathbf{c},\quad&\alpha<0.
    \end{cases}
\end{align*}
Similarly, for $\mathbf{f}_-(\mathbf{k})=\sum_{j=1}^mc_{-j}R_{-j,\alpha}(\mathbf{k})$ with $\mathbf{c}=\begin{bmatrix}
    c_{-1}&\hdots&c_{-m}
\end{bmatrix}^T$ we have
\begin{align*}
    \mathcal{C}^+\mathbf{f}_-(\mathbf{k})&=\begin{cases}
        \mathbf{f}_-(\mathbf{k})+\mathbf{M}_{-1}(\mathbf{k})\mathbf{c},\quad&\alpha\geq0,\\
        \mathbf{0},\quad&\alpha<0,
    \end{cases}\\
    \mathcal{C}^-\mathbf{f}_-(\mathbf{k})&=\begin{cases}
        \mathbf{M}_{-1}(\mathbf{k})\mathbf{c},\quad&\alpha\geq0,\\
        -\mathbf{f}_-(\mathbf{k}),\quad&\alpha<0.
    \end{cases}
\end{align*}
To compute the Cauchy operator, one separates the resulting function into an oscillatory function and a non-oscillatory function. The coefficients in the expansion of the oscillatory function are always found by multiplying the original coefficients by 1, -1, or 0. For the non-oscillatory function, one evaluates it point-wise on the real axis at the appropriate points for the interpolation operator $\mathcal{R}_{n,0}$ using the matrix $\mathbf{M}_\sigma$. Then the interpolation operator can be applied to compute the coefficients.

For example, if one wants to compute $\mathcal{C}^+\mathbf{f}({\mathbf{k}})$ where $\mathbf{f}(\mathbf{k})=\sum_{j=-m}^mc_jR_{j,\alpha}(\mathbf{k})$ when $\alpha>0$, we have
\begin{align*}
    \mathcal{C}^+\mathbf{f}(\mathbf{k})&=\sum_{j=-m}^mc_jR_{j,\alpha}(\mathbf{k})+\mathcal{R}_{2m+1,0}\mathbf{v}(\mathbf{k}),\\
    \mathbf{v}&=\mathbf{M}_\sigma(\mathbf{k})\mathbf{c},\quad\sigma=-1,\\
    \mathbf{k}&=\begin{bmatrix}
        T^{-1}(e^{i\theta_0})&\hdots&T^{-1}(e^{i\theta_{2m}})
    \end{bmatrix},\\
    \mathbf{c}&=\begin{bmatrix}
        c_{-1}&\hdots&c_{-m}
    \end{bmatrix}^T.
\end{align*}
Or, when $\alpha<0$,
\begin{align*}
    \mathcal{C}^+\mathbf{f}(\mathbf{k})&=-\mathcal{R}_{2m+1,0}\mathbf{v}(\mathbf{k}),\\
    \mathbf{v}&=M_\sigma(\mathbf{k})\mathbf{c},\quad\sigma=1,\\
    \mathbf{k}&=\begin{bmatrix}
        T^{-1}(e^{i\theta_0})&\hdots&T^{-1}(e^{i\theta_{2m}})
    \end{bmatrix},\\
    \mathbf{c}&=\begin{bmatrix}
        c_1&\hdots&c_m
    \end{bmatrix}^T.
\end{align*}
The formulae for $\mathcal{C}^-\mathbf{f}(\mathbf{k})$ can then be deduced from $\mathcal{C}^+-\mathcal{C}^-=\mathbb{I}$. This gives a reasonably fast method to compute the coefficients of the expansion of $\mathcal{C}^\pm[\mathbf{f}_++\mathbf{f}_-]$ in the basis $R_{j,\alpha}$ in $\mathcal{O}(m^2)$ operations. This efficiency is a key motivation for using these basis functions with infinite-dimensional \gls{gmres}, as the action of the Cauchy operators is closed on the basis.

\subsection{An Integration Formula}\label{subsec:Integration}

Define the inner product
    \begin{equation*}
        I_{j,\ell,\alpha_1,\alpha_2} = \int_\mathbb{R}R_{j,\alpha_1}(k)\overline{R_{\ell,\alpha_2}(k)}\text{d}z.
        \end{equation*}
    Since $\overline{R_{\ell,\alpha_2}(k)}=R_{-\ell,-\alpha_2}(k)$,
    \begin{equation*}
        I_{j,\ell,\alpha_1,\alpha_2}=\dashint_{\mathbb{R}}\left(R_{j-\ell,\alpha_1,\alpha_2}(k)-R_{j,\alpha_1-\alpha_2}(k)-R_{-\ell,\alpha_1-\alpha_2}(k)\right)\text{d}k.
    \end{equation*}
Therefore, the problem reduces to computing $\dashint_\mathbb{R}R_{j,\alpha}(k)\text{d}k$.
A formula for the integral of the oscillatory rational basis functions, $R_{j,\alpha}(k)$, was derived in~\cite{trogdon_application_2015}:
\begin{equation}\label{eq:Integration}
    \dashint_{\mathbb{R}}R_{j,\alpha}(k)\mathrm{d}k=\begin{cases}
        0,\quad &\mathrm{sign}(j)=\mathrm{sign}(\alpha),\\
        -2\pi|j|,\quad &\alpha=0,\\
        -4\pi e^{-|\alpha|}L_{|j|-1}^{(1)}(2|\alpha|),\quad &\mathrm{otherwise}.
    \end{cases}
\end{equation}

\section{Deferred proofs}\label{sec:Proof}

We include the following for completeness.

\subsection{Proof of Lemma~\ref{lemma:Jost}}

\begin{proof}[Proof of Lemma~\ref{lemma:Jost}]
Define the linear operator
\begin{equation*}
    \mathcal{L}\mathbf{r}(x;\lambda) := \int_{-\infty}^x\mathbf{K}(x,s;\lambda)\mathbf{P}(s)\mathbf{r}(s;\lambda)\mathrm{d}s.
\end{equation*}
Then the Volterra equation can be written as
\begin{equation*}
    \mathbf r - \mathcal L \mathbf r = \mathbf f.
\end{equation*}
It is immediate that
\begin{equation*}
    \mathcal{L}^n\mathbf{f}(x)=\int_{-\infty<s_n<...<s_1<x}\left[\mathbf{K}(x,s_1;\lambda)\mathbf{P}(s_1)\mathbf{K}(s_1,s_2;\lambda)\mathbf{P}(s_2)...\mathbf{K}(s_{n-1},s_n;\lambda)\mathbf{P}(s_n)\right]
    \mathbf{f}(s_n)\mathrm{d}s_1...\mathrm{d}s_n.
\end{equation*}
Using $\|\mathbf{K}(x,s;\lambda)\|\leq C_0$ gives the classical estimate
\begin{align*}
    \|\mathcal{L}^n\mathbf{f}(x)\|&\leq C_0^n\|\vec f\|_{L^\infty(
    \mathbb R,\mathbb C^2)} \int_{-\infty<s_n<...<s_1<x}\prod_{j=1}^{n}\|\mathbf{P}(s_j)\|\mathrm{d}s_1...\mathrm{d}s_n\\
    &=\frac{C_0^n\|\vec f\|_{L^\infty(
    \mathbb R,\mathbb C^2)}}{n!}\left(\int_{-\infty}^x\|\mathbf{P}(s)\|\mathrm{d}s\right)^n\leq \|\vec f\|_{L^\infty(
    \mathbb R,\mathbb C^2)}\frac{C_0^nM^n}{n!}, \quad M = \|\vec P\|_{L^1(
    \mathbb{R},\mathbb C^{2\times 2})}.
\end{align*}
The first equality follows from the fact that $\prod_{j=1}^n\|\mathbf{P}(s_j)\|$ is invariant under permutations of $(s_1,...,s_n)$, so integrating over the full cube $(-\infty,x)^n$ is $n!$ times the integral over the simplex ${-\infty<s_n<...<s_1<x}$.

This establishes the convergence of the Neumann series 
\begin{align*}
    \vec r(x;\lambda) = \sum_{n=0}^\infty \mathcal L^n \vec f(x),
\end{align*}
as a sequence of operators on $L^\infty(
\mathbb R,\mathbb C^2)$.  Thus
\begin{equation*}
    \|\mathbf{r}(x;\lambda)\|\leq e^{C_0M}.
\end{equation*}
It follows from the dominated convergence theorem that $\lambda \mapsto \mathcal L^n \vec f(x;\lambda)$ is continuous in $\lambda$ for $\lambda \in \overline{\mathbb C^+}$.  Similarly, it follows that $\lambda \mapsto \mathcal L^n \vec f(x;\lambda)$ is complex differentiable for $\lambda \in \mathbb C^+$, and is therefore analytic.  By the uniform convergence of the Neumann series, $\vec r(x;\lambda)$ is analytic for $\lambda \in \mathbb C^+$.  

\end{proof}

\subsection{Proof of Lemma~\ref{lemma:JostDeriv}}

\begin{proof}[Proof of Lemma~\ref{lemma:JostDeriv}]
    We will prove the result for $\partial_x\mathbf{r}_{1,-}(x;\lambda)$. The arguments for $\partial_x\mathbf{r}_{2,+}(x;\lambda)$, $\partial_x\mathbf{r}_{2,-}(x;\lambda)$, and $\partial_x\mathbf{r}_{1,+}(x;\lambda)$ are analogous. We begin by considering $\mathbf{r}_{1,-}$, 
    \begin{align*}
        \mathbf{r}_{1,-}(x;\lambda)&=\begin{bmatrix}
            1\\0
        \end{bmatrix}+\int_{-\infty}^x\begin{bmatrix}
            1 & 0\\
            0 & e^{2i\lambda(x-s)}
        \end{bmatrix}\begin{bmatrix}
            0 & q(s)\\
            \tau\overline{q}(s) & 0
        \end{bmatrix}\mathbf{r}_{1,-}(s;\lambda)\mathrm{d}s\\
        &=\begin{bmatrix}
            1\\0
        \end{bmatrix}+\int_{-\infty}^x\begin{bmatrix}
            q(s)r_{1,-}^{(2)}(s;\lambda)\\
            \tau\overline{q}(s)r_{1,-}^{(1)}(s;\lambda)e^{2i\lambda(x-s)}
        \end{bmatrix}\mathrm{d}s,
    \end{align*}
    where $r_{1,-}^{(n)}$, $n = 1,2$ denotes two entries of $\mathbf{r}_{1,-}$. Letting $f(s;\lambda):=\tau\overline{q}(s)r_{1,-}^{(1)}(s;\lambda)$ and integrating by parts gives
    \begin{align*}
        r_{1,-}^{(2)}(x;\lambda)=-\frac{f(x;\lambda)}{2i\lambda}+\frac{1}{2i\lambda}\int_{-\infty}^x\partial_s[f(s;\lambda)]e^{2i\lambda(x-s)}\mathrm{d}s.
    \end{align*}
    Observe that $\partial_sf(s;\lambda)=\tau\partial_s\overline{q}(s)r_{1,-}^{(1)}(s;\lambda)+\tau|q(s)|^2r_{1,-}^{(2)}(s;\lambda)$ since $\partial_x r_{1,-}^{(1)}(x;\lambda)=q(x)r_{1,-}^{(2)}(x;\lambda)$ (see~\eqref{eq:Component1}).
    By Corollary~\ref{cor:Jost}, $\left\|\mathbf{r}_{1,-}(s;\lambda)\right\|\leq C$, uniformly in $s,\lambda$, so $|\partial_sf(s;\lambda)|\leq C(|\partial_s\overline{q}(s)|+|q(s)|^2)$ and therefore $\|\partial_s f(s;\lambda)\|_{L^1(\mathbb{R},\mathbb{C})} \leq C'$, for a new constant $C'$. We have
    \begin{equation*}
        \left|r_{1,-}^{(2)}(x;\lambda)\right|\leq \frac{C''}{|\lambda|}\left(|f(x;\lambda)|+\int_{-\infty}^x|\partial_sf(s;\lambda)|\mathrm{d}s\right).
    \end{equation*}
    Therefore
    \begin{equation*}
        r_{1,-}^{(2)}(x;\lambda)=\mathcal{O}\left(\frac{1}{|\lambda|}\right),\quad|\lambda|\rightarrow\infty,\quad \lambda \in \mathbb C^+.
    \end{equation*}
    Then from the first component equation,
    \begin{equation}\label{eq:Component1}
        r_{1,-}^{(1)}(x;\lambda)=1+\int_{-\infty}^x q(s)r_{1,-}^{(2)}(s;\lambda)\mathrm{d}s,
    \end{equation}
    and again using uniform boundedness, $r_{1,-}^{(1)}(x;\lambda)= 1 + \mathcal{O}(1/|\lambda|)$ as $|\lambda|\rightarrow\infty$, $\lambda \in \mathbb C^+$. 
    Therefore, 
    \begin{equation*}
        \mathbf{r}_{1,-}(x;\lambda)=\begin{bmatrix}
            1\\0
        \end{bmatrix} + \mathcal{O}\left(\frac{1}{|\lambda|}\right),\quad|\lambda|\rightarrow\infty,\quad \lambda \in \mathbb C^+.
    \end{equation*}
    
    Using $\partial_x r_{1,-}^{(1)}(x;\lambda)=q(x)r_{1,-}^{(2)}(x;\lambda)$, we see immediately that $\partial_xr_{1,-}^{(1)}(x;\lambda)=\mathcal{O}\left(\frac{1}{|\lambda|}\right)$, $|\lambda|\rightarrow\infty$, $\lambda\in\mathbb{C}^+$. Substituting the expression for $r_{1,-}^{(1)}(x;\lambda)$ into the expression for $r_{1,-}^{(2)}(x;\lambda)$ results in
    \begin{equation*}
        r_{1,-}^{(2)}(x;\lambda)=\int_{-\infty}^x \tau\bar{q}(s)e^{2i\lambda(x-s)}\left[1+\int_{-\infty}^s q(t)r_{1,-}^{(2)}(t;\lambda)\mathrm{d}t\right]\mathrm{d}s.
    \end{equation*}
    Hence, 
    \begin{align*}
        \partial_x r_{1,-}^{(2)}(x;\lambda)&=\tau\bar{q}(x)\left[1+\int_{-\infty}^x q(t)r_{1,-}^{(2)}(t;\lambda)\mathrm{d}t\right] + 2i\lambda r_{1,-}^{(2)}(x;\lambda)\\
        &=\tau\bar{q}(x)r_{1,-}^{(1)}(x;\lambda) + 2i\lambda r_{1,-}^{(2)}(x;\lambda).
    \end{align*}
    Using integration by parts on the formula for $r_{1,-}^{(2)}(x;\lambda)$, we obtain
    \begin{align*}
        \partial_x r_{1,-}^{(2)}(x;\lambda)
        &=
        \tau\overline{q}(x)r_{1,-}^{(1)}(x;\lambda)
        +
        2i\lambda
        \left[
            -\frac{\tau\overline{q}(x)r_{1,-}^{(1)}(x;\lambda)}{2i\lambda}
            +
            \frac{1}{2i\lambda}
            \int_{-\infty}^x
            \partial_s[f(s;\lambda)]e^{2i\lambda(x-s)}
            \mathrm{d}s
        \right]\\
        &=
        \int_{-\infty}^x
        \partial_s[f(s;\lambda)]e^{2i\lambda(x-s)}
        \mathrm{d}s.
    \end{align*}
    Since $\partial_s f(s;\lambda)\in L^1(\mathbb{R},\mathbb{C})$ with norm bounded independently of $\lambda$, it follows that
    \begin{align*}
        \left|
            \partial_x r_{1,-}^{(2)}(x;\lambda)
        \right|
        \leq
        \int_{-\infty}^x
        \left|
            \partial_sf(s;\lambda)
        \right|
        \mathrm{d}s
        \leq
        \left\|
            \partial_s f(s;\lambda)
        \right\|_{L^1(\mathbb{R},\mathbb{C})}
        \leq C_1
    \end{align*}
    Therefore,
    \begin{equation*}
        \left\|
            \partial_x r_{1,-}^{(2)}(x;\lambda)
        \right\|
        \leq C_1.
    \end{equation*}
\end{proof}

\bibliographystyle{plain}
\bibliography{references}

@book{reed_functional_1980,
	address = {San Diego, Calif.},
	edition = {Revised},
	title = {Functional {Analysis}},
	volume = {1},
	isbn = {978-0-12-585050-6},
	language = {English},
	publisher = {Academic Press},
	author = {Reed, M and Simon, B},
	month = jan,
	year = {1980},
}

@book{muskhelishvili_singular_2013,
	title = {Singular {Integral} {Equations}: {Boundary} {Problems} of {Function} {Theory} and {Their} {Application} to {Mathematical} {Physics}},
	isbn = {978-0-486-14506-8},
	shorttitle = {Singular {Integral} {Equations}},
	language = {en},
	publisher = {Courier Corporation},
	author = {Muskhelishvili, N I},
	month = feb,
	year = {2013},
}

@article{levin_analysis_1997,
	title = {Analysis of a collocation method for integrating rapidly oscillatory functions},
	volume = {78},
	issn = {0377-0427},
	url = {https://www.sciencedirect.com/science/article/pii/S0377042796001379},
	doi = {10.1016/S0377-0427(96)00137-9},
	number = {1},
	urldate = {2025-02-20},
	journal = {Journal of Computational and Applied Mathematics},
	author = {Levin, D},
	month = feb,
	year = {1997},
	pages = {131--138},
}

@misc{trogdon_ultraspherical_2024,
	title = {The ultraspherical rectangular collocation method and its convergence},
	url = {http://arxiv.org/abs/2401.03608},
	doi = {10.48550/arXiv.2401.03608},
	urldate = {2024-09-19},
	publisher = {arXiv},
	author = {Trogdon, T},
	month = jan,
	year = {2024},
	note = {arXiv:2401.03608 [cs, math]},
}

@book{olver_nist_2010,
	address = {New York, NY},
	title = {{NIST} handbook of mathematical functions},
	isbn = {978-0-521-19225-5 978-0-521-14063-8},
	language = {en},
	publisher = {Cambridge University Press},
	author = {Olver, F W J and Lozier, D W and Boisvert, R F and Clark, C W},
	year = {2010},
}

@article{trogdon_scattering_2021,
	title = {Scattering and inverse scattering for the {AKNS} system: {A} rational function approach},
	volume = {147},
	copyright = {© 2021 Wiley Periodicals LLC},
	issn = {1467-9590},
	shorttitle = {Scattering and inverse scattering for the {AKNS} system},
	url = {https://onlinelibrary.wiley.com/doi/abs/10.1111/sapm.12434},
	doi = {10.1111/sapm.12434},
	language = {en},
	number = {4},
	urldate = {2024-09-19},
	journal = {Studies in Applied Mathematics},
	author = {Trogdon, T},
	year = {2021},
	pages = {1443--1480},
}

@article{trogdon_rational_2014,
	title = {Rational {Approximation}, {Oscillatory} {Cauchy} {Integrals}, and {Fourier} {Transforms}},
	volume = {43},
	doi = {10.1007/s00365-015-9294-2},
	journal = {Constructive Approximation},
	author = {Trogdon, T},
	month = mar,
	year = {2014},
}

@article{trogdon_application_2015,
	title = {On the application of {GMRES} to oscillatory singular integral equations},
	volume = {55},
	issn = {0006-3835},
	url = {https://doi.org/10.1007/s10543-014-0502-4},
	doi = {10.1007/s10543-014-0502-4},
	number = {2},
	urldate = {2025-06-14},
	journal = {BIT},
	author = {Trogdon, T},
	month = jun,
	year = {2015},
	pages = {591--620},
}

@book{hislop_introduction_1996,
	address = {New York, NY},
	series = {Applied {Mathematical} {Sciences}},
	title = {Introduction to {Spectral} {Theory}: {With} {Applications} to {Schrödinger} {Operators}},
	copyright = {http://www.springer.com/tdm},
	isbn = {978-1-4612-6888-8 978-1-4612-0741-2},
	shorttitle = {Introduction to {Spectral} {Theory}},
	url = {http://link.springer.com/10.1007/978-1-4612-0741-2},
	language = {en},
	urldate = {2025-07-09},
	publisher = {Springer},
	author = {Hislop, P D and Sigal, I M},
	year = {1996},
	doi = {10.1007/978-1-4612-0741-2},
	note = {ISSN: 0066-5452},
}

@book{ablowitz_solitons_1991,
	address = {Cambridge},
	series = {London {Mathematical} {Society} {Lecture} {Note} {Series}},
	title = {Solitons, {Nonlinear} {Evolution} {Equations} and {Inverse} {Scattering}},
	isbn = {978-0-521-38730-9},
	url = {https://www.cambridge.org/core/books/solitons-nonlinear-evolution-equations-and-inverse-scattering/5DBED94706868291C74AB13ACC750A79},
	urldate = {2025-01-31},
	publisher = {Cambridge University Press},
	author = {Ablowitz, M J and Clarkson, P A},
	year = {1991},
	doi = {10.1017/CBO9780511623998},
}

@book{chicone_ordinary_2006,
        address = {New York},
	series = {Texts in {Applied} {Mathematics}},
	title = {Ordinary {Differential} {Equations} with {Applications}},
	copyright = {http://www.springer.com/tdm},
	isbn = {978-0-387-30769-5},
	url = {http://link.springer.com/10.1007/0-387-35794-7},
	language = {en},
	urldate = {2025-07-09},
	publisher = {Springer},
	author = {Chicone, C},
	year = {2006},
	doi = {10.1007/0-387-35794-7},
}

@book{trogdon_riemannhilbert_2016,
	address = {Philadelphia},
	title = {Riemann–{Hilbert} {Problems}, their {Numerical} {Solution}, and the {Computation} of {Nonlinear} {Special} {Functions}},
	isbn = {978-1-61197-419-5},
	language = {English},
	publisher = {Society for Industrial and Applied Mathematics},
	author = {Trogdon, T and Olver, S},
	month = oct,
	year = {2016},
}

@Misc{OperatorApproximation,
author =   {Trogdon, T and Lilly, K and Ballew, C and Vaes, W},
title =    {Operator Approximation: a {Julia} package for approximating functions and operators and solving operator equations.},
howpublished = {\url{https://github.com/tomtrogdon/OperatorApproximation.jl/}},
year = {2025}
                }

@book{stakgold_boundary_1967,
	edition = {Second Printing},
	title = {Boundary {Value} {Problems} of {Mathematical} {Physics}},
	volume = {1},
	language = {English},
	publisher = {Macmillan Company},
	author = {Stakgold, I},
	month = jan,
	year = {1967},
}

@book{faddeev_hamiltonian_1987,
	address = {Berlin, Heidelberg},
	title = {Hamiltonian {Methods} in the {Theory} of {Solitons}},
	copyright = {http://www.springer.com/tdm},
	isbn = {978-3-540-69843-2 978-3-540-69969-9},
	url = {http://link.springer.com/10.1007/978-3-540-69969-9},
	urldate = {2025-01-30},
	publisher = {Springer},
	author = {Faddeev, L D and Takhtajan, L A},
	year = {1987},
	doi = {10.1007/978-3-540-69969-9},
}

@article{fokas_boundary-value_2004,
	title = {Boundary-{Value} {Problems} for {Linear} {PDEs} with {Variable} {Coefficients}},
	volume = {460},
	issn = {1364-5021},
	url = {https://www.jstor.org/stable/4143236},
	number = {2044},
	urldate = {2024-09-19},
	journal = {Proceedings of the Royal Society of London. Series A: Mathematical, Physical and Engineering Sciences},
	author = {Fokas, A S},
	year = {2004},
	pages = {1131--1151},
}

@article{deconinck_variable-coefficient_2025,
	title = {Variable-{Coefficient} {Evolution} {Problems} via the {Fokas} {Method} {Part} {I}: {Dissipative} {Case}},
	volume = {154},
	copyright = {© 2024 Wiley Periodicals LLC.},
	issn = {1467-9590},
	shorttitle = {Variable-{Coefficient} {Evolution} {Problems} via the {Fokas} {Method} {Part} {I}},
	url = {https://onlinelibrary.wiley.com/doi/abs/10.1111/sapm.12800},
	doi = {10.1111/sapm.12800},
	language = {en},
	number = {1},
	urldate = {2025-01-03},
	journal = {Studies in Applied Mathematics},
	author = {Deconinck, B and Farkas, M},
	year = {2025},
	pages = {e12800},
}

@article{ablowitz_inverse_1974,
	title = {The {Inverse} {Scattering} {Transform}-{Fourier} {Analysis} for {Nonlinear} {Problems}},
	volume = {53},
	issn = {1467-9590},
	url = {https://onlinelibrary.wiley.com/doi/abs/10.1002/sapm1974534249},
	doi = {10.1002/sapm1974534249},
	language = {en},
	number = {4},
	urldate = {2025-08-20},
	journal = {Studies in Applied Mathematics},
	author = {Ablowitz, M J and Kaup, D J and Newell, A C and Segur, H},
	year = {1974},
	note = {},
	pages = {249--315},
}

@article{deift_inverse_1979,
	title = {Inverse scattering on the line},
	volume = {32},
	copyright = {Copyright © 1979 Wiley Periodicals, Inc., A Wiley Company},
	issn = {1097-0312},
	url = {https://onlinelibrary.wiley.com/doi/abs/10.1002/cpa.3160320202},
	doi = {10.1002/cpa.3160320202},
	language = {en},
	number = {2},
	urldate = {2025-08-20},
	journal = {Communications on Pure and Applied Mathematics},
	author = {Deift, P and Trubowitz, E},
	year = {1979},
	note = {},
	pages = {121--251},
}

@book{ablowitz_complex_2003,
	address = {Cambridge},
	edition = {2},
	series = {Cambridge {Texts} in {Applied} {Mathematics}},
	title = {Complex {Variables}: {Introduction} and {Applications}},
	isbn = {978-0-521-53429-1},
	shorttitle = {Complex {Variables}},
	url = {https://www.cambridge.org/core/books/complex-variables/08A62E6DB03F5D5435F5DE6260618002},
	urldate = {2025-08-20},
	publisher = {Cambridge University Press},
	author = {Ablowitz, M J and Fokas, A S},
	year = {2003},
	doi = {10.1017/CBO9780511791246},
}

@book{beals_direct_1988,
	address = {Providence, Rhode Island},
	series = {Mathematical {Surveys} and {Monographs}},
	title = {Direct and {Inverse} {Scattering} on the {Line}},
	volume = {28},
	isbn = {978-0-8218-1530-4 978-1-4704-1255-5},
	url = {http://www.ams.org/surv/028},
	language = {en},
	urldate = {2025-08-20},
	publisher = {American Mathematical Society},
	author = {Beals, R and Deift, P and Tomei, C},
	year = {1988},
	doi = {10.1090/surv/028},
}

@book{drazin_solitons_1989,
	address = {Cambridge},
	edition = {2},
	series = {Cambridge {Texts} in {Applied} {Mathematics}},
	title = {Solitons: {An} {Introduction}},
	isbn = {978-0-521-33655-0},
	shorttitle = {Solitons},
	url = {https://www.cambridge.org/core/books/solitons/3992154606336D7459B839EB29BF9C38},
	urldate = {2024-09-19},
	publisher = {Cambridge University Press},
	author = {Drazin, P G and Johnson, R S},
	year = {1989},
	doi = {10.1017/CBO9781139172059},
}

@article{zakharov_exact_1970,
	title = {Exact {Theory} of {Two}-dimensional {Self}-focusing and {One}-dimensional {Self}-modulation of {Waves} in {Nonlinear} {Media}},
	url = {https://www.semanticscholar.org/paper/Exact-Theory-of-Two-dimensional-Self-focusing-and-Zakharov-Shabat/73ef172d46291e1b2235d4fea6d80078f5e705ef},
	urldate = {2025-08-20},
	journal = {Journal of Experimental and Theoretical Physics},
	author = {Zakharov, V and Shabat, A},
	year = {1970},
}

@article{wilkening_spectral_2015,
	title = {A {Spectral} {Transform} {Method} for {Singular} {Sturm}--{Liouville} {Problems} with {Applications} to {Energy} {Diffusion} in {Plasma} {Physics}},
	volume = {75},
	issn = {0036-1399},
	url = {https://epubs.siam.org/doi/10.1137/130941948},
	doi = {10.1137/130941948},
	number = {2},
	urldate = {2024-09-19},
	journal = {SIAM Journal on Applied Mathematics},
	author = {Wilkening, J and Cerfon, A},
	month = jan,
	year = {2015},
	pages = {350--392},
}

@Misc{Code,
author =   {Lilly, K},
title =    {{Generalized}{Transforms}: a {Julia} package for computing the generalized transform pair associated with the {Dirac} equation.},
howpublished = {\url{https://github.com/klilly50/GeneralizedTransforms}},
year = {2025}
                }

@book{lovitt_linear_1924,
	title = {Linear {Integral} {Equations}},
	isbn = {0-486-44285-3},
	url = {https://store.doverpublications.com/products/9780486174648},
	language = {en},
	urldate = {2026-01-01},
	publisher = {McGraw-Hill Book Co.},
	author = {Lovitt, W V},
	year = {1924},
}

@article{weber_numerical_1980,
	title = {Numerical computation of the fourier transform using {Laguerre} functions and the {Fast} {Fourier} {Transform}},
	volume = {36},
	issn = {0945-3245},
	url = {https://doi.org/10.1007/BF01396758},
	doi = {10.1007/BF01396758},
	language = {en},
	number = {2},
	urldate = {2026-01-16},
	journal = {Numerische Mathematik},
	author = {Weber, H},
	month = jun,
	year = {1980},
	pages = {197--209},
}

@inproceedings{deift_orthogonal_2000,
	address = {Providence, Rhode Island},
	series = {Courant {Lecture} {Notes}},
	title = {Orthogonal {Polynomials} and {Random} {Matrices}: {A} {Riemann}-{Hilbert} {Approach}},
	volume = {3},
	isbn = {978-0-8218-2695-9 978-1-4704-3107-5},
	booktitle = {Orthogonal {Polynomials} and {Random} {Matrices}},
	url = {http://www.ams.org/cln/003},
	language = {en},
	urldate = {2026-02-11},
	publisher = {American Mathematical Society},
	author = {Deift, P},
	month = oct,
	year = {2000},
	doi = {10.1090/cln/003},
}

@article{Olver2011ComputingTH,
  title={Computing the Hilbert transform and its inverse},
  author={S Olver},
  journal={Math. Comput.},
  year={2011},
  volume={80},
  pages={1745-1767},
  url={https://api.semanticscholar.org/CorpusID:33190302}
}

@article{11e5cc21-6b4a-3b97-8220-bac5264f5b1f,
 ISSN = {00255718, 10886842},
 URL = {http://www.jstor.org/stable/2153449},
 author = {J A C Weideman},
 journal = {Mathematics of Computation},
 number = {210},
 pages = {745--762},
 publisher = {American Mathematical Society},
 title = {Computing the Hilbert Transform on the Real Line},
 urldate = {2025-04-04},
 volume = {64},
 year = {1995}
}

@book{fokas_unified_2008,
	address = {Philadelphia, PA},
	series = {{CBMS}-{NSF} {Regional} {Conference} {Series} in {Applied} {Mathematics}},
	title = {A {Unified} {Approach} to {Boundary} {Value} {Problems}},
	isbn = {978-0-89871-651-1},
	language = {English},
	number = {78},
	publisher = {Society for Industrial and Applied Mathematics},
	author = {Fokas, A S},
	year = {2008},
}

@book{titchmarsh_e_c_introduction_1948,
	edition = {2nd},
	title = {Introduction {To} {The} {Theory} {Of} {Fourier} {Integrals}},
	isbn = {0-19-853320-9},
	url = {http://archive.org/details/dli.ernet.2568},
	language = {eng},
	urldate = {2026-04-17},
	publisher = {Oxford University Press},
	author = {E C Titchmarsh},
	year = {1948},
}

@article{GMRES,
author = {Saad, Y and Schultz, M H},
title = {GMRES: A Generalized Minimal Residual Algorithm for Solving Nonsymmetric Linear Systems},
journal = {SIAM Journal on Scientific and Statistical Computing},
volume = {7},
number = {3},
pages = {856-869},
year = {1986},
doi = {10.1137/0907058},
URL = {https://doi.org/10.1137/0907058},
eprint = {https://doi.org/10.1137/0907058},
}

\end{document}